\documentclass[11pt]{amsart}
\usepackage{amsmath}
\usepackage{amssymb}
\usepackage{amsthm}
\usepackage[margin=3cm, marginparwidth=2.4cm]{geometry}

\usepackage[pagebackref=false]{hyperref}

\usepackage{graphicx}
\usepackage{overpic}
\usepackage{xcolor}
\usepackage{mathtools}

\usepackage{tikz}

\usetikzlibrary{matrix}

\newtheorem{theorem}{Theorem}[section]
\newtheorem{lemma}[theorem]{Lemma}
\newtheorem{proposition}[theorem]{Proposition}
\newtheorem{corollary}[theorem]{Corollary}

\newtheorem{introtheorem}{Theorem}

\theoremstyle{definition}
\newtheorem{definition}[theorem]{Definition}
\newtheorem{example}[theorem]{Example}

\theoremstyle{remark}
\newtheorem{remark}{Remark}[theorem]

\usepackage{marginnote}
\newcounter{mynote}

\newcommand{\G}{\Gamma}

\usepackage[foot]{amsaddr}

\title[Subgroups of RACGs]{Subgroups of right-angled Coxeter groups via Stallings-like techniques}
\author{Pallavi Dani$^1$}
\author{Ivan Levcovitz$^2$}
\address{$^1$Department of Mathematics, Louisiana State University, Baton Rouge, LA 70803-4918, USA}
\address{$^2$Technion - Israel Institute of Technology, Department of Mathematics, Haifa, 32000, Israel}
\email{pdani@math.lsu.edu (Pallavi Dani), Levcovitz@technion.ac.il (Ivan Levcovitz)}

\thanks{This work of the first author was supported
by a grant from the Simons Foundation (\#426932, Pallavi Dani) and by NSF Grant No.~DMS-1812061. This work of the second author was supported by the Israel Science Foundation and in part by a Technion fellowship.}
\begin{document}

\begin{abstract}
	We associate cube complexes called completions to each subgroup of a right-angled Coxeter group (RACG).
	A completion characterizes many properties of the subgroup such as whether it  is quasiconvex, normal, finite-index or torsion-free. We use completions to show that reflection subgroups are quasiconvex, as are one-ended Coxeter subgroups of a 2-dimensional RACG. We provide an algorithm that determines whether a given one-ended, 2-dimensional RACG is isomorphic to some finite-index subgroup of another given RACG. In addition, we answer several algorithmic questions regarding quasiconvex subgroups. Finally, we give a new proof of Haglund's result that quasiconvex subgroups of RACGs are separable.  
\end{abstract}

\maketitle



\section{Introduction} \label{sec_intro}
In the highly influential article \cite{Stallings}, Stallings introduced new tools to study subgroups of free groups. A crucial
 idea in Stallings' work is that given a finite set of words in a free group, one can associate a labeled graph to this set, and perform a sequence of operations, now known as ``Stallings folds,'' to this graph. The resulting graph is, in some sense, a canonical object associated to the subgroup generated by the given words. This topological viewpoint provided clean new proofs for many theorems regarding subgroups of free groups.  
 In~\cite{KM}, Kapovich--Miasnikov use Stallings' ideas, recast in a more combinatorial form, to systematically study the subgroup structure of free groups and to answer a number of algorithmic questions about such free subgroups. 
 
	Arzhantseva--Olshanskii 
were the first to apply Stallings' techniques to groups that are not free in \cite{Arzhantseva-Olshanskii}, and Arzhantseva uses these techniques to show quasiconvexity of certain subgroups in \cite{Arzhantseva-comm}.
	More recently, 
  Kharlampovich--Miasnikov--Weil
 use
similar techniques  to study several algorithmic questions for automatic groups 
 (which include RACGs)~\cite{KMW}.
 Beeker--Lazarovich, building on work of Brown \cite{SamuelBrown}, define a version of Stallings folds for CAT(0) cube complexes, which they use to give a characterization of finitely presented
 quasiconvex subgroups of cubulated hyperbolic groups, in terms of hyperplane stabilizers~\cite{BL}.
Other related 
articles are discussed in Section~\ref{sec_relation}.
 
The main goal of this article is to study subgroups of right-angled Coxeter groups (RACGs)
through generalizations of
 Stallings' techniques.  
Restricting to RACGs (rather than considering all automatic groups or CAT(0) cube complex groups) allows for some stronger results.  
 We do not make any hyperbolicity 
 assumption, which is often crucial in more general frameworks.  
 Our approach is 
  inspired by the combinatorial methods used in~\cite{KM}, and is 
quite different from that of~\cite{BL}. 
This is further discussed in Section~\ref{sec_relation}.

Given a finite  simplicial graph $\G$, the associated RACG $W_\G$ is generated by order two elements corresponding to vertices of~$\G$, with the additional relations that two such generators commute if there is an edge in $\G$ between the corresponding vertices.
RACGs form a wide class of groups which have become central objects in geometric group theory.  We refer to~\cite{Dani} for a survey of recent work on  these groups.  
One interesting feature of RACGs is that they have a rich variety of subgroups, which includes all free groups, RAAGs (right-angled Artin groups)~\cite{DJ} 
 and surface groups.
Incredibly, all hyperbolic 
3-manifold groups \cite{Ago} \cite{Wise} and Coxeter groups~\cite{Haglund-Wise}
are virtually subgroups of RACGs as well.  

Given a subgroup $G$ of a RACG $W_{\G}$, we abstractly define a completion of $G$ as an edge-labeled cube complex satisfying certain properties.
If $G$ is additionally finitely generated, we explicitly build a  \textit{standard} 
completion for $G$ by the following procedure.
 Starting with
a subdivided ``rose'' graph whose petals are labeled by the words generating $G$,
we perform a sequence of operations of three possible types: fold, cube attachment, and cube identification. A completion $\Omega$ can always be obtained as the direct limit of the complexes in this sequence. 
 Many properties of the subgroup $G$ can be characterized in terms of properties of 
 $\Omega$, 
 as summarized in the following theorem:
 
\begin{introtheorem} \label{intro_thm_main}
	Let $G$ be a subgroup of the RACG $W_{\Gamma}$. Then
	\begin{enumerate}
		\item $G$ is quasiconvex in $W_\G$ if and only if $G$ is finitely generated and every (equivalently, some) standard completion for $G$ is finite (Theorem~\ref{thm_qc}).
		\item  There exist characterizations of $G$ having finite index (Theorem~\ref{thm_omega_finite_index}), $G$ being torsion-free (Proposition~\ref{prop_torsion_free}), and $G$ being normal (Theorem~\ref{thm_normal}) in terms of properties of a completion of $G$. 
	
		\item 
		If $G$ is torsion-free, then any completion is non-positively curved  
		(Proposition~\ref{prop_torsion_free_npc}) and has fundamental group isomorphic to $G$ (Theorem~\ref{thm_fundamental_group}). 
	\end{enumerate}
\end{introtheorem}

We note that quasiconvex subgroups are always quasi-isometrically embedded and, in RACGs, are 
 separable~\cite{Haglund}.
We discuss the relation of our results 
to~\cite{BL, KMW} and other works  in Section~\ref{sec_relation}.

Our approach to producing completions 
is constructive.  In particular, given a quasiconvex subgroup, a finite completion can be explicitly constructed in finite time (by Theorem~\ref{intro_thm_main}(1) and Proposition~\ref{prop_finite_completion}).
  Furthermore, as our 
completions are cube complexes (and not just graphs
as is the case in many generalizations of Stallings' methods to non-free settings),
 powerful tools from cubical geometry can be applied to study them.
This enables us to use the characterizations in Theorem~\ref{intro_thm_main} to prove 
  several new results on subgroups of RACGs 
 including showing that particular classes of subgroups of RACGs are always quasiconvex and solving various algorithmic problems
 (see Theorem~\ref{intro_thm_reflection_subgroups}--Theorem~\ref{introtheorem_algorithms} below).

We expect that Theorem~\ref{intro_thm_main} will have other applications.  For instance,  
the characterization of finite-index subgroups in Theorem~\ref{intro_thm_main} is an essential tool in our proof of~\cite[Theorem~B]{DL_visual_raags}, which states that any $2$-dimensional,
one-ended RACG with planar defining graph is quasi-isometric to a RAAG if
and only if it is commensurable to a RAAG.

A given subgroup may have multiple completions.  
Indeed two completions for a given subgroup need not even be homotopy equivalent (see Example~\ref{ex_non-unique}). 
Nevertheless, our results 
do not depend on the specific completions chosen.    
In fact, the flexibility of choosing different generating sets and order of operations 
is useful in our arguments. 
Despite the non-uniqueness of completions, every completion for $G$ has a 1-dimensional subcomplex called its core graph, and any two core graphs for $G$ are isomorphic (see Proposition~\ref{prop_unique_core}).

Theorem~\ref{intro_thm_main} provides a tool to show  that subgroups of a RACG are quasiconvex, by showing that their associated completions must be finite. 
We apply this approach to subgroups of $W_\G$ generated by \emph{reflections}, i.e., conjugates of the generators corresponding to vertices of $\G$.  We call such subgroups \emph{reflection subgroups}.  More precisely, we prove:

\begin{introtheorem}[Theorem~\ref{thm_completion_of_reflection_subgroup}] 
\label{intro_thm_reflection_subgroups}
	Every finitely generated reflection subgroup of a RACG is quasiconvex. 
\end{introtheorem}

We next turn our attention to Coxeter subgroups,
i.e.~subgroups that are  isomorphic to some abstract finitely generated Coxeter group. 
A result 
proved independently by Dyer~\cite{Dyer} and Deodhar~\cite{Deodhar}
shows that every reflection subgroup of a RACG is a Coxeter subgroup. The converse to this statement is not true in general (see Remark~\ref{rmk_reflection_counterexample}), but 
 we show that
it holds under certain hypotheses:
\begin{introtheorem}[Theorem~\ref{thm_racg_subgroup_gen_set}, Corollary~\ref{cor_RACG_subgroups_qc}] \label{intro_thm_racg_subgroups}
	Every one-ended Coxeter subgroup of a $2$-dimensional RACG is a reflection subgroup. Consequently, every such Coxeter subgroup is quasiconvex by Theorem~\ref{intro_thm_reflection_subgroups}.
\end{introtheorem}

Completions can be used to answer several  algorithmic questions about subgroups of RACGs.  For instance, 
	 we consider the problem of finite-index embeddability between RACGs:

\begin{introtheorem}[Theorem~\ref{thm_algorithm_finite_index}] \label{intro_thm_algorithm_finite_index}
	There is an algorithm 
	 (explicitly constructed in Section~\ref{sec_algorithm_for_2d_subgroups})
	which, given a one-ended, $2$-dimensional RACG  $W_{\Gamma}$, and  any RACG $W_{\G'}$, determines whether or not 
	$W_{\Gamma'}$ is isomorphic to a finite-index subgroup of $W_{\Gamma}$.
\end{introtheorem}

The above theorem gives an algorithm that can often  determine when two RACGs are commensurable; thus, it provides a tool for studying commensurability classification.  A few specific families of RACGs have been classified up to commensurability (see~\cite{crisp-paoluzzi, dani-stark-thomas, hruska-stark-tran}), but not much is known in general.  
We remark that the precise statements of Theorem~\ref{intro_thm_racg_subgroups} and Theorem~\ref{intro_thm_algorithm_finite_index} use a significantly weaker hypothesis than one-endedness.

As noted above, when $G$ is quasiconvex, a finite completion for $G$ can be computed in finite time.  
Consequently, we can provide  algorithms to check several properties of~$G$:

\begin{introtheorem}
\label{introtheorem_algorithms}
Let $G$ be a quasiconvex subgroup of a RACG $W_\G$,
 given by a finite generating set of words in $W_\G$.
Then there exist algorithms (explicitly constructed in Section~\ref{sec_other_algorithmic_properties}) to solve the following problems.
\begin{enumerate}
\item Determine whether or not $G$ is torsion-free.
\item Determine the index of $G$ in $W_\G$ (even if infinite).
\item Given a word $w$ representing an element $g \in W_\G$, determine whether or not a positive power of $g$ is in $G$.
\item Determine whether or not $G$ is normal in $W_\G$. 
\end{enumerate}
\end{introtheorem}

Finally, we use completions to recover some known results. 
Specifically, we
give new proofs 
 of  the well-known result that RACGs are residually finite (see Theorem~\ref{thm_res_finite}), and of a result of Haglund that every quasiconvex subgroup of a RACG is separable (see Theorem~\ref{thm_virtual_retract}).

We note that much of the work presented here can also be used to study RAAGs.
Any RAAG $A$ embeds as a finite-index subgroup of a RACG $W$ by a construction of Davis--Januszkiewicz~\cite{DJ}.  Thus, one can construct a completion for $G$ considered as a subgroup of $W$, and then use Theorem~\ref{intro_thm_main} and 
Theorem~\ref{introtheorem_algorithms} to study properties of $G$ as a subgroup of $A$.

\subsection{Relation to other works} \label{sec_relation}
Every RACG $W_\G$ admits an automatic structure whose associated language consists of reduced, i.e.~geodesic, words of $W_\G$~\cite{Brink-Howlett}. 
Thus, the algorithms provided by Kharlampovich--Miasnikov--Weil in~\cite{KMW} for 
quasiconvex 
subgroups of automatic groups apply to RACGs.  
These include computing their intersections and (in some cases including hyperbolic groups) the conjugacy and almost malnormality problems. 
The algorithmic applications in~\cite{KMW} are distinct 
 from those in our Theorems~\ref{intro_thm_algorithm_finite_index}~and~\ref{introtheorem_algorithms}.
We note that, 
for automatic groups satisfying an additional assumption which does not hold for all RACGs,~\cite{KMW} 
gives an algorithm for
 (2) of  Theorem~\ref{introtheorem_algorithms}, i.e.~determining the index of a quasiconvex subgroup.

To provide the algorithms above, 
Kharlampovich--Miasnikov--Weil show that finite \emph{Stallings-like graphs} can be constructed for 
subgroups that are quasiconvex with respect to an automatic structure.
We refer to~\cite{KMW} for a precise definition, but note that the 1-skeleton of every completion,   in our sense, is a Stallings-like graph.
However, completions have  strictly stronger properties than Stallings-like graphs. 
For instance, completions are \textit{cube-full}, which means they do not have ``missing relations'' (see Section \ref{sec_completion}). A priori, there is no guarantee that a Stallings-like graph can be completed
to a finite cube-full  complex, and this is  what we prove can be done for quasiconvex subgroups of RACGs (see Theorem \ref{intro_thm_main}). The cube-full property is essential to our arguments.  It allows us to prove,  in many cases, that completions are non-positively curved, which is a crucial step in our proofs of Theorem~\ref{intro_thm_reflection_subgroups} and Theorem~\ref{intro_thm_algorithm_finite_index}. 

Another important property of the completion of a subgroup $G < W_\G$ is that, 
given \textit{any} 
geodesic
word $w$ of $W_\G$ representing an element of $G$, there exists a loop  with label $w$ in the completion.
Stallings-like graphs do not necessarily
have this property, which can be traced back to the fact that, in general, the language associated to a geodesically automatic structure is strictly smaller than the set of geodesic words. 
Additionally, loops in a completion whose  labels are reduced expressions for the same group element have Hausdorff distance proportional to the number of ``moves'' required to get from one expression to the other (see Section \ref{sec_completion_properties}). These properties are used throughout our paper. 

To produce Stallings-like graphs, Kharlampovich--Miasnikov--Weil begin with a construction similar to ours and many others: they start with a rose graph, and alternately fold and ``add relations.''
For each folded graph $\Lambda$ in the resultant sequence, they use a separate algorithm, which involves constructing a new automaton associated to $\Lambda$, 
to check if the graph obtained is Stallings-like (and stopping if it is).
In contrast, our algorithm only involves checking, at each stage, whether the complex is folded and cube-full (and stopping if it is).  Consequently, it is a more practical algorithm for doing explicit computations in RACGs.

The name ``completion'' is inspired by the terminology in~\cite{KMW}, which talks of a ``completion process'' for constructing Stallings-like graphs.  We caution the reader that our completions are not the same as the canonical completions considered in~\cite{Haglund-Wise-special}.

We mention 
 some 
other works which use generalizations of Stallings' methods to study algorithmic and structural properties of subgroups, referring to~\cite{KMW}  for a more detailed summary. 

Arzhantseva--Olshanskii use Stallings-like graphs to study groups that are generic in a certain statistical sense \cite{Arzhantseva-Olshanskii}. 
They define a series of operations, now known as AO-moves, that, when applied to a Stallings-like graph, preserve the subgroup associated to it.
	These techniques have been generalized to the free product setting in \cite{Steenbock} and \cite{Gruber-thesis}, and the latter uses a notion conceptually similar to our completions.
	In a series of articles, Arzhantseva and Arzhantseva--Cherix utilize Stallings-like graphs to prove that several results hold generically in finitely presented groups  \cite{Arzhantseva-thesis, Arzhantseva-fund, Arzhantseva-comm,Arzhantseva-Cherix}. Notably, Stallings-like graphs are first used to show quasiconvexity of subgroups in \cite{Arzhantseva-comm}.

McCammond--Wise use a generalization of Stallings' construction to study local quasiconvexity and coherence in groups satisfying certain small cancellation conditions~\cite{McW}. 
Schupp uses Stallings graphs to show that certain 
 extra-large type Coxeter groups 
(which don't include RACGs) are locally quasiconvex and to answer algorithmic questions \cite{Schupp}.

Beeker--Lazarovich define a version of Stallings folds for CAT(0) cube complexes and 
use  it 
to give a characterization of finitely presented  quasiconvex  subgroups of hyperbolic groups that act geometrically on CAT(0) cube complexes~\cite{BL}.
This generalizes work of Brown on hyperbolic VH square complexes~\cite{SamuelBrown}.
The Beeker--Lazarovich characterization is different in nature to ours; it is in terms of finiteness and quasiconvexity properties of stabilizers of hyperplanes, conditions which are not necessarily easier to check algorithmically.  
Despite also using folds, their procedure follows a very different outline. 
Given a hyperbolic group $H$ acting geometrically on a CAT(0) cube complex $X$, and a finitely presented subgroup $G$, they start with an equivariant map $K\to X$, where $K$ is 
the universal cover of the presentation complex of $G$. They then apply 
apply a ``resolution process'' (generalizing work of Dunwoody) which uses the pullback in $K$ of the wall structure of $X$ to produce an
unbounded CAT(0)
 cube complex $Y$. Next, 
 	folds are $G$-equivariantly
  applied to $Y$ 
to obtain the cube complex associated to~$H$.  
In contrast, our completion process is algorithmic in nature (our starting point is an explicit generating set), and the results we obtain do not require any hyperbolicity and finite-presentation hypotheses.
Indeed, many RACGs are not even relatively hyperbolic.  
 On the other hand, we strongly rely on  combinatorial properties of RACGs 
 which do not hold in the generality of cubulated groups.

Other authors have studied ``cores'' associated to quasiconvex subgroups.  For instance,
Sageev--Wise show that relatively quasiconvex subgroups of relatively hyperbolic groups 
acting 
on a finite-dimensional, locally finite CAT(0) cube complex admit a convex core \cite{SW}. This generalizes ideas of Haglund \cite{Haglund}, who shows that a quasiconvex  subgroup $G$ of a RACG acts cocompactly on the combinatorial convex hull $\Sigma(G)$ of $G$ in the Davis complex. Indeed  the quotient  $\Sigma(G)/G$  is very close to being a completion (one has to take care, as 
cubes may be folded along midcubes under this quotient).
We note that characterizations of quasiconvexity using this approach cannot immediately be used to answer algorithmic questions.

In the spirit of Theorem~\ref{intro_thm_algorithm_finite_index}, Kim--Koberda find conditions for a RAAG to be realized as a (not necessarily finite-index) subgroup of another in terms of properties of the associated extension and clique graphs~\cite{KK}.  
Using their work, Casals-Ruiz 
 gives an algorithm which determines if a 2-dimensional RAAG is isomorphic to a subgroup of another RAAG~\cite{casals-ruiz}.

\subsection*{Acknowledgements}
	The authors would like to thank Jason Behrstock, Anthony Genevois, Ilya Kapovich, Sang-hyun Kim, Thomas Koberda, Ignat Soroko, and the referee for comments.  The second author would like to thank Nir Lazarovich and Michah Sageev for helpful conversations.

\section{Preliminaries} \label{sec_prelim}

Given a graph $\G$, we denote  the vertex and edge sets of $\G$ by $V(\G)$ and $E(\G)$ respectively. 

\subsection{Right-angled Coxeter groups}
We summarize some well-known facts regarding right-angled Coxeter group (RACGs) which are needed throughout this article. We refer the reader to \cite{Dani} for a survey on RACGs and to \cite{Davis} and \cite{BB} as references on Coxeter groups.

Let $\G$ be a simplicial graph with finite vertex set $S = V(\G)$ and edge set $E = E(\G)$. The \textit{RACG} $W_{\G}$ associated to $\G$ is the group given by the presentation:
	\[W_{\Gamma} = \langle S ~| ~ s^2 = 1 \text{ for } s \in S,~ st = ts \text{ for } (s, t) \in E  \rangle \]
We say that $S$ is a \textit{standard Coxeter generating set} for $W_{\G}$. Given $s, t \in V(\Gamma)$, we write $m(s,t) = 1$  if $s=t$, $m(s,t) = 2$ if $s$ is adjacent to $t$ and $m(s,t) = \infty$ otherwise.

We refer to the elements of $S$ as \textit{letters}. A word $w$ in $W_{\G}$ is a 
(possibly empty)
sequence of letters in $S$. 
Let $w = s_1 \dots s_n$ be a word in $W_\G$, where $s_i \in S$ for $1 \le i \le n$. We let $|w| = n$ denote the \textit{length} of $w$. If $w'$ is another word in $W_\G$ such that $w$ and $w'$ are equal as elements of $W_\G$, then we say that $w'$ is an \textit{expression} for $w$. We say that $w$ is \textit{reduced} if $|w| \le |w'|$ for any expression $w'$ for $w$. Finally, we 
define the support of $w$, denoted by $\text{Support}(w)$, to be  the set of vertices of $\Gamma$ which appear as a letter in $w$. 

Given a graph $\G$ and a subset $V'$ of $V(\G)$, the graph $\Delta$ \textit{induced} by $V'$ is the graph which has vertex set $V'$ and an edge between two vertices of $V'$ if and only if there is an edge between them in $\G$. We also say that $\Delta$ is an induced subgraph of $\G$.

Throughout this article, given any simplicial graph $\G$, we will always denote by $W_\G$ the corresponding RACG. If $\Delta$ is an induced subgraph of a graph $\Gamma$, then $W_\Delta$ is naturally isomorphic to the subgroup of $W_\Gamma$ generated by the generators corresponding to vertices of $\Delta$ (see for instance \cite{Davis}). 
Such a subgroup of $W_\G$ is called a \textit{special subgroup}.

	Given a vertex $v$ of $\G$, the \textit{link of $v$}, denoted by $\text{link}(v)$, is the set of all vertices of~$\G$ which are adjacent to $v$. The \textit{star of $v$}, denoted by $\text{star}(v)$, is the 
set $\text{link}(v) \cup \{v\}$. 
We will often consider the special subgroup of $W_{\G}$ generated by the link or star of a vertex.

Recall that a graph is a \textit{clique} if any pair of distinct vertices of the graph are adjacent. 
A RACG $W_\G$ is finite if and only if $\G$ is a clique.
Furthermore, $W_\Gamma$ is one-ended if and only if $\Gamma$ is connected and does not contain a clique which separates $\Gamma$ \cite{MT}. 
We say a graph is \textit{triangle-free} if it does not contain a subgraph that is a clique with three vertices, i.e.~a triangle.  If $\G$ is triangle-free, we say that the RACG $W_\G$ is \textit{$2$-dimensional}.

A graph $\G$ decomposes as a \textit{join graph $\G = \G_1 \star \G_2$} if there are induced subgraphs $\G_1$ and $\G_2$ of $\G$ such that $V(\G) = V(\G_1) \cup V(\G_2)$
and  $v_1$ and $v_2$ are adjacent in $\G$ for all $v_1 \in V(\G_1)$ and  $v_2 \in V( \G_2)$.
The graph $\G$ decomposes as a join $\G = \G_1 \star \G_2$ if and only if $W_\G = W_{\G_1} \times W_{\G_2}$.

We say that a subgroup of a Coxeter group is a \textit{Coxeter subgroup} if it is isomorphic to some finitely generated Coxeter group. 
In our setting, it follows from \cite[Theorem~12.3.4]{Davis} and \cite[Proposition 1.1.1]{BB} that Coxeter subgroups are themselves RACGs:
\begin{proposition} \label{prop_coxeter_subgroups_are_racgs}
	Let $G$ be a Coxeter subgroup of a RACG, then $G$ is a RACG. 
	\qed
\end{proposition}

We will need the following definition in Section~\ref{sec_algorithm_for_2d_subgroups}.

\begin{definition}\label{def_doubling}
Given a graph $\G$ and $s \in V(\G)$, let $\Lambda$ be the subgraph of $\G$ induced by $V(\G) \setminus \{s\}$. We define $D(\Gamma, s)$ to be the graph consisting of the union of two copies 
of $\Lambda$ which are identified along the subgraph of $\Lambda$ 
induced by $\text{link}(s)$.
 
 It is well-known that $W_{D(\Gamma, s)}$ is isomorphic to the kernel $K$ of the homomorphism 
$\phi_s: W_\G \to \mathbb{Z}_2 $ 
defined by 
$\phi_s(s) = 1$ and $\phi_s(t) = 0$
 for all $t \in V(\Gamma) \setminus \{s\}$. 
To see this, first observe that $K$ is generated by  
$T=\{t,  sts \mid t \in V(\G) \setminus \{s\} \}$. Since $T$ consists of reflections, a result proved independently by Dyer~\cite{Dyer} and Deodhar~\cite{Deodhar} implies that $K$ is a Coxeter group, and $K$ is therefore a RACG by Proposition \ref{prop_coxeter_subgroups_are_racgs}.  By the criterion in the first paragraph of~\cite{Dyer}, the set $T'$ is a Coxeter generating set, where $T'$ consists of all reduced reflections $r$ in $K$ with the property that for every reflection $r' \neq r$ in $K$, the word $r'r$ cannot be represented by a word of length less than $|r|$.  It is easily seen from Tits' solution to the word problem that $T \subset T'$.  Since both these sets generate $K$ and as no subset of $T'$ can generate $K$, it follows that $T=T'$, and consequently, that $K$ is isomorphic to $W_{D(\G,s)}$. 

\end{definition}

\subsection{The word problem in RACGs} \label{subsec_word_prob}
We discuss Tits' solution to the word problem  and the deletion property in RACGs. We  refer the reader to \cite{Davis} or \cite{BB} for more details.

Let $w = s_1 \dots s_n$ be a word in the RACG $W_\G$. Suppose that $m(s_i, s_{i+1}) = 2$ for some $1 \le i \le n$. Then we may ``swap'' the letters $s_i$ and $s_{i+1}$ to obtain another expression $w' = s_1 \dots s_{i-1}s_{i+1}s_i s_{i+2} \dots s_n$ for $w$. We say that $w'$ is obtained from $w$ by a \textit{swap move} or by \textit{swapping} $s_i$ and $s_{i+1}$. On the other hand, if $s_i = s_{i+1}$ (as vertices of $\G$) 
for some  $1 \le i \le n$, then we can obtain an expression $w' = s_1 \dots s_{i-1}s_{i+2}\dots s_n$ for $w$ by \textit{cancelling} $s_i$ and $s_{i+1}$.

Let $w$ be a word in $W_\G$, and let $w'$ be a reduced expression for $w$.  There exists a sequence of words $w = w_1, \dots, w_m = w'$ such that $w_{i+1}$ is obtained from $w_i$ by either a swap move or a cancellation. This is known as Tits' solution to the word problem. We call such a sequence of expressions for $w$ \textit{a sequence of Tits moves}.

Given a word  $w = s_1 \dots s_n$ in $W_{\G}$, suppose
 $s_i = s_{i'} = s$ for some $1 \le i < i' \le n$ and that $m(s, s_j) = 2$ for all $i < j < i'$. It follows that $w' = s_1 \dots s_{i-1}s_{i+1} \dots s_{i'-1}s_{i'+1} \dots s_n$ is an expression for $w$. We say that $w'$ is obtained from $w$ by a \textit{deletion} (as two occurrences of $s$ have been deleted). 
We remark that this definition is slightly stronger than the classical definition of a deletion in  general Coxeter groups.  
The  \emph{deletion property} states that if 
if $w$ is a word in $W_\G$ such that $w$ is not reduced, a deletion can be applied to $w$.  
The proof of this fact is similar to that of the corresponding statement for the standard definition of a deletion~\cite{Bahls}.
The deletion property guarantees that a reduced expression for $w$ can be obtained by performing  a sequence of deletions.  

The following lemma, which is required later, easily follows from the deletion property.
\begin{lemma} \label{lemma_reduced_expression}
	Let $h$ and $k$ be reduced words in a RACG $W_\G$. Then there is a reduced expression $\hat{h} \hat{k}$ for the word $hk$ such that $\hat{h}s_1\dots s_m$ is a reduced expression for $h$ and $s_m \dots s_1 \hat{k}$ is a reduced expression for $k$ where $s_i \in V(\G)$ for $1 \le i \le m$.
\end{lemma}

\subsection{Cube complexes}
A \textit{cube complex} is a cell complex whose cells are Euclidean unit cubes, $[-\frac{1}{2}, \frac{1}{2}]^n$, of varying dimension. We refer the reader to \cite{Caprace-Sageev} and to \cite{Wise-riches} for a background on cube complexes.  

Let $\G$ be a simplicial graph. A cube complex is \emph{$\G$-labeled} if every edge in its 
$1$-skeleton is labeled by a vertex of $\Gamma$. 
In this article, the labels of a $\G$-labeled cube complex will be the generators of the RACG $W_\G$. 
Given a simplicial path $\alpha$ in the $1$-skeleton of a $\G$-labeled complex, the \textit{label of $\alpha$} is the word formed by the sequence of labels of consecutive edges in~$\alpha$.

We say 
that a cube complex $\Omega$ is \textit{non-positively curved} if the 
(simplicial) link of each vertex in $\Omega$ is a flag 
complex. 
(Recall that a simplicial complex is called a \emph{flag} complex if 
	any finite clique in its $1$-skeleton
spans a simplex.)
If $\Omega$ is both non-positively curved and simply connected then 
$\Omega$ is a CAT(0) cube complex. 

A \textit{path} in the cube complex $\Omega$ is a simplicial path in its $1$-skeleton. Given a path $p$, we denote by $|p|$ the number of edges in $p$. Given two paths, $p$ and $p'$, such that the endpoint of $p$ is equal to the startpoint of $p'$, we let $pp'$ denote their concatenation. A \textit{loop} in a complex is defined to be a closed path (possibly with backtracking). We define a \textit{graph-loop} to be an edge in a complex that connects a vertex to itself. 

We will work with the combinatorial path metric on cube complexes. Namely, given two vertices of $\Omega$, we define their distance to be the length of a shortest path in $\Omega$ between them.

A \textit{midcube} of a cube $c = [-\frac{1}{2}, \frac{1}{2}]^n$ is the restriction of one of the coordinates of $c$ to $0$. 
A \textit{hyperplane} $H$ in $\Omega$ is a maximal collection of midcubes in $\Omega$, such that for any two midcubes $m$ and $m'$ in $H$, it follows there is a sequence of midcubes $m = m_1, \dots, m_n = m'$ in $H$ such that $m_i \cap m_{i+1}$ is a midcube in $\Omega$ for all $1 \le i < n$. 
The \textit{carrier} of a hyperplane $H$, denoted by $N(H)$, is the set of all cubes which have non-empty intersection with $H$. If $H$ intersects an edge $e$, then we say that $e$ is dual to $H$. 

Let $\Omega$ be a CAT(0) cube complex. 
A path $\gamma$ in $\Omega$ is geodesic if and only if every hyperplane is dual to at most one edge of $\gamma$. Thus, if $\gamma$ is geodesic then $|\gamma|$ is equal to the number of hyperplanes intersecting~$\gamma$. 
Given a hyperplane $H$ in $\Omega$, $\Omega \setminus H$ contains exactly two components, and the carrier $N(H)$ is convex in the combinatorial path metric.
We refer the reader to \cite[Chapters 3.2 and 3.3]{Wise-riches} for proofs of these well-known facts. 
\subsection{Disk diagrams in cube complexes}
We now recall some basic facts about disk diagrams, and 
refer to \cite{Wise} and \cite{Wise-riches} for further details.  

In the setting of cube complexes, a \textit{disk diagram} $D$ is a contractible, finite, $2$-dimensional cube complex (i.e.,~a square complex) equipped with a planar embedding $\Psi : D \to \mathbb{R}^2$. The map $\Psi$ gives a natural cellulation of the $2$-sphere $S^2 = \mathbb{R}^2 \cup \infty$. We call the path traced  
by an attaching map of the cell containing $\infty$ in this cellulation the \textit{boundary of $D$} and denote it by $\partial D$.

Given a cube complex $\Omega$, a \textit{disk diagram 
over
 $\Omega$} is a disk diagram $D$ which admits a map  $\Phi: D \to \Omega$ mapping $n$-cubes isometrically onto $n$-cubes (i.e.,~a combinatorial map). As the edges of the cube complexes we consider in this  article will 
be labeled, we accordingly further require the edges of $D$ to be labeled and the map from $D$ to $\Omega$ to respect this labeling. 

Given 
a closed null-homotopic loop $p: \mathbb{S}^1 \to \Omega$, the van Kampen Lemma says that  there exists a disk diagram in $\Omega$ given by some 
$\Phi: D \to \Omega$, and an identification of $\partial D$ with $\mathbb{S}^1$ such that $\Phi$ restricted to $\partial D$ is equal, as a map, to $p$ (see for instance \cite[Lemma 3.1]{Wise-riches}).

Given a disk diagram $D$ in $\Omega$ and an edge $e$ of $D$, \textit{ 
the 
dual curve 
intersecting $e$} is
 the 
 hyperplane in $D$ dual to $e$.
As $D$ is planar, a dual curve in $D$ can intersect at most two edges along $\partial D$.  
	
Let $W_\G$ be a RACG. 
The Cayley graph of $W_\Gamma$ (with the usual presentation) is the $1$-skeleton of a CAT(0) cube complex (the \textit{Davis complex}) whose $2$-cells correspond to the commuting relations between generators (see for instance \cite{Davis}). 
In this setting, a $\G$-labeled disk diagram over the Davis complex of $W_\Gamma$
has the property that opposite sides of squares must have the same label.
In particular, we can define the \textit{type} of a dual curve to be the label of an edge (equivalently, all edges) 
intersecting the dual curve. Furthermore, if two dual curves intersect, then their types are adjacent vertices of $\G$.
These observations are used in the following technical lemma which is useful in later sections.  
Note that any word considered in the lemma below could be the empty word.

\begin{lemma} \label{lemma_disk_diagram_subwords}
	Let $w$ and $z$ be words that are equal as elements of a RACG $W_{\Gamma}$. Suppose that $w = w'w''$ and $z = z'z''$, where $w'$, $w''$, $z'$ and $z''$ are words in $W_\Gamma$, and suppose that $z'$ is reduced. Let $D$ be a disk diagram 
whose boundary label, starting from a base vertex $v$, is $wz^{-1}$. 
Let $\alpha_{w'}$ and $\alpha_{z'}$  be the oriented paths starting at $v$ in 
$\partial D$
whose labels, read in the direction of the orientation, are $w'$ and  $z'$ respectively. 
	Suppose further, that every dual curve intersecting $\alpha_{w'}$ also intersects $\alpha_{z'}$. Then $z'$ has a reduced expression $w'x$, where $x$ is some word in $W_{\Gamma}$. 
\end{lemma}
\begin{proof}	
	Let $H_1, \dots, H_k$, with $k = |\alpha_{w'}|$, be the dual curves intersecting $\alpha_{w'}$ ordered by the orientation of $\alpha_{w'}$ (i.e. for $i < j$, $H_i \cap \alpha_{w'}$ occurs before $H_j \cap \alpha_{w'}$ along this orientation). 
This order is well-defined, since each $H_i$ necessarily also intersects $\alpha_{z'}$ by hypothesis, and therefore intersects $\alpha_{w'}$ exactly once. 
	
	Let $e_1, \dots, e_k$ be the first $k$ edges of $\alpha_{z'}$ ordered by its orientation. If $H_i$ intersects $e_i$ for each $1 \le i \le k$, then we are done. For then $w'$ is precisely the initial subword of $z'$ of length~$k$.
	Otherwise, let $c$ be the smallest integer such that $H_c$ does not intersect $e_c$.
	In particular, the initial subword of $w'$ of length $c-1$ is equal to the initial subword of $z'$ of length $c-1$.
	
	We claim that we can construct a new disk diagram $D'$ with boundary label $w'w''(yz'')^{-1}$, such that $y$ is a reduced word equal in $W_\G$ to $z'$, every dual curve intersecting the part of 
	$\partial D'$ 
	with label $w'$ also intersects the part of 
		$\partial D'$ labeled $y$, 
	and the initial subword of $y$ of length~$c$ is equal to the initial subword of $w'$ of length $c$. The lemma then follows by iteratively applying this claim.

	Orient the edges of $\alpha_{z'}$ with the induced orientation from $\alpha_{z'}$.  	
		Let $e$ be the edge of $\alpha_{z'}$ that intersects $H_c$. 
		Define $\beta$ to be the subpath of $\alpha_{z'}$ from $v$ to the initial vertex of $e$ if $c=1$ and from the terminal vertex of $e_{c-1}$ to the initial vertex of $e$ otherwise.  If $H$ is a dual curve intersecting $\beta$, then $H\neq H_i$ for $1 \le i \le c$ by construction, and $H$ cannot intersect $\alpha_{z'}$ twice, since $z'$ is reduced.  Thus, any such $H$ must intersect $H_c$.  It follows that the label $s$ of $e$ commutes with the label $b$ of $\beta$.  
		We may now attach a $|b| \times 1$ rectangle with label $bsb^{-1}s$ to the 
			$\partial D$ 
by identifying the sides labeled $b$ and $s$ of the rectangle with $\beta$ and $e$ respectively, to get a new diagram $D'$. Define $y$ to be the word obtained from $z'$ by replacing the subword labeled $bs$ by $sb$.  It is clear by construction that $D'$ has the properties in the claim. 
\end{proof}

\section{A complex for subgroups of a RACG} \label{sec_completion}

The main goal of this section is to define a completion of a subgroup of a RACG, and to construct completions for finitely generated subgroups.  We begin by defining 
 a completion of a $\G$-labeled complex as the direct limit of a certain sequence of $\G$-labeled complexes.   We then show that there is a natural labeled graph associated to any 
finite generating set, such that a completion of this graph is also a completion of the group generated by the set.

\subsection{Completion of a complex}
\label{subsec:completion}
Let $\Gamma$ denote a simplicial graph. In this paper, we only consider $\G$-labeled cube complexes whose labeled cubes have two additional properties. 
Firstly, any pair of edges dual to a  common
mid-cube have the same label. As a result, hyperplanes in the cube complex have a well-defined label.  
Secondly, given a cube in such a complex and a set of edges of this cube which are all incident to  a common vertex, no two edges in this set have the same label, and the full subgraph of $\G$ induced by the vertices of $\G$ corresponding to the labels of the edges in this set is a clique.
 When we mention a $\G$-labeled cube complex, it will be implicit that the labeling has these additional properties. 

Let $C$ be a $\G$-labeled cube complex.  We describe three operations that can be applied to $C$  to produce a new $\G$-labeled cube complex. 

\subsection*{Fold operation:} 
 A \textit{fold operation} corresponds to collapsing a pair of adjacent edges with the same label into a single edge.  More precisely, for $i=1,2$, let $e_i$ be an edge in $C$ with endpoints $v$ and $v_i$, where $e_1 \neq e_2$, but two or more of the vertices $v, v_1, v_2$ could be equal.  
Furthermore, suppose that $e_1$ and $e_2$ have the same label.  
 Temporarily orient the edge $e_i$ from $v$ to $v_i$ (choosing the orientation arbitrarily if $v=v_i$, i.e.~if $e_i$ is a graph-loop).  Then the fold operation consists of forming a quotient of $C$ by identifying $e_1$ and $e_2$ so that their orientations agree, and then forgetting the orientation. 
 
 We remark that although the fold map corresponding to $e_1$ and $e_2$ is not unique when one of these edges is a graph-loop, this does not affect any of our applications.

\subsection*{Cube identification operation:}
This is a higher dimensional analogue of a fold operation. 
Consider a collection of two or more distinct $i$-cubes in $C$, with $i \ge 2$, whose boundaries are equal. A cube identification operation consists of forming the quotient of $C$ in which all of the $i$-cubes in the collection have been identified to a single cube.  Note that the $1$-skeleton does not change in this process.

\subsection*{Cube attachment operation:} Consider an $i$-tuple $e_1, \dots, e_i$ of edges in $C$, with labels $s_1, \dots, s_i$, which are all incident to a single vertex $v$.  Suppose furthermore, that the vertices corresponding to 
$s_1, \dots, s_i$ in $\G$ form an $i$-clique.
A cube attachment operation consists of adding an $i$-cube $c$ to $C$ by identifying the edges 
$e_1, \dots, e_i$ to $i$ edges in $c$ which are all incident to a single vertex of $c$.  In the process, we  add some vertices and edges to $C$.  
Each new edge added is dual to a mid-cube of $c$ which is also dual to a unique edge in the set $\{e_1, \dots e_i\}$.  This induces a labeling on the newly added edges, making the resultant complex $\G$-labeled. 

\bigskip
We say a complex is \textit{folded} if no fold or cube identification operation can be performed to it.
As fold operations and cube identification operations reduce the number of cells, any finite complex $C$ can be transformed into a folded complex through finitely many such operations.

We say a complex is \textit{cube-full} if for any 
$i$-tuple of edges all incident to the same vertex such that the vertices corresponding to their labels 
form an $i$-clique in $\G$, 
there exists an $i$-cube of $C$ whose boundary contains  these $i$ edges.  

Given a connected finite $\G$-labeled
 complex $X$, consider a possibly infinite sequence: 
\[\Omega_0 = X \xrightarrow{f_0}  \Omega_1 \xrightarrow{f_1} \Omega_2 \cdots \]
where for each $i$, the map $f_i : \Omega_i \to \Omega_{i+1}$ is either a fold, cube identification or cube attachment operation.
Let $\Omega_X$ be the direct limit of this sequence. If $\Omega_X$ is folded and cube-full, we call $\Omega_X$ a \textit{completion} of $X$. We say 
\[\Omega_0 = X \xrightarrow{f_0}  \Omega_1 \xrightarrow{f_1} \Omega_2 \cdots \to \Omega_X \] 
is a completion sequence for $X$. We sometimes leave the maps $f_i$ out of the notation when these maps are not relevant. We also set $\hat{f}: X \to \Omega_X$ as the direct limit of the maps $\{f_i\}$. 

\begin{example}\label{ex_completion_ex1}
Let $\G_1$ be the graph in Figure~\ref{fig_completion_ex1}.  
The right of the figure shows a completion sequence for the $\G_1$-labeled complex $X$. The completion $\Omega$ is obtained from $X$ by a fold operation followed by a cube attachment operation. Note that not all labels of $\Omega$ are shown.
	\begin{figure}[h!]
	\centering
	\begin{overpic}
	[scale=1.3]
	{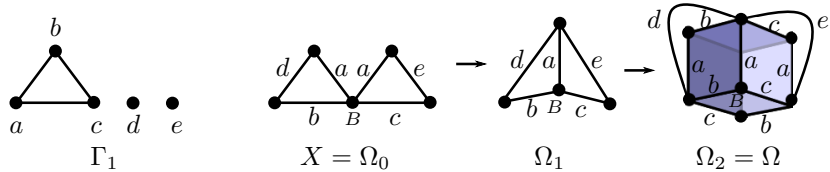}
		\put(0,-1){\small $a$}
		\put(5,11){\small $b$}
		\put(10,-1){\small $c$}
		\put(14.5,-1){\small $d$}
		\put(20,-1){\small $e$}
		\put(40.25, 6){\small $a$}
		\put(43, 6){\small $a$}
		\put(33,6){\small $d$}
		\put(37,0){\small $b$}
		\put(47,0){\small $c$}
		\put(50,6){\small $e$}
		\put(41.5, 0){\tiny $B$}
		
		\put(64, 1){\small $b$}
		\put(70,1){\small $c$}
		\put(66,7){\small $a$}
		\put(62,7){\small $d$}
		\put(72,7){\small $e$}
		\put(66.5, 1.5){\tiny $B$}

		\put(84.5,6.5){\small $a$}
		\put(91, 7){\small $a$}
		\put(95,6){\small $a$}
		\put(85.5,12){\small $b$}
		\put(86.5, 4){\small $b$}
		\put(93,-1){\small $b$}
		\put(86,0){\small $c$}
		\put(94,11.5){\small $c$}
		\put(93,4){\small $c$}
		\put(79,12){\small $d$}
		\put(100,12){\small $e$}
		\put(89,2.25){\tiny $B$}
 		\put(10,-5){\small $\G_1$}
		\put(36,-5){\small $X = \Omega_0$}
		\put(65,-5){\small $\Omega_1$}
		\put(85,-5){\small $\Omega_2= \Omega$}

%
%
%
%
%
%
	\end{overpic}
	\bigskip
	\caption{A completion $\Omega$ for the $\G_1$-labeled complex $X$.}
	\label{fig_completion_ex1}
\end{figure}	
\end{example}

\begin{example}\label{ex_completion_ex2}
Figure~\ref{fig_completion_ex2} shows a graph $\G_2$ and a $\G_2$-labeled complex $X$.
A standard completion $\Omega'$ of $X$ (see Definition~\ref{def_standard_completion_of_complex}) is shown on the right (the labels of $\Omega'$ are omitted). The cube complex $\Omega'$ is topologically a bi-infinite cylinder.
\begin{figure}[h!]
	\centering
	\begin{overpic}
	[scale=.6]
	{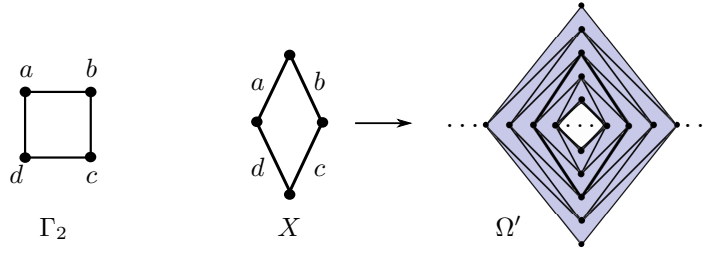}
		\put(0, 26){\small $a$}
		\put(10, 26){\small $b$}
		\put(-1.5,10){\small $d$}
		\put(10,10){\small $c$}
		
		\put(35, 24){\small $a$}
		\put(44.5, 24){\small $b$}
		\put(35,11){\small $d$}
		\put(44.5,11){\small $c$}
		
		\put(3,2){\small $\G_2$}
		\put(39, 2){\small $X$}
		\put(72,2){\small $\Omega'$}
		\put(64.5, 18.5){$\dots$}
		\put(100.5, 18.5){$\dots$}
		\put(82.5, 18.5){\small $\dots$}
%
%
	\end{overpic}
	\caption{A completion $\Omega'$ for the $\G_2$-labeled complex $X$. }
	\label{fig_completion_ex2}
\end{figure}
\end{example}

Additional examples of completions can be found in 
	Examples~\ref{ex:fi_completion}, \ref{ex:nonqc_completion} and~\ref{ex:reflection_subgroup}.

\begin{proposition}[Existence of a completion] \label{prop_existence}
	Given any finite $\G$-labeled complex $X$, there exists a completion $\Omega_X$ of $X$.
\end{proposition}

\begin{proof}

We set $\Omega_0 = X$ and build $\Omega_i$ inductively. 
Suppose a finite complex $\Omega_i$ was obtained from $\Omega_0$ by a sequence of fold, cube identification and cube attachment operations.
We iteratively perform fold and cube identification operations to $\Omega_{i}$ to obtain the complexes $\Omega_{i+1}, \Omega_{i+2}, ..., \Omega_{i+j}$, where $\Omega_{i+j}$ is folded. (This includes the case $i=0$.)
As $\Omega_{i}$ is finite,  we conclude that $j$ (and hence $\Omega_{i+j}$) is finite as well. 

Next we describe a sequence of operations to be performed to the finite folded complex $\Omega_{i+j}$.  
Choose a vertex $v$ of $\Omega_{i+j}$.  Consider a maximal tuple of edges incident to $v$ such that their labels form a clique in $\G$.  If there is no cube in $\Omega_{i+j}$ whose boundary contains the tuple of edges, then attach an appropriately labeled cube of the appropriate dimension along the tuple of edges.  
Do this for each such maximal tuple at $v$, and then proceed to do the same for all the vertices of $\Omega_{i+j}$.  The result is a sequence of complexes 
$\Omega_{i+j+1}, \Omega_{i+j+2},...\Omega_{i+j+k}$, such that each one is obtained from the previous one by attaching an $n$-cube for some $n$ to an $n$-tuple of edges of $\Omega_{i+j}$ which are all incident to a single vertex of $\Omega_{i+j}$.  As $\Omega_{i+j}$ is finite 
 we conclude that $k$ (and hence $\Omega_{i+j+k}$) is finite. We then repeat the above procedure starting with the finite 
complex $\Omega_{i+j+k}$.

	Let $\Omega_X$ be the direct limit of these complexes. Consider a pair of edges, say $e$ and $f$, in $\Omega_X$ incident to the same vertex $v$. It follows that some $\Omega_i$ contains  preimages of $e$ and $f$ which are incident to the same vertex. 
	Consequently, there is some folded $\Omega_{i'}$, with $i' \ge i$, which contains preimages of $e$ and $f$ that are incident to the same vertex. Thus, if $e$ and $f$ have the same label, then their preimages in $\Omega_{i'}$ must be identified.  A similar argument shows that two cubes with the same boundary in $\Omega_X$ must be identified. It follows that $\Omega_X$ is folded.

	Let $e_1, \dots, e_n$ be edges all incident to a common vertex $v \in \Omega_X$ whose labels form an $n$-clique in $\G$. There is some folded  $\Omega_i$ which contains a preimage $v'$ of $v$ and preimages $e_1', \dots, e_n'$ of $e_1, \dots, e_n$ so that $e_i', \dots, e_n'$ are all incident to $v'$. As a result of the procedure described above, some $\Omega_{i+j}$ is the complex resulting from a cube attachment operation where an $n$-cube, say $c$, is attached to the image of the edges $e_1', \dots, e_n'$ in $\Omega_{i + j - 1}$. Thus, the image of $c$ in $\Omega_X$ is an $n$-cube containing the edges $e_1, \dots, e_n$. This shows $\Omega_X$ is cube-full.
\end{proof}

\begin{definition}[Standard Completion] \label{def_standard_completion_of_complex}
	We call the completion algorithm given in the proof of Proposition~\ref{prop_existence}
	 a
	\textit{standard completion} and the associated sequence
	$\Omega_0 \to \Omega_1 \to \dots \to \Omega_X$
	a standard completion sequence. 
\end{definition}
We show that if a completion is finite, then there is indeed 
an algorithm to obtain it that terminates.


\begin{proposition} \label{prop_finite_completion}
	Let $X$ be a $\G$-labeled complex. 
Consider a standard completion sequence 
	\[X = \Omega_0 \to \Omega_1 \to\dots \to \Omega.\]
	If $\Omega$ is finite, then the completion sequence is finite, i.e.~$\Omega = \Omega_N$ for some $N$.
\end{proposition}
\begin{proof}
	As $\Omega$ is finite, by the definition of a direct limit, for some $M$ and all $n \ge M$, $\Omega_n$ contains an isometrically embedded subcomplex $Y_n$ isometric to $\Omega$, and the natural map $f: \Omega_n \to \Omega$ is a label-preserving isometry when restricted to $Y_n$. 
	As we have a standard completion, there exists an $N \ge M$ such that $\Omega_N$ is folded.
	
	We claim $\Omega_N = \Omega$. Suppose instead that $Y_N \subsetneq \Omega_N$. If $\Omega_N \setminus Y_N$ contains a vertex or an edge, then since $\Omega_N$ is connected, 
	it follows that some vertex $v \in Y_N$ is incident to an edge $e$ that is not contained in $Y_N$. Let $s$ be the label of~$e$. As $\Omega_N$ is folded, no edge in $Y_N$ that is incident to $v$ has label $s$. However, the continuous map $f: \Omega_N \to \Omega$ sends $e$ to an edge incident to $f(v)$ labeled by $s$. This is a contradiction.  Thus $Y_N$ and $\Omega_N$ have the same $1$-skeleton.  

	Suppose there is a $2$-cell $c$ in $\Omega_N \setminus Y_N$.  Then the boundary $\partial c$  of $c$ is contained in $Y_N$. Since $\Omega_N$ folded, there is no cube in $Y_N$ with boundary $\partial c$.  However, $f$ sends $\partial c$ to the boundary of a cube in $\Omega$, and $f(c)$ is a cube in $\Omega$ with boundary $f(\partial c)$, leading to a contradiction.   Thus  $\Omega_N $ and $ \Omega$ have the same $2$-skeleton.  Proceeding inductively, we conclude that $\Omega_N = \Omega$
\end{proof}

\subsection{Completion of a subgroup}

In this subsection, we define the completion of a subgroup of a RACG and show a completion is guaranteed to exist for any finitely generated subgroup. 

\begin{definition}[Completion of a subgroup] \label{def_completion}
	 Let $G$ be a subgroup of a  RACG $W_{\Gamma}$, and 
	 let $\Omega$ be a connected 
	$\G$-labeled cube complex
	with basepoint the vertex $B \in \Omega$. We say that $(\Omega, B)$ is \textit{a completion of $G$} if:
	\begin{enumerate}
		\item \label{def_completion_folded_cube_full} $\Omega$ is folded and cube-full.
		\item \label{def_completion_loops} Given any loop in $\Omega$ based at $B$, its label is a word which represents an element of $G$.
		\item \label{def_completion_reduced_words} For any \textit{reduced} word $w$ in $W_\Gamma$ which represents an element of $G$, there is a loop $l$ based at $B$ with label $w$.
	\end{enumerate}
\end{definition}
\begin{remark}
When $B$ is not relevant, we may simply say $\Omega$ is a completion of $G$.
\end{remark}

\begin{example}\label{ex_non-unique}
Let $G$ be a finite subgroup of $W_\G$ generated by adjacent vertices $a$ and $b$ in  $\G$.   
Let $X$ be the rose graph consisting of one vertex, one graph-loop  labeled by $a$ and one graph-loop labeled by $b$.  Let $\Omega_1$ and $\Omega_2$ respectively be the torus and Klein bottle obtained by attaching a (square) 2-cell to $X$. 
Then both $\Omega_1$ and $\Omega_2$ are completions for $G$.  
\end{example}

When the group $G< W_\G$ is finitely generated, we can construct a completion for $G$ as 
follows. 
 Let  $G$ be generated by the finite generating set of words 
\[S_G =  \{w_i = s_{i_1}s_{i_2}...s_{i_{m_i}} ~|~  1 \le i \le n \}\]
 where $s_{i_j} \in V(\Gamma)$ for each $i, j$. 
We associate to $S_G$ the following $\G$-labeled complex. We begin with a single base vertex $B$. For each generator $w_i$, we attach a circle subdivided to have $m_i$ edges, such that edges of this circle are sequentially labeled, beginning at $B$, by the letters $s_{i_j}$ for $1 \le j \le m_i$. We denote this  based complex by $(X(S_G), B)$ and call it the \textit{$S_G$-complex}.

Let $\Omega$ be a completion of $X(S_G)$ with completion sequence
	\[X(S_G) = \Omega_0 \to \Omega_1 \to \Omega_2\to \cdots \to \Omega. \] 
By a slight abuse of notation, we use $B$ to denote the image of the base point 
$B$ in $\Omega_i$ for any $i$, as well as in $\Omega$.
The next few lemmas show that $\Omega$ is a completion of $G$.

\begin{lemma} \label{lemma_omega_reduced_words} Let $G$ be a subgroup of a RACG $W_\G$, given by a finite generating set $S_G$. Let $\Omega$ be any completion of $X(S_G)$ where $(X(S_G), B)$ is the $S_G$-complex. 
If $w$ is a reduced word in $W_\G$ which represents an element of $G$, then $w$ is the label of some loop  in $\Omega$ based at $B$.
\end{lemma}

\begin{proof}
	Let 
	$X(S_G) = \Omega_0 \to \Omega_1 \to \Omega_2 \to \cdots \to \Omega$
	be a completion sequence for $X(S_G)$, and let $w$ be a reduced word in $W_\G$ which represents an element of $G$.
	
	As $S_G$ is a generating set of $G$, it follows that $w$ is equal in $W_{\Gamma}$ to a word $w' = h_{1}...h_{k}$ where $h_i \in S_G$ for each $i$. 
	By construction, for each $1 \le i \le k$, there is a loop $l_i$ based at $B$ in $\Omega_0$, with label $h_i$. Let $l$ be the loop in $\Omega_0$ formed as a concatenation of loops: $l_1 l_2 \dots l_k$. Let $\hat{l}$ be the image of $l$ in $\Omega$. Then $\hat{l}$ has the same label as $l$.
	
	As $w'$ and $w$ are equal as elements of $W_{\Gamma}$, the word $w$ can be obtained from $w'$ through a sequence of Tits moves. Suppose the first Tits move in this sequence is a swap performed to $w'$ to obtain a new word $w''$.
	
	We claim $w''$ is the label of a loop in $\Omega$ as well. Note that there are adjacent edges $e$ and $f$ of $\hat{l}$, labeled by $s$ and $t$ where $s, t \in V(\Gamma)$ and $m(s,t) = 2$, such that  $w' = a_1...a_ista_{i+1}...a_m$ and $w'' = a_1...a_itsa_{i+1}...a_m$, with $a_j \in V(\G)$. As $m(s,t) = 2$ and $\Omega$ is cube-full, there must be a square $Q$ in $\Omega$ whose boundary contains $ef$. We now obtain the desired loop by replacing $ef$ in $\hat{l}$ with the opposite path in the boundary of $Q$.
	
	On the other hand, suppose the first Tits move is a cancellation. In other words, $w' = a_1\dots a_issa_{i+1} \dots a_m$ is replaced by $w'' = a_1 \dots a_m$ where $s \in V(\G)$ and $a_i \in V(\G)$ for each $i$. As $\Omega$ is folded, $\hat{l}$ must traverse an edge $e$, labeled by $s$, twice consecutively. It follows that either $e$ is a graph-loop or that $\hat{l}$ traverses $e$ in one direction and immediately backtracks in the other direction. In either case, we can simply remove these two occurrences of the edge $e$ from $\hat{l}$ to obtain a new loop based at $B$ with label $w''$.
	
	By repeating this procedure for each Tits move, we obtain a loop in $\Omega$ with label  $w$.
\end{proof}

\begin{lemma} \label{lemma_pull_back}
	Let $\Omega$ be a $\G$-labeled complex obtained by applying either a fold, cube identification or cube attachment operation to the $\G$-labeled complex $\bar{\Omega}$. Let $F: \bar{\Omega} \to \Omega$ be the natural map. Let $B$ be a vertex of $\Omega$ and let $\bar{B}$ be a vertex that is in the preimage under $F$ of $B$. Let $w$ be the label of a loop $l$ in $\Omega$ based at the vertex $B$. Then there exists a loop $\bar{l}$ in $\bar{\Omega}$ based at $\bar{B}$, with label $\bar{w}$, such that $\bar{w}$ and $w$ represent the same element of the RACG $W_\G$.
\end{lemma}
\begin{proof}
	We analyze each type of operation separately:

	\medskip \noindent
	\textbf{Cube identification operation}: 
	If $F$ is a 
	cube identification operation, 
	then $\Omega$ and $\bar{\Omega}$ have the same $1$-skeleton. Thus $l$ is the image of a loop $\bar{l}$ in $\bar{\Omega}$ with the same label as $l$.
	
		\medskip \noindent
	\textbf{Fold operation:}
Suppose $F$ is a fold, and let $B = u_1, \dots, u_m = B$ be the vertices of $l$ listed sequentially by the orientation of $l$.

	Suppose some vertex, say $v$, of $\Omega$ has preimage $F^{-1}(v) = \{\bar{v}_1, \bar{v}_2\}$. As only a single edge is folded in a fold operation, there is at most one such vertex. Let $\bar{f}_1$ and $\bar{f}_2$ be the two edges in $\bar{\Omega}$ which are folded and let $f$ be the edge in $\Omega$ which is their image. Let $s$ be the label of $f$. The endpoints of $f$ must be $v$ and some vertex $v'$ (which is possibly equal to $v$). 
	
	We say a vertex or edge of $l$ has unique preimage if its preimage under $F$ is a single vertex or edge. 
	It is straightforward to check that $l$ can be subdivided into subpaths of the five types described below (though not all the types may be used):
	\begin{enumerate}
		\item An edge $p$ from $u = u_i$ to $u' = u_{i+1}$ where $u$ and $u'$ each have unique preimage.  
		\item A path $p$ from $u = u_i$ to $u'$ = $u_{i'}$, where $u$ and $u'$ each have unique preimage, and $u_{j} = v$ for every $i < j < i'$.
		\item A path $p$ from $B=u_1$ to $u = u_i$, where $u$ has unique preimage 
		and $u_{j} = v$ for $j < i$.
		\item A path $p$ from $u = u_i$ to $B=u_m$, where $u$ has unique preimage 
		 and $u_{j} = v$ for $j > i$.
		\item $p = l$ and $u_i = v$ for all $i$.
	\end{enumerate}
	
	We claim that for each path $p$ of a type described above, there is a path $\bar{p}$ in $\bar{\Omega}$ such that the label of $p$ and the label of $\bar{p}$ are equal in $W_\G$. Additionally, the image under $F$ of the endpoints of $\bar{p}$ are equal to the endpoints of $p$. Finally, if $p$ is of type 3, then $\bar{p}$ begins at $\bar{B}$, if $p$ is of type 4 then $\bar{p}$ ends at $\bar{B}$ and if $p$ is of type 5 then $\bar{p}$ begins and ends at $\bar{B}$. The lemma clearly follows from this claim. We proceed to prove the claim for each type of subpath of $l$.
	
	\textit{Type 1:} 
	Let $p=e$ be the edge in $l$ between $u = u_i$ and $u' = u_{i+1}$.
	Let $\bar{u}$ and $\bar{u}'$ be the unique preimages of $u$ and $u'$ under $F$. The preimage $F^{-1}(e)$ is either a single edge between $\bar{u}$ and $\bar{u}'$ 
	or is a pair of edges between $\bar{u}$ and $\bar{u}'$ 
	 (in this case the fold operation identifies this pair of edges to get 
	 $e$). Let $\bar{e}$ be a choice of edge in $F^{-1}(e)$. We define $\bar{p}$ to the path that traverses $\bar{e}$.  Clearly $\bar p$ and $p$ have the same label.  

	\textit{Type 2:} 
	In this case $p$ consists of an edge $e_1$ from $u$ to $v$, followed by a collection of graph-loops $q_1, \dots, q_k$ based at $v$, followed by an edge $e_2$ from $v$ to $u'$.

	Let $\bar{u}$ and $\bar{u}'$ be the unique preimages of ${u}$ and ${u}'$. Let $\bar{e}_1$ and $\bar{e}_2$ be  edges (not necessarily unique) in the preimage of $e_1$ and $e_2$ respectively. Let $\bar{q}_1, \dots, \bar{q}_k$  each be a choice of edge in the preimages of $q_1, \dots, q_k$. For each $1 \le i \le k$, the edge $\bar q_i$ is either a graph-loop at $\bar v_1$, a graph-loop at $\bar v_2$ or an edge between $\bar v_1$ and $\bar v_2$. 

	\begin{figure}[htp]
	\centering
		\vspace{0.2in}
	\begin{overpic}[scale=1.3]{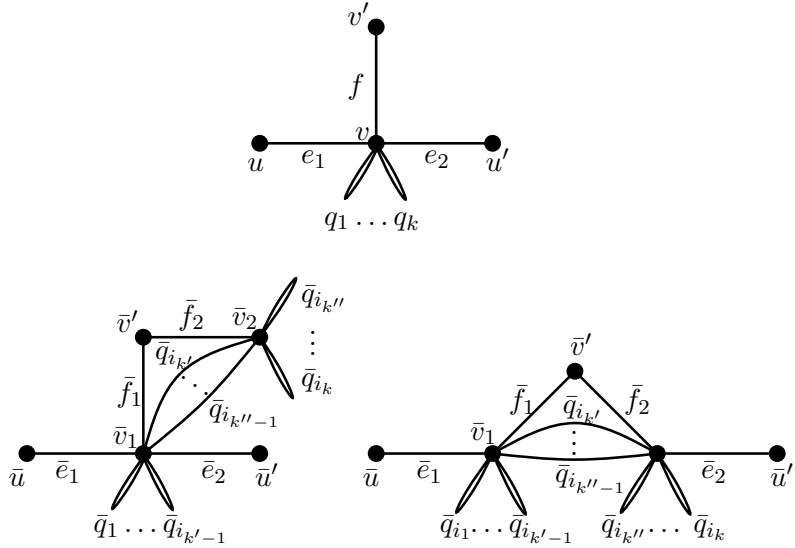}
		\put(30,45){$u$}
		\put(44,49){$v$}
		\put(61,45){$u'$}
		\put(37,46){$e_1$}
		\put(53,46){$e_2$}
		\put(43,55){$f$}
		\put(40,38){$q_1$}
		\put(49,38){$q_k$}
		\put(44,38){$\dots$}
		\put(43, 64){$v'$}

		\put(-1,4){$\bar{u}$}
		\put(12.5,9.5){$\bar{v}_1$}
		\put(31,4){$\bar{u}'$}
		\put(5,5){$\bar{e}_1$}
		\put(24,5){$\bar{e}_2$}
		\put(13,15){$\bar{f}_1$}
		\put(21,25){$\bar{f}_2$}
		\put(10,-2){$\bar{q}_1$}
		\put(19,-2){$\bar{q}_{i_{k' - 1}}$}
		\put(14,-2){$\dots$}
		\put(13, 24){$\bar{v}'$}
		\put(18,20.5){$\bar{q}_{i_{k'}}$}
		\put(21, 15.5){$\ddots$}
		\put(25,13){$\bar{q}_{i_{k'' - 1}}$}
		\put(37,28){$\bar{q}_{i_{k''}}$}
		\put(37,17){$\bar{q}_{i_k}$}
		\put(38,21){$\vdots$}
		\put(28, 25){$\bar{v}_2$}
		
		\put(45,4){$\bar{u}$}
		\put(59,10){$\bar{v}_1$}
		\put(52,5){$\bar{e}_1$}
		\put(64,14){$\bar{f}_1$}
		\put(79,14){$\bar{f}_2$}
		\put(72,21){$\bar{v}'$}
		\put(55,-2){$\bar{q}_{i_1}$}
		\put(64,-2){$\bar{q}_{i_{k' - 1}}$}
		\put(59,-2){$\dots$}
		\put(89,5){$\bar{e}_2$}
		\put(98,4){$\bar{u}'$}
		\put(71, 13.5){$\bar{q}_{i_{k'} }$}
		\put(72, 8){$\vdots$}
		\put(70, 4.5){$\bar{q}_{i_{k'' - 1} }$}
		\put(76,-2){$\bar{q}_{i_{k''} }$}
		\put(88,-2){$\bar{q}_{i_k}$}
		\put(82,-2){$\dots$}
	\end{overpic}
		\vspace{0.2in}
	\caption{The graph on the top shows a path of type 2 in $\Omega$. The two graphs below it show the two possible choices of preimages for $p$. The graph-loops $\bar{q}_{i_1}, \dots, \bar{q}_{i_{k'}}$ form a subsequence of the graph-loops $\bar{q}_1, \dots, \bar{q}_k$ consisting of those 
	based at $\bar{v}_1$. Similarly, $\bar{q}_{i_{k'}}, \dots, \bar{q}_{i_{k'' - 1}}$ are edges between $\bar{v}_1$ and $\bar{v}_2$ and $\bar{q}_{i{k''}}, \dots, \bar{q}_{i_k}$ are graph-loops based at $\bar{v}_2$. Note that some of the vertices and edges shown may actually be equal $\bar{\Omega}$. For instance, it could be that $\bar{f}_1 = \bar{e}_1$.}
\end{figure}	
	Let $\bar{z}$ be the path from $\bar{v}_1$ to $\bar{v}_2$ obtained by traversing $\bar{f}_1$ and then $\bar{f}_2$. Note that the label of $\bar{z}$ is equal to the identity element of $W_{\Gamma}$ as $\bar{f}_1$ and $\bar{f}_2$ have the same label. Form the path 
	\[\bar{p} = \bar{e}_1 \bar{p}_0 \bar{q}_1 \bar{p}_1 \bar{q}_2 \bar{p}_2 \dots \bar{q}_{k} \bar{p}_k \bar{e}_2, \]
	where for $1 \le i < k$, we define $\bar{p}_i$ to either be $\bar{z}$, $\bar{z}^{-1}$ or the empty word in order for the endpoint of $\bar{q}_i$ to coincide with the startpoint of $\bar{q}_{i+1}$. The paths $\bar{p}_0$ and $\bar{p}_k$ are defined similarly in order for the endpoint of $\bar{e}_1$ to coincide with the startpoint of $\bar{q}_1$, and in order for the endpoint of $\bar{q}_k$ to coincide with the startpoint of $\bar{e}_2$.
		The claim then follows for this case as the label of $\bar{p}$ is equal as an element of $W_{\G}$ to the label of $p$.

	\textit{Type 5:} In this case $v = B$ and $l$ consists of a sequence of graph-loops $q_1, \dots, q_k$. Let $\bar{q}_1, \dots, \bar{q}_k$ be a choice of edges in the preimages of $q_1, \dots, q_n$. As above, these preimages consist of graph-loops at $\bar{v}_1$ or  $\bar{v}_2$, and edges between $\bar{v}_1$ and $\bar{v}_2$. We may define the path
	\[ \bar{p} = \bar{p}_0 \bar{q}_1 \bar{p}_1 \dots \bar{q}_k \bar{p}_k \]
	where $\bar{p}_i$, for $1 \le i < k$, is defined similarly as in the previous case. Note that $\bar{B}$ is either equal to $\bar{v}_1$ or $\bar{v}_2$. We then define $\bar{p}_0$ to either be $\bar{z} = \bar{f}_1\bar{f}_2$, $\bar{z}^{-1}$ or the empty path in order to guarantee $\bar{p}$ begins at $\bar{B}$. Similarly define $\bar{p}_k$ to guarantee that $\bar{p}$ ends at $\bar{B}$.
	
	\textit{Type 3 and 4:} 
	The proofs in these cases are very similar to those above, and are omitted.

	\medskip \noindent	\textbf{Cube attachment operation:}
	Suppose $\Omega$ is obtained by attaching
	a cube $c$ to $e_1 \cup  \dots \cup e_n \subset \bar\Omega$,  where for $1 \le i \le n$, $e_i$ is an edge labeled $s_i$ from a vertex $v$ to a vertex $v_i$ (with $v=v_i$ if $e_i$ is a graph-loop), and the vertices of $\G$ corresponding to $s_1, \dots,  s_n$ form an $n$-clique. 
	
	If $l$ does not intersect $c \setminus \{e_1, \dots, e_n\}$ then $l$ is clearly the image of a loop in $\bar \Omega$ with the same label. Otherwise let $q$ be the closure of a maximal connected subpath of $l$ that is contained in $c \setminus \{e_1, \dots, e_n\}$.  In particular, $q$ is a path between $v_k$ and $v_{k'}$ for some (not necessarily distinct) $k, k'$. 
	Let $h$ be the label of $q$. Since the vertices $s_1, \dots , s_n$ form a clique in $\G$, there is a reduced expression for $h$ given by $h' = s_1^{\epsilon_1}\dots s_n^{\epsilon_n}$ where for each $1 \le i \le n$, $\epsilon_i = 1$ if there is an odd number of occurrences of the generator $s_i$ in $h$ and $\epsilon_i = 0$ otherwise.
	 
	First assume $k \neq k'$.  After renaming if necessary, we may assume $k = 1$ and $k'=n$. 
	We claim that if 
	$1< j <n$ and $\epsilon_j =1$, then $e_j$ is a graph-loop.  To prove this, note that if $1< j < n$ and $e_j$ is not a graph-loop,  then 
	$v_1$ and $v_n$ are on the same side of 	the midcube of $c$ dual to $e_j$.  Thus $q$ crosses this midcube an even number of times, and therefore $\epsilon_j = 0$. 

	Consider  the union of edges $q'= e_1^{\epsilon_1} \cup e_2^{\epsilon_2} \cup \cdots \cup e_n^{\epsilon_n}$, where $e_i^{\epsilon_i}$ is interpreted as empty if $\epsilon_i =0$.  The claim in the previous paragraph implies that this 
 	is in fact a path $q'$ with the same endpoints as $q$ (even if $e_1$ and $e_n$ are graph-loops).  Observe that the label of $q'$ is $h'$.  
		
	By a similar argument, if $k=k'=1$ and $e_1$ is a graph-loop, then 
	$q'= e_1^{\epsilon_1} \cup e_2^{\epsilon_2} \cup \cdots \cup e_n^{\epsilon_n}$ is a path with label $h'$ and the same endpoints as $q$.  Finally, suppose $k=k'=1$ and $e_1$ is not a graph-loop.  As before, if $j \neq 1$ and $\epsilon_j =1$ then $e_j$ is a graph-loop.   Moreover,  $\epsilon_1 = 0$, because $q$ crosses the mid-cube dual to $e_1$ an even number of times.  Thus in this case we define $q'$ to be the concatenation $e_1  e_2^{\epsilon_2} \cup \cdots \cup e_n^{\epsilon_n} e_1$, and note that this is a continuous path with the same endpoints as $q$, and with label $s_1 h' s_1$, which is equal in $W_\G$ to $h'$.  
		
	In each case we have produced a path $q'$ in $F(\bar\Omega)$ with the same endpoints as $q$ and whose label is a word equal in $W_\G$ to $h$.  We replace $q$ with $q'$ in $l$. 
	By performing all possible replacements of this sort, we obtain a loop in $F(\bar \Omega)$ 
	whose label is a word equal to $w$ in $W_{\Gamma}$. Thus the lemma follows for this case.
\end{proof}

\begin{lemma} \label{lemma_omega_loops}
	Let $G$ be a subgroup of a RACG $W_\G$ given by a finite generating set $S_G$. Let $\Omega$ be any completion of $X(S_G)$ where $(X(S_G), B)$ is the $S_G$-complex. Given a loop in $\Omega$ based at $B$, its label is a word representing an element of $G$.
\end{lemma}
\begin{proof}
	Let $X(S_G) = \Omega_0 \xrightarrow{f_0} \Omega_1 \xrightarrow{f_1} \Omega_2 \xrightarrow{f_2} \cdots \to \Omega$
	be a completion sequence for $X(S_G)$. Let $B$ denote the basepoint of $X({S_G})$ as well as all its images in this sequence. 
	
	Consider a loop $l$ based at $B$ in $\Omega$, with label $w$. 
Then there exists $n$ such that $l = \hat{f}(l')$ for some loop $l'$ based at $B$ in $\Omega_n$
which also has label $w$, where $\hat{f}: \Omega_n \to \Omega$ is the natural map. By iteratively applying Lemma~\ref{lemma_pull_back} starting with $\hat l$, it follows there is a loop in $\Omega_0$ based at $B$ whose label is a word equal to $w$ in $W_{\Gamma}$. As the label of any loop based at $B$ in $\Omega_0$ represents an element of  $G$, the lemma follows.	
\end{proof}

The existence of completions is an immediate consequence of Lemma~\ref{lemma_omega_reduced_words} and Lemma~\ref{lemma_omega_loops}:

\begin{theorem} \label{thm_subgroup_completion}
	Let $G$ be a subgroup of a RACG $W_\G$ given by a finite generating set $S_G$. 
	Then any completion $(\Omega, B)$ of the based $S_G$-complex  $(X(S_G), B)$ is 
	a completion of $G$.
	\hfill{\qed}
\end{theorem}

\begin{definition}[Standard completion of a subgroup]
	Let $G$ be a subgroup of the RACG $W_{\G}$ generated by a finite generating set $S_G$. We call a completion $\Omega$ of $G$ obtained by 
	Theorem~\ref{thm_subgroup_completion} a \textit{standard completion of $G$ with respect to $S_G$}. In cases where it is understood that  there is a finite generating set for $G$, we simply say $\Omega$ is a \textit{standard completion of $G$}.
\end{definition}

\begin{example}
The $\G_1$-labeled cube complex $\Omega$ in Example~\ref{ex_completion_ex1} is a standard completion of the subgroup $\langle adb, aec\rangle $
of $W_{\G_1}$. Similarly, the $\G_2$-labeled cube complex $\Omega'$ in Example~\ref{ex_completion_ex2} is a standard completion of the subgroup 
$\langle abcd \rangle$
of $W_{\G_2}$.
\end{example}

\begin{remark}
	Although every reduced word in $W_\G$ representing an element of $G$ labels a loop in the completion $\Omega$, it is not true that every word in $W_\G$ representing an element of $G$ labels a loop in $\Omega$. For instance, let $s \in V(\Gamma)$ and suppose  
	$G < W_\G$ is generated by a set of words, 
	none of which contains the letter $s$. It follows that no edge in $\Omega$ is labeled $s$. Then the word $ss$, which is equal to the identity element, cannot be the label of any path in~$\Omega$.
\end{remark}

\section{Basic properties of completions} \label{sec_completion_properties}

We prove a few facts regarding completions that will be used throughout the rest of the paper. Recall from Section~\ref{subsec_word_prob} that a deletion performed to a word $w$ in a RACG produces an expression for $w$ with a pair of generators of the same type removed.

\begin{lemma} \label{lemma_hausdorff_bound}
	Let $\Omega$ be a folded, cube-full, $\G$-labeled complex. Let $p$ be a path in $\Omega$ with label $w$. Let $w'$ be  an expression for $w$ obtained by performing $k$ deletions to~$w$. Then there exists a path $p'$ in $\Omega$ such that the following properties hold. 
	\begin{enumerate} 
		\item The path $p'$ has label $w'$. 
		\item The paths $p$ and $p'$ have the same endpoints. 
		\item The Hausdorff distance  between $p$ and $p'$  is at most $k$.
		\item If $p$ does not traverse any graph-loops, then $p$ and $p'$ are homotopic relative to their endpoints.
	\end{enumerate}
\end{lemma}
\begin{proof}
Let $w_1$ be the word obtained by performing the first deletion to $w$.  If $w = s_1 \dots s_n$, with $s_i \in V(\Gamma)$, then $w_1 = s_1 \dots s_{i-1}s_{i+1} \dots s_{i'-1}s_{i'+1}\dots s_n$ where $s_i = s_{i'} = s$ for some $1 \le i < i' \le n$ and $m(s, s_j) = 2$ for all $i < j < i'$.  
	Let $\alpha$ be the subpath of $p$ labeled by $s_i s_{i+1} \dots s_{i'}$.
	
	Suppose first that $i' - i > 1$.
	As $\Omega$ is cube-full, there exists a sequence of squares in $\Omega$ such that Figure~\ref{fig_hausdorff} holds (although there may be additional edge or vertex identifications that are not shown).
		The subpath $\alpha$ of $p$, which runs along the bottom of Figure~\ref{fig_hausdorff}, can be replaced with the path which runs along the top of Figure~\ref{fig_hausdorff}, to obtain a new path $p_1$ with label $w_1$. Then $p_1$ is homotopic relative to endpoints to $p$ and is at Hausdorff distance at most $1$ from~$p$.  

\begin{figure}[h!]
		\centering
		\begin{overpic}[]{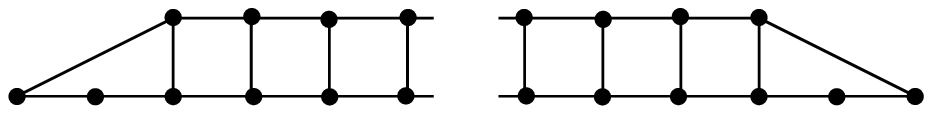}
			\put(-1,-3){\footnotesize$s = s_i$}
			\put(10,-3){\footnotesize $s_{i+1}$}
			\put(19,-3){\footnotesize $s_{i+2}$}
			\put(27.5,-3){\footnotesize$s_{i+3}$}
			
			\put(48, 6){$\dots$}
			
			\put(73.5,-3){\footnotesize$s_{i'-2}$}
			\put(83,-3){\footnotesize$s_{i'-1}$}
			\put(91,-3){\footnotesize$s_{i'} = s$}

			\put(5, 7){\footnotesize$s_{i+1}$}
			\put(19,11.5){\footnotesize$s_{i+2}$}
			\put(27.5,11.5){\footnotesize$s_{i+3}$}
			
			\put(73.5,11.5){\footnotesize$s_{i'-2}$}
			\put(90, 7){\footnotesize$s_{i'-1}$}
			
			\put(18.5, 5){\footnotesize$s$}
			\put(27, 5){\footnotesize$s$}
			\put(35.5, 5){\footnotesize$s$}
			\put(44, 5){\footnotesize$s$}
			\put(57, 5){\footnotesize$s$}
			\put(65.5, 5){\footnotesize$s$}
			\put(74, 5){\footnotesize$s$}
			\put(82.5, 5){\footnotesize$s$}
			
		\end{overpic}
		\caption{}
		 \label{fig_hausdorff}
	\end{figure}


	On the other hand, if $i' - i = 1$, then as $\Omega$ is folded, $\alpha$ either  traverses an edge labeled by $s$ twice in opposite directions, or a graph-loop labeled by~$s$ twice consecutively. In either case we can simply remove both occurrences of $\alpha$ from $p$ to obtain a new path $p_1$ with the same endpoints as $p$, such that $p_1$  is labeled by $w_1$ and is at Hausdorff distance at most $1$ from $p$. 
	If $p$ does not traverse any graph-loops, then $p_1$ is homotopic relative to endpoints to $p$.  
	
	Finally, if $p$ does not traverse any graph-loops, then $p_1$ does not either.  This is clear when $i' - i = 1$.  Now suppose $i' - i > 1$, and suppose $p$ has no graph-loops, but $p_1$ has a graph-loop $e$.  Then $e$ is an edge of one of the squares in Figure~\ref{fig_hausdorff}, 
and since $\Omega$ is folded, the two edges of the square incident to $e$ are identified in $\Omega$.  It follows that the edge opposite to $e$ in the square, which is a part of $p$, is also a graph-loop.  This is a contradiction.

By repeating this process of obtaining $p_1$ from $p$ inductively, we obtain the result.
\end{proof}

\begin{lemma} \label{lemma_omega_paths}
	Let $\Omega$ be a folded, cube-full, $\G$-labeled complex. Let $p$ be a path in $\Omega$ with endpoints the vertices $u$ and $v$, and let $w$ be its label.
	\begin{enumerate}
		\item \label{lemma_omega_paths1} Any reduced word $w'$ equal to $w$ in $W_{\Gamma}$ is the label of some path $p'$ in $\Omega$ from $u$ to $v$. Furthermore, if $p$ does not traverse any graph-loops, then $p$ and $p'$ are homotopic relative to their endpoints.
		\item \label{lemma_omega_paths2} If $p$ has minimal length then $w$ is reduced.
		\item \label{lemma_omega_paths3} Let $p$ and $p'$ be paths in $\Omega$ with the same endpoints. If $p$ and $p'$ are homotopic relative to their
		endpoints, then their labels are equal as elements of  $W_{\Gamma}$.
	\end{enumerate}
\end{lemma}
\begin{proof}
		Let ${w''}$ be a reduced word equal to $w$ in $W_{\Gamma}$ which is obtained by a sequence of deletions. By Lemma~\ref{lemma_hausdorff_bound}, there is a path $p''$ with label $w''$ and the same endpoints as $p$. Furthermore, if $p$ does not traverse any 
		graph-loops, then $p''$ is homotopic relative endpoints to $p$
		
		By Tits' solution to the word problem, there is a sequence of words $w'' = w_0, w_1, \dots , w_n = w'$ so that $w_{i+1}$ is obtained from $w_i$ by swapping a pair of consecutive 
		commuting generators.
		
		Suppose that $w_0 = a_1\dots a_i st a_{i+1} \dots a_n$ and $w_1 =  a_1\dots a_i ts a_{i+1}$, where $s, t \in V(\Gamma)$, with $m(s,t) = 2$, and $a_j \in V(\Gamma)$ for each $j$. 		
		It follows there are consecutive edges of $p''$ with labels $s$ and $t$. As $\Omega$ is cube-full, these edges are in a square with label $stst$. By replacing these two edges of $p''$ with the other edges in this square, we obtain a new path whose label corresponds to swapping $s$ and $t$. Furthermore, this new path is homotopic, relative to endpoints, to $p''$. By iteratively repeating this process, we obtain the desired path $p'$.
		This proves (1).
		
		Let $p$ be a path with minimal length and label $w$. If $w$ is not reduced, let $w'$ be a reduced expression for $w$. By (\ref{lemma_omega_paths1}) there is a path $p'$ having the same endpoints as $p$, with label $w'$. However, $|p'| < |p|$, a contradiction.  This proves (2).
		
		Let $p$ and $p'$ be as in (3).	As the concatenation $pp'^{-1}$ is null-homotopic, there exists a disk diagram $D$ with boundary $p p'^{-1}$. It readily follows that a homotopy of $p$ to $p'$ in $D$ induces a sequence of Tits moves which show that $p = p'$.  This proves~(3).
\end{proof}

The next definition allows us to ``go backwards''  from a $\G$-labeled complex to a subgroup.

\begin{definition}[Associated subgroup] \label{def_associated_subgroup} 
	Let $\Omega$ be a connected, 
		$\G$-labeled complex with base vertex $B$.  Consider the set of all $g \in W_\G$ such that there exists a loop in $\Omega$ based at $B$ whose label is a word in $W_\G$ that represents of $g$.  This set is easily seen to be a subgroup of $W_\G$, and is called the \emph{subgroup of $W_\G$ associated to $(\Omega, B)$}.
\end{definition}

\begin{proposition} \label{prop_subgroup_gen_by_loops} 
	Let $(\Omega, B)$ be a connected, folded, cube-full, $\G$-labeled complex, and let $G$ be the subgroup of $W_\G$ associated to $(\Omega, B)$. Then $(\Omega, B)$ is a completion of $G$.
\end{proposition}
\begin{proof}
Properties (\ref{def_completion_folded_cube_full}) and 
	(\ref{def_completion_loops}) in the definition of a completion of a subgroup (Definition~\ref{def_completion})  are immediate.  To check property (\ref{def_completion_reduced_words}), let $w$ be a reduced word representing an element $g$ of $G$.  By the definition of $G$, there exists a loop in $\Omega$ based at $B$ whose label, say $w'$, is a representative of $g$.  Since $w$ is a reduced representative of $w'$,
	Lemma~\ref{lemma_omega_paths} (\ref{lemma_omega_paths1}) implies that there is a loop in $\Omega$ based at $B$ with label $w$.  	
	\end{proof}

The following lemma describes the effect of changing the basepoint in a $\G$-labeled complex on the associated subgroup of $W_\G$.  

\begin{lemma} \label{lem_basepoint_change}
	Let $\Omega$ be a connected, folded, cube-full, $\G$-labeled complex.  Let $B_1$ and $B_2$ be vertices of $\Omega$, and for $i=1, 2$, let $G_i$ be the subgroup of $W_\G$ associated to $(\Omega, B_i)$.  
	Then $G_2 = w^{-1}G_1 w$, where $w$ is the label of some path from $B_1$ to $B_2$.    
\end{lemma}

\begin{proof}
	Let $\alpha$ be a path from $B_1$ to $B_2$ with label $w$.  If $\beta$ is a loop in $\Omega$ based at $B_1$ representing an element of $G_1$, then the concatenation $\alpha^{-1} \beta \alpha$ represents an element of $G_2$.  It follows that $w^{-1} G_1 w \subseteq G_2$.  Similarly $w G_2 w^{-1} \subseteq G_1$.
\end{proof}

It is easy to detect torsion in subgroups of RACGs using completions:
\begin{proposition} \label{prop_torsion_free}
	Let $G$ be a subgroup of a RACG $W_\G$ and let $(\Omega, B)$ be a completion for $G$.  Then $G$ has torsion if and only if there exists a loop in $\Omega$ (not necessarily passing through $B$) whose label is a reduced word representing an element in a finite special subgroup of $W_\G$. 
\end{proposition}

\begin{proof}
	If $g\in G$ has finite order, then $g$ is conjugate into a finite special subgroup of $W_\G$ (see \cite[Theorem 12.3.4]{Davis} for instance).  Write $g=w u w^{-1}$, where 
	$u$ and $w$ are reduced and $w$ is the shortest word for which such an expression for $g$ exists.  
	We claim 
	$wuw^{-1}$ is reduced.  If not, then a deletion is possible.  It follows from our choices that some letter, 
	say $s$, occurring in $w$ or $w^{-1}$ cancels with an occurrence of $s$ in $u$.  Since $u$ belongs to a finite special subgroup, $s$ commutes with $u$.  Thus we can write $g = w_1 u w_1^{-1}$, where $w = w_1 s$, a contradiction.

	Since $wuw^{-1}$ is reduced, there is a loop $\alpha$  in $\Omega$ based at $B$ with label $wuw^{-1}$.  
	Let $v$ be the vertex along $\alpha$ such that the label of $\alpha $ between $B$ and $v$ is $w$. 
	As $\Omega$ is folded, 
		the subpaths of $\alpha$ with labels $w$ and $w^{-1}$ are identified, and
there is a loop
	 based at $v$ with label $u$. This proves one direction of the claim.

	For the other direction, suppose that there is a loop based at some vertex $x$ of $\Omega$ with label a reduced word $r$ representing an element in a finite special subgroup of $W_\G$. As finite special subgroups of $W_\G$ correspond to clique subgraphs of $\G$, it follows that the support of $r$ is contained in a clique of $\G$. Let $s \in V(\G)$ be a letter in $r$. As $r$ is reduced, there is exactly one occurrence of $s$ in $r$. Let $h$ be the label of a path from $B$ to $x$ in $\Omega$. It follows that there is a loop in $\Omega$ based at $B$ with label $k = hrh^{-1}$. Furthermore, as there are an odd number of occurrences of the letter $s$ in the word $k$, 
     the deletion property implies that $k$ cannot be an expression for the identity element. As $k$ has finite order, it follows that $G$ has torsion.
\end{proof}

For torsion-free subgroups, the following holds: 

\begin{theorem} \label{thm_fundamental_group}
	Suppose $G$ is a torsion-free subgroup of the RACG $W_{\Gamma}$. Then the fundamental group of any
	completion $\Omega$ of $G$ is isomorphic to~$G$. 
\end{theorem}
\begin{proof}
	Let $B$ be a vertex of $\Omega$.  
	We  define the isomorphism $\phi: \pi_1(\Omega, B) \to G$ as follows. Let $\alpha$ be a loop in $\Omega$ based at $B$ representing an element of $ \pi_1(\Omega, B) $.
	 We may assume that $\alpha$ is contained in the $1$-skeleton of $\Omega$.
By property (\ref{def_completion_loops}) of the definition of a completion, the label of $\alpha$ represents an element of $G$, and we define $\phi(\alpha)$ to be this element.  
	To see that $\phi$ is well-defined, let $\alpha$ and $\alpha'$ be loops based at $B \in \Omega$ that are homotopic relative basepoint.  Then by Lemma~\ref{lemma_omega_paths}(\ref{lemma_omega_paths3}), the labels of $\alpha$ and $\alpha'$ are equal as elements of $G$. 
	
	It is clear that $\phi$ is a surjective	homomorphism.
	To  check that $\phi$ is injective, suppose that $\phi(\alpha)$ is a word in $W_\G$ equal in $G$ to the identity element. Note that $\Omega$ cannot contain a graph-loop by Proposition~\ref{prop_torsion_free}. Thus, by
Lemma~\ref{lemma_omega_paths}(\ref{lemma_omega_paths1}) we conclude that $\alpha$ is null-homotopic.
\end{proof}

The next result, which in particular implies that any standard completion of a finite $\G$-labeled complex has finitely many hyperplanes, is used in several proofs in this article. 

\begin{proposition} \label{prop_omega_hyps}
	Let $X$ be a $\G$-labeled complex. Let 
	\[X = \Omega_0 \to \Omega_1 \to \cdots \to \Omega\] 
	be a completion sequence for $X$. Then, for all $i \ge 0$, every hyperplane in $\Omega_i$ intersects the image of $X$ in $\Omega_i$. Consequently, every	hyperplane in $\Omega$ intersects the image of $X$ in $\Omega$. 
\end{proposition}
\begin{proof}
	We prove the first claim 
	by induction on $n$. The base case for $\Omega_0$ is trivially true. 
	Suppose every hyperplane in $\Omega_{n-1}$ intersects the image of $X$. 
	It is clear that cube identifications and folds do not produce new hyperplanes. Now suppose that $\Omega_n$ is obtained by attaching a $k$-cube $c$ 
	along edges $e_1, \dots, e_k$, all incident to a common vertex of $\Omega_{n-1}$. Since each midcube of $c$ extends a hyperplane dual to one of the $e_i$'s, it follows that no new hyperplanes are created 
in~$\Omega_n$.   Hence, the claim also holds for $\Omega_n$. 
	The claim follows for the completion $\Omega$ as any hyperplane in $\Omega$ contains the image of some hyperplane in $\Omega_n$ for some~$n$.
\end{proof}

In a few arguments throughout this article, we will have a finite, cube-full, folded $\G$-labeled complex, and we will want to attach certain additional graph-loops to this complex. We will then need the original complex to be isometrically embedded in the completion of the new complex. The following lemma guarantees this property.

\begin{lemma} \label{lemma_adding_graph_loops}
	Let $\Omega$ be a finite, cube-full, folded $\G$-labeled complex. Let $\Omega'$ be a complex obtained by attaching a set $\mathcal{L}$ of labeled graph-loops to vertices of $\Omega$. Suppose that the label of an attached graph-loop is distinct from the labels of every other edge  incident to the vertex it is attached to, i.e., $\Omega'$ is a folded complex. Then there exists a completion $\Omega''$ of $\Omega'$ such that
	\begin{enumerate}
		\item \label{lemma_adding_graph_loops_isometry} The natural inclusion $i: \Omega \xhookrightarrow{} \Omega''$ is an isometry.
		\item \label{lemma_adding_graph_loops_paths} Every edge of $\Omega''$ that is not in $i(\Omega)$ is a graph-loop attached to a vertex $v \in i(\Omega)$. Let $l$ be such a graph-loop 
		and let $s$ be its label. Then there exists a graph-loop in $\mathcal{L}$ with label $s$ attached to a vertex $u \in \Omega$ and a path 
 in $i(\Omega)$ 
		from $i(u)$ to $v$ whose label 
is a word in $\text{link}(s) \subset V(\G)$. 
		\item \label{lemma_adding_graph_loops_bound} The number of operations performed to obtain $\Omega''$ from $\Omega'$ is finite and only depends on the number of edges of $\Omega'$.
	\end{enumerate}
\end{lemma}
\begin{proof}
	We build the completion $\Omega''$ by alternately performing a single cube attachment  operation followed by all possible fold and cube identification operations. 
%
	Since $\Omega'$ is folded by assumption, 
	each cube attachment operation is only done to a folded complex. We show that each folded complex in this completion sequence satisfies the conclusion of the lemma. 
	
	Let $v$ be a vertex of $\Omega'$  incident to edges $e_1, \dots, e_n$ with distinct pairwise commuting labels $s_1, \dots , s_n$, such that 
	$v \cup e_1 \cup \dots \cup e_n$ are not  contained in a common $n$-cube and that $n$ is maximal out of such possible choices. Consider the 
	operation which attaches a labeled cube to $v \cup e_1 \cup \dots \cup e_n$. Let $c$ denote the image of this cube in the resulting complex.

	By possibly relabeling, we may assume that there exists $0 \le k \le n$ such that 
	if $i \le k$, then $e_i$ is not a graph-loop 
	(and is therefore necessarily in $\Omega$), while if $i > k $, then $e_i$ is a graph-loop (and may or may not be in $\Omega$).  Note that if $k = 0$, then $e_i$ is a graph-loop for all $i$.

	If $k>0$, it follows (since $\Omega$ is cube-full) that 
$e_1, \dots, e_k$ are contained in a common $k$-cube $q$ of $\Omega$. We perform fold and cube identification operations to identify $q$ to a face of $c$.  
Otherwise if $k=0$, define $q$ to be the $0$-cube $v$.
	
	Next, starting at $v$, we perform all possible fold operations to pairs of edges which are both in $c$. It readily follows that the $1$-skeleton of this resulting complex will consist of $q$ 
	 and a graph-loop with label $s_i$, for each $k+1 \le i \le n$ and each vertex of $q$. By a slight abuse of notation, we call this resulting complex $c$ as well.
	
	We now check what other fold operations are possible. As $\Omega'$ is folded, the only type of possible additional fold operation would consist of an edge $e$ in $\Omega'$ and a graph-loop $f$ in $c$ such that $e$ and $f$ have the same label, $s$, and share an endpoint 
$u \in q$.
	
	We claim that $e$ must be a graph-loop. This is clear if $e$ is in $\Omega' \setminus \Omega$. 
	Suppose $e \in \Omega$. There is a path $p$ from $u$ to $v$ in $q$	with label a word consisting only of generators that are distinct from and commute with $s$. Let $e_i$ be the edge at $v$ with label $s$. Then $e_i$ must be a graph-loop since otherwise $e$ would have already been folded onto $c$. As $\Omega$ is cube-full and $p\cup e \subset \Omega$, we conclude that $e_i \in \Omega$.  Thus since $e_i$ is a graph-loop, $e$ is a graph-loop as well. 

	Thus, we simply fold $e$ onto $f$. 	After performing all such folds to $\Omega' \cup c$ we obtain the complex $\Omega'_1$. After possibly applying some cube identification operations, it follows that $\Omega'_1$ is folded. Furthermore, the $1$-skeleton of $\Omega'_1$ is the same as the $1$-skeleton $\Omega'$ with the possible addition of some new graph-loops. Thus, $\Omega$ is isometrically embedded in this new complex. The second conclusion of the lemma also readily follows from the construction.
	
	We then iteratively repeat such cube attachments followed by such a sequence of fold and cube identification operations. After each iteration we have a folded complex satisfying the first two claims of the lemma. As the number of such operations is bounded by a function of the number of edges of $\Omega'$, the third claim of the lemma follows.
\end{proof}

\section{Core graphs}\label{subsec_core}
In general, a subgroup $G$ of $W_\G$ does not have a unique completion.   
However, we now
use completions to define a certain graph associated to $G$ called a core, 
which is unique.  
This is used in Theorem~\ref{thm_normal}
to obtain a characterization for normality of a subgroup. 

\begin{definition}[Core graph]\label{def_core_graph}
Given a $\G$-labeled complex $(\Omega, B)$, define its \emph{core graph at $B$}, denoted $C(\Omega, B)$, to be the $1$-dimensional subcomplex  consisting of the union of all the loops in $\Omega$ based at $B$ whose labels are reduced words in $W_\G$.    
\end{definition}

The core graph  of a completion is not necessarily its entire $1$-skeleton.
For instance, 
in Figure~\ref{fig_completion_ex1}, 
the core graph is  the part of $\Omega^{(1)}$ (the 1-skeleton of $\Omega$)
that is shown in black.  That is, 
$C(\Omega, B)$ is $\Omega^{(1)}$ minus the open star of the vertex $v$ diametrically opposite to $B$. The paths based at $B$ with (reduced) labels $adabec, aeacdb$ and $adcbea$ show that $C(\Omega, B)$ contains this subcomplex.  On the other hand, let $v_a, v_b$ and $v_c$ denote the vertices adjacent to $B$ in $\Omega^{(1)}$. 
Any loop at $B$ which passes through $v$ contains a subpath which passes through~$v$, starts and ends in $\{v_a, v_b, v_c\}$, and, apart from its endpoints, remains in the component of $\Omega^{(1)} \setminus \{v_a, v_b, v_c\}$ containing $v$.  
We leave it to the reader to check that any such subpath $\gamma$ must cross some midcube twice. 
This, together with the fact that the label of $\gamma$ only uses the letters
$a, b,$ and $c$, which pairwise commute, implies that $\gamma$ cannot have 
reduced label.

The following proposition states that core graphs are unique in a certain sense.  We omit the proof as it very closely follows the proof of~\cite[Theorem 5.2]{KM}, which is 
the corresponding result in the free group setting.

\begin{proposition}\label{prop_unique_core}
Let $G$ be a subgroup of $W_\G$, and let $(\Omega_1, B_1)$ and $(\Omega_2, B_2)$ be completions of $G$. 
Then there is an isomorphism $f: C_1 (\Omega_1, B_1)\to C_2(\Omega_1, B_2)$, such that $f(B_1) = B_2$.  
\qed
\end{proposition}

We now characterize normal subgroups of RACGs in terms of core graphs. 
In Section~\ref{sec_other_algorithmic_properties}, we give a different 
characterization of normality  that is more useful for algorithmic applications.

\begin{theorem}\label{thm_normal} 
	Let $G$ be a subgroup of $W_\G$, with completion $(\Omega, B)$. Consider 
		\[\Delta = \{ s \in V(\G) \mid s \text{ commutes with every element of } G\}. \]
	 Then $G$ is normal if and only if the following conditions are satisfied.  
	\begin{enumerate}
		\item[(N1)] Given any $s \in V(\G) \setminus \Delta$, there is an edge in $\Omega$ incident  to $B$ with label $s$. 
		\item[(N2)] For every vertex $v$ of $\Omega$, there is an isomorphism 
		from $C(\Omega, B)$ to $C(\Omega, v)$ which takes  $B$ to $v$. 
	\end{enumerate}
\end{theorem}
\begin{proof}
	First suppose N1 and N2 are satisfied.  To show that $G$ is normal, it is enough to show that $sGs \subset G$ for all $s \in V(\G)$.  This is obvious when $s \in \Delta$, so consider $s\in V(\G) \setminus \Delta$.  By N1 there is an edge incident to $B$ with label $s$.  Let $v$ be its other endpoint.  Then by Lemma~\ref{lem_basepoint_change}, $(\Omega, v)$ is a completion for $sGs$.  
	Thus, given any element $g$ in $sGs$, there is a loop in $\Omega$ whose label is a reduced word representing $g$.  It follows that the group associated to $(C(\Omega, v), v)$ (in the sense of Definition~\ref{def_associated_subgroup}) is $sGs$. 
 On the other hand, since $C(\Omega, v) \cong C(\Omega, B)$ by N2, the group associated to it is $G$.  Thus $sGs = G$.
	
	Now suppose $G$ is normal.  	We first show N1 is satisfied. Let $s \in V(\G) \setminus \Delta$,
and let  $w$ be a reduced word representing an element of $G$ which does not commute with $s$.  
If $w$ has a reduced expression $w'$ which either begins or ends with $s$, then since $\Omega$ is a completion, there is a loop based at $B$ with label $w'$. It follows that there is an edge incident to $B$ in $\Omega$ labeled $s$.  
If no expression for $w$ begins  or ends with $s$, then $sws$ is reduced
	by Lemma~\ref{lemma_reduced_expression}.  Moreover, since $G$ is normal, $sws \in G$, and consequently there is a loop in $\Omega$ based at $B$ with label $sws$.  
	Once again, $B$ is incident to an edge labeled by $s$. Thus N1 holds in all cases.

	To prove N2, let $v \neq B$ be a vertex of $C(\Omega, B)$, and let $\alpha$ be a path in $C(\Omega, B)$ from $B$ to $v$ with label $w$.   
	By Lemma~\ref{lem_basepoint_change}, we know that $(\Omega, v)$ is a completion for $w^{-1}Gw=G$. 
	Then by Proposition~\ref{prop_unique_core}, there is an isomorphism from $C(\Omega, B) $ to $C(\Omega, v)$ which takes $B$ to~$v$.  
\end{proof}

\section{Index of a subgroup} \label{sec_index}	
The main result of this section is Theorem~\ref{thm_omega_finite_index}, which 
states that completions characterize the index of a subgroup 
	of a RACG. This is 
analogous to a 
	result for subgroups of free groups, first proved in~\cite{Arzhantseva-pams}.

\begin{definition}[Full valence] \label{def_full_valence} 
	We say a vertex $v$ of a $\G$-labeled complex has \textit{full valence} if for each $s \in V(\Gamma)$ there is an edge with label $s$ incident to $v$. We say a $\G$-labeled complex $\Omega$ is
	\textit{full valence} if every vertex of $\Omega$ has full valence.
\end{definition}
We would like to  be able to state that a subgroup has finite index if and only if all its completions are finite and have full valence,  and moreover,
that the number of vertices of a completion determines the index of a subgroup.
However, due to a slight technical issue, this is not quite true in general.
Namely, suppose that $\G$ contains a vertex $s$ that is adjacent to every other vertex. 
Then  $W_\G = W_{\Delta} \times \mathbb{Z}_2$, where $\Delta$ is the subgraph of $\G$ induced by $V(\Gamma) \setminus s$. We can then build a completion for $W_\Delta \subset W_\Gamma$ by using the generating set $V(\Delta)$. Such a completion has no edges labeled by $s$ and so does not have full valence. Furthermore, it contains only one vertex, and yet, the index of $W_\Delta$ is not one in $W_\Gamma$.
To remedy this, we define resolved completions below, for which the desired statements do hold (see Theorem~\ref{thm_omega_finite_index} and Lemma~\ref{lemma_index_of_subgroup}). 
Any finitely generated subgroup admits a resolved completion (see Remark~\ref{rmk_resolved}), so resolved completions do indeed remedy this technical issue.

\begin{definition}[Resolved Completion]
	Let $\Omega$ be a completion of a subgroup $G$ of $W_\G$. We say that $\Omega$ is \textit{a resolved completion} if given any $s \in \Gamma$ such that $ V(\Gamma)=\text{star}(s) $, it follows that some edge of $\Omega$ is labeled by $s$.
\end{definition}

\begin{lemma} \label{lemma_full_vertex}
	Let $G$ be a subgroup of $W_{\Gamma}$, and let $\Omega$ be a resolved completion of $G$. Let $s \in \Gamma$ be such that $V(\G) = \text{star}(s)$. Then every vertex of $\Omega$ is incident to an edge labeled by $s$.
\end{lemma}
\begin{proof}
	As $\Omega$ is resolved, let $e$ be an edge of $\Omega$ labeled by $s$. Let $v$ be a vertex of $\Omega$ incident to $e$. Let $u$ be any vertex adjacent to $v$, and let $t$ be the label of the edge $e'$ between $u$ and $v$. Either $t = s$ or $t$ is adjacent to $s$ in $\Gamma$. As $\Omega$ is cube-full, in the latter case there must be a square with label $stst$ in $\Omega$ that contains both the edge $e$ and $e'$. In either case $u$ is incident to an edge labeled by $s$ as well. Proceeding in this manner, since $\Omega$ is connected, we conclude every vertex in $\Omega$ is incident to an edge labeled by $s$.
\end{proof}

\begin{lemma} \label{lemma_full_valence}
	Let $G$ be a subgroup of $W_{\Gamma}$, and let $(\Omega, B)$ be a resolved completion of $G$. If $\Omega$ is not full valence, then $G$ has infinite index in $W_{\Gamma}$.  
\end{lemma}
\begin{proof}
	
Suppose there exists a vertex $v$ in $\Omega$ that is not incident to an edge labeled by $s$, for some $s \in V(\G)$.
Let $\alpha$ be a minimal length path in $\Omega$ from the base vertex $B$ to $v$, and let $w$ be the label of this path. We assume $|w|$ is minimal among the  possible choices for $w$ and $v$. By Lemma~\ref{lemma_omega_paths} (\ref{lemma_omega_paths2}), the word $w$ is reduced.

We begin by establishing a few facts, which will be used later in the proof:

\medskip
{\it (i) 	The word $ws$ is reduced.} 
If not, by the deletion property there exists a  
	 reduced word $w'$, ending with $s$, which is an expression for $w$.
	 By Lemma~\ref{lemma_omega_paths}(\ref{lemma_omega_paths1}) there is a path in $\Omega$ from $B$ to $v$ with label $w'$.  However, this is not possible as $v$ is not incident to an edge labeled by~$s$.
	
	\medskip

	{\it (ii) No reduced expression for $w$ ends in a generator that commutes with $s$.} 
	For suppose $w' = s_1 \dots s_n$, with $s_i \in \Gamma$, is a reduced word such that $s_n$ commutes with $s$ and $w$ is equal to $w'$ in $W_{\Gamma}$. Let $\alpha'$ be the path from $B$ to $v$ with label $w'$ and let $\hat{\alpha}$ be the subpath of $\alpha'$ 
	with label $s_1 \dots s_{n-1}$. Such paths exist by Lemma~\ref{lemma_omega_paths}(\ref{lemma_omega_paths1}). Let $\hat{v}$ be the endpoint of $\hat{\alpha}$. No edge incident to $\hat{v}$ is labeled by $s$. For if there were such an edge, the fact that 
	$\Omega$ is cube-full would imply that there is a square with label $ss_nss_n$ containing both $v$ and $\hat{v}$, contradicting the fact that $v$ is not incident to an edge labeled by $s$. However, it now follows that $\hat{v}$ is a vertex that is not incident to an edge labeled by $s$ and $|\hat{w}| < |w|$, contradicting the minimality of our choice of $w$. Thus no expression for $w$ can end with a generator that commutes with $s$.

\medskip
{\it (iii) No reduced word representing an element of $G$ begins with the label $ws$. }
	To see this, note that every reduced word representing an element of $G$ labels a loop in $\Omega$ based at $B$.  As $\Omega$ is folded, $\alpha$ is the only path beginning at $B$ with label $w$.  The claim follows, since the endpoint $v$ of $\alpha$ is not incident to an edge labeled by $s$.
	 
	\medskip
	We now proceed with the proof. 
	As $\Omega$ is resolved and by Lemma~\ref{lemma_full_vertex}, there exists a vertex $t$ of $\Gamma$ that is not adjacent to $s$ in $\Gamma$. 
	By Tits' solution to the word problem, it follows that $(st)^n$ is reduced for all integers $n \ge 1$. 
	Similarly, $w(st)^n$ is reduced for all integers $n \ge 1$, since $ws$ is reduced by {\it (ii)} above. 
	
	Suppose now, for a contradiction, that $G$ is a finite-index subgroup of $W_{\Gamma}$. In particular, as a set we have $W_{\Gamma} = Gg_1 \sqcup Gg_2 ... \sqcup Gg_n$ for finitely many elements $g_1, \dots, g_n \in W_\G$. Let $w_1, \dots, w_n$ be reduced words representing $g_1, \dots, g_n$. Let $M = \max\{|w_1|, ..., |w_n|\}$, and let $k = (st)^M$. Consider the word $h = wk$ which we know to be reduced. It follows that $h$ is equal to $h' h''$ in $W_{\Gamma}$, where $h'' = w_i$ for some $i$ and $h'$ is a reduced word in $G$. Form a disk diagram $D$ with boundary label $h(h'h'')^{-1} = wkh''^{-1}h'^{-1}$. Let $p_{w}$, $p_k$, $p_{h'}$ and $p_{h''}$ be the paths along the boundary of $D$ with labels respectively $w$, $k$, $h'$ and $h''$. Thus, $p_w p_k p_{h''}^{-1} p_{h'}^{-1}$ is a path tracing the boundary of $D$.

	As $h$ is reduced, every dual curve intersecting $p_wp_k$ must necessarily intersect either $p_{h'}$ or $p_{h''}$. Furthermore, a pair of dual curves which each intersect $p_k$ cannot intersect one another as $s$ and $t$ do not commute. Let $C$ be the dual curve intersecting the first edge of $p_k$. Note that $C$ is of type $s$. 
	%
	Now, $C$ cannot intersect $p_{h''}$. For if it did, every dual curve intersecting $p_k$ would intersect $p_{h''}$ as well. However, as $|h''| \le M$ and $|k| = 2M$, this is not possible. Thus, $C$ intersects $p_{h'}$. 
	
	Additionally, no dual curve intersecting $p_w$ intersects $C$. For suppose there is such a dual curve, and suppose that it is the furthest such dual curve along $p_w$. It follows that  the type of such a dual curve commutes with $s$ and commutes with every label of an edge appearing further along $p_w$.  However, this implies that $w$ has an expression ending with a generator that commutes with $s$, which contradicts {\it (ii)} above. Thus, every dual curve intersecting $p_w$ intersects $p_{h'}$ at an edge occurring before (in the orientation of $p_{h'}$)  the edge of $p_{h'}$ intersecting~$C$.
	
	By Lemma~\ref{lemma_disk_diagram_subwords}, there is a reduced word equal to ${h'}$ in $W_{\Gamma}$ 
	with prefix $ws$. 
	This contradicts {\it (iii)} above.  Note that the argument above also holds when $w$ is empty, i.e.~when $B=v$.
\end{proof}

\begin{lemma} \label{lemma_index_of_subgroup} 
	Let $G$ be a subgroup of $W_{\Gamma}$, and let $\Omega$ 
	be a completion of $G$ which is full valence. Then the index of $G$ in $W_{\Gamma}$ is equal to the number of vertices in $\Omega$ (which could be infinite).
\end{lemma}
\begin{proof}
	Let  $B = v_1, v_2, \dots$ be an enumeration of the vertices of $\Omega$ (where $B$ denotes the basepoint). For each $i$, choose a minimal length path $\alpha_i$ from $B$ to $v_i$ and let $w_i$ be its label. We will show that the words $w_1, w_2,...$ are expressions for right coset representatives for $G$.

	Let $w$ be a reduced word in $W_{\Gamma}$. As every vertex of $\Omega$ has full valence, there is a path $\alpha$ in $\Omega$ beginning at the vertex $B$ with label $w$.  Then, for some $i$, the concatenation $\alpha\alpha_i^{-1}$ is a loop based at $B$, and its label $w w_i^{-1}$ represents an element of $G$. Thus, $w$ can be represented by the coset $(ww_i^{-1})w_i$.
	
	Let $n$ be the number of vertices in $\Omega$ (where $n$ could be infinite).	We have shown that the index of $G$ is at most $n$. We now show it is exactly $n$. Suppose, to the contrary, that there exist words $h$ and $h'$ representing elements of $G$ such that $hw_i$ is an expression for $h'w_j$, for some $1 \le i < j \le n$. Then $w_iw_j^{-1}$ is an expression for an element of $G$. Now consider the path $\beta$ in $\Omega$ with initial vertex $B$ and label $w_iw_j^{-1}$. This path exists 
	as $\Omega$ is full valence. 
	
	We claim $\beta$ is a loop. For suppose not. Then $\beta$ ends in some vertex $v_k \neq B$ and it follows that $w_iw_j^{-1}w_k^{-1}$ is a loop. By the definition of a completion, the word $w_iw_j^{-1}w_k^{-1}$ represents an element of $G$. Consequently, as $w_iw_j^{-1}$ is a word representing an element of $G$, we conclude that  $w_k$ represents an element of $G$ as well. However, this contradicts $\Omega$ being a completion, as $w_k$ is a reduced word and is not the label of a loop in $\Omega$ based at $B$. Thus $\beta$ must be a loop. However, since $\beta$ is labeled by $w_iw_j^{-1}$, this implies that $v_i = v_j$, a contradiction.
\end{proof}

 Combining Lemma~\ref{lemma_full_valence} and Lemma~\ref{lemma_index_of_subgroup}, we obtain:

\begin{theorem} \label{thm_omega_finite_index}
	Let $G$ be a subgroup of $W_{\Gamma}$, and let $\Omega$ be a resolved completion of $G$. The subgroup $G$ has finite index in $W_{\Gamma}$ if and only if $\Omega$ is finite and full valence.
	\qed
\end{theorem}

\begin{remark} \label{rmk_resolved}
	For finitely generated subgroups, a resolved completion can always be constructed. 	For let $T = \{w_1, w_2 \dots w_n\}$ be a generating set of words for $G < W_\G$. Note that $G$ is still generated by 
$T' = T \cup \{s^2 ~| ~s \in \Gamma \text{ and }   V(\Gamma)= \text{star}(s) \}$, and furthermore,
	 every $s \in V(\Gamma)$ occurs in some word in $T'$. Thus, 
	a standard completion for $G$  built using $T'$ is necessarily resolved. 
	\end{remark}

\begin{example}\label{ex:fi_completion}
Let $\G$ be the graph in Figure~\ref{fig:fi_completion}, and let $G = \langle ca, cb\rangle < W_\G$.  The right of this 
figure shows a standard completion $\Omega$ for $G$. 
Since $\G$ does not have a vertex adjacent to every other vertex, $\Omega$ is automatically resolved. Note that 
 $\Omega$ is finite and full-valence. 
 Thus, $G$ has finite index in $W_\G$ by Theorem~\ref{thm_omega_finite_index}.
	
\bigskip

\begin{figure}[h!]
		\centering
		\begin{overpic}
			[scale=1.3]
			{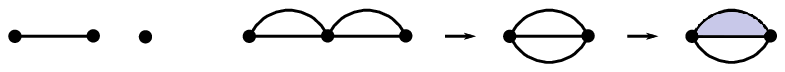}
			\put(-.5,0){\small $a$}
			\put(10,0){\small $b$}
			\put(17,0){\small $c$}
			
			\put(10,-6){\small $\Gamma$}
			
			\put(35,1.5){\small $c$}
			\put(46,1.5){\small $c$}
			\put(35,8){\small $a$}
			\put(46,8){\small $b$}
			
			\put(40,-6){\small $X$}
			\put(93,-6){\small $\Omega$}

			\put(69,-2){\small $c$}
			\put(69,1.15){\small $b$}
			\put(69,8){\small $a$}
			
			\put(93,-2){\small $c$}
			\put(93,1.15){\small $b$}
			\put(93,8){\small $a$}

			\put(91, 4.5){\tiny $\mathbb{R}P^2$}
		
		\end{overpic}
		\bigskip
		\caption{
		Figure illustrating 
		 Example~\ref{ex:fi_completion}}
		\label{fig:fi_completion}
\end{figure}	
\end{example}

\section{Nonpositive curvature} \label{sec_nonpos}
	This section  establishes criteria which guarantee that a completion is non-positively curved or a CAT(0) cube complex, which we rely on in Sections~\ref{sec_separability} and~\ref{sec_reflection_subgroups}.

Recall that a graph-loop is an edge that connects a vertex to itself.
A \textit{bigon} in a CW complex  is a pair of edges $e_1$ and $e_2$, such that the set of vertices that are endpoints of $e_1$ is the same as the set of vertices that are endpoints of $e_2$. 
Note that a bigon could consist of two graph-loops based at the same vertex. 
A \textit{commuting bigon} in a $\G$-labeled complex is a bigon 
whose edges are labeled by adjacent vertices of $\Gamma$. 
Next, we show that the presence of a commuting bigon is the only obstruction to a $\G$-labeled complex being non-positively curved.

\begin{proposition} \label{prop_no_bigon_nonpos}
	Let $\Omega$ be a folded $\G$-labeled cube complex which does not contain a commuting bigon. 
	If 
either $\Omega$ is  cube-full or $\G$ is triangle-free, then  
$\Omega$ 
 is non-positively curved. 
\end{proposition}
\begin{proof}
	Let $v$ be any vertex of $\Omega$, and let $\Delta$ denote the link of $v$ in $\Omega$. We verify that $\Omega$ is non-positively curved by checking that $\Delta$ is a flag simplicial complex. We first check that $\Delta$ is 
	a simplicial complex.
	This part of the proof does not require that $\Omega$ is cube-full or that $\G$ is triangle-free.

Since $\Delta$ is the link of a vertex in a cube complex, to check that it is a simplicial complex, it is enough to show that its $1$-skeleton $\Delta^{(1)}$ does not contain any graph-loops or bigons. 
Suppose $\Delta$ contains a graph-loop, and let $c$ be the square of $\Omega$ contributing this graph loop to $\Delta$.  It follows that 
 a pair of adjacent sides of $c$ 
are identified in $\Omega$. 
However, adjacent sides of squares in $\Omega$ are always labeled by distinct elements of $\Gamma$, so such an identification cannot occur.  
Thus, $\Delta$ cannot contain a graph-loop.
	
	We next check that
	$\Delta^{(1)}$
	does not contain a bigon. Suppose there is such a 
	bigon, whose edges come from (possibly non-distinct) squares $c_1$ and $c_2$ of $\Omega$, and whose vertices come from distinct edges $e_1$ and $e_2$ of $\Omega$, each incident to $v$.
	 Then $e_1$ and $e_2$ are adjacent sides of $c_i$, for $i=1,2$.  Suppose $c_1$ and $c_2$ are distinct. Since opposite sides of squares in $\Omega$ have the same label, it follows that $c_1$ and $c_2$ have the same boundary label, read starting from $v$  in the direction of $e_1$. 
	  However, this is not possible as $\Omega$ is folded. 
	
	On the other hand, suppose 
	$c_1 = c_2 =c$. 
	Then the  bigon in $\Delta$ comes from two corners of $c$. Since
	adjacent sides of a square in $\Omega$ are never identified, it follows that these are necessarily opposite corners of $c$, and furthermore, that 
	opposite sides of $c$ are identified with orientation reversing isometries.
 Consequently,  the attaching map of $c$ maps the boundary of $c$ to a bigon in $\Omega$, consisting of the edges $e_1$ and $e_2$.  Since $e_1$ and $e_2$ are adjacent sides of a square in $\Omega$, their labels commute.  
	This contradicts the assumption that $\Omega$ does not contain commuting bigons. 
	Hence 
	$\Delta^{(1)}$
	cannot contain a bigon.  We have thus verified that $\Delta$ is 
a	simplicial complex.
	
	We now check that $\Delta$ is flag. 
Let $u_1, \dots, u_n$ be the vertices of a complete graph contained in  the $1$-skeleton of $\Delta$.
For $ 1 \le i \le n$, let $e_i$ be the edge of $\Omega$  incident to $v$ which $u_i$ lies on, and let $s_i$ be the label of $e_i$. 
For each $1 \le i < j \le n$,  we know that $s_i$ and $s_j$ are adjacent in $\Gamma$ since $u_i$ is adjacent to $u_j$ in $\Delta$. If $\Omega$ is cube-full, it follows that there is some cube in $\Omega$ containing $v \cup \bigcup_{i = 1}^ne_i$. Thus, $\Delta$ is flag. On the other hand suppose that $\G$ is triangle-free. As $\Omega$ is folded, the labels $s_1, \dots, s_n$ are distinct. Furthermore, as $\G$  is triangle-free and $s_1, \dots, s_n$ as vertices of $\G$ form a complete graph, we have that $n \le 2$. Thus, the flag condition holds under the triangle-free assumption as well.
\end{proof}

If $G$ is a torsion-free subgroup of $W_\G$, then by Proposition~\ref{prop_torsion_free}, a completion of~$G$ cannot have commuting bigons.  Then  Proposition~\ref{prop_no_bigon_nonpos} immediately implies:

\begin{proposition}\label{prop_torsion_free_npc}
Any completion of a 
torsion-free subgroup of a RACG is non-positively curved. 
\hfill{\qed}
\end{proposition}

Next, we show that a completion of a finite $\G$-labeled tree is a finite CAT(0) cube complex.

\begin{proposition} \label{prop_finite_cat0}
	Let $X$ be a $\G$-labeled finite tree and let
	\[X = \Omega_0 \to \Omega_1 \to \dots \to \Omega \]
	be a completion sequence for $X$. Then $\Omega$ is a finite CAT(0) cube complex. 
	Furthermore, there is a finite bound on the length of the completion sequence.
\end{proposition}

\begin{proof}
We begin by proving that each complex in the completion sequence is simply connected.  

As a first step, we 
show that $\Omega_n$ does not contain any graph-loops for $n \ge 0$. Note that the label of every loop in $\Omega_0$ based at $B$ represents the trivial element in $W_\G$. For a contradiction, suppose that for some $n$, $\Omega_n$ contains a graph-loop $l$ with label~$s$. Suppose $l$ is incident to a vertex $v \in \Omega_n$. Let $p$ be a geodesic in $\Omega_n$ from $B$ to $v$, and let $w$ be the label of $p$. It follows that the loop $p l p^{-1}$ in $\Omega_n$ has label $wsw^{-1}$. As $wsw^{-1}$ has an odd number of occurrences of the letter $s$, it represents a nontrivial element of $W_\G$. However, by iteratively applying Lemma~\ref{lemma_pull_back} we conclude that the label of some loop in $\Omega_0$ based at $B$ represents a non-trivial element of $W_\G$. This is a contradiction. Thus, $\Omega_n$ does not contain a graph-loop for any $n$.

Next, we 
 show by induction that $\Omega_n$ is simply connected for all $n \ge 0$.   The base case is true by hypothesis.   Now assume that $\Omega_n$ is simply connected.
 	
	Suppose $\Omega_{n+1}$ is obtained from $\Omega_n$ by attaching a $k$-cube $c$ to the edges $e_1, \dots, e_k$ of $\Omega_{n}$ which are all  incident to the same vertex $v$. Then $\Omega_{n+1}$ can be homotoped onto $\Omega_{n}$ by homotoping $c$ onto $v \cup \bigcup_{i=1}^ke_i$. 
	If $\Omega_{n+1}$ is obtained from $\Omega_n$ by identifying a collection $\{c_i\}$ of $k$-cubes $(k \ge 2)$ with identical boundary to a single cube $c$, then any null homotopy using the  $c_i$'s can be replaced with one that only uses $c$. 
	In both cases, $\Omega_{n+1}$ is simply connected.

Now suppose $\Omega_{n+1}$ is obtained from $\Omega_n$ by a fold operation.  Specifically, suppose that the edges $e_i$ (with endpoints $v$ and $v_i$, for $i = 1,2$)  in $\Omega_n$ are identified to get the edge $e$ in $\Omega_{n+1}$.  By the first paragraph these edges are not graph-loops.  

If $v_1=v_2$, then $e_1 \cup e_2$ is a loop, which is null homotopic because $\Omega_n$ is simply connected.  Since identifying $e_1$ and $e_2$ is equivalent to attaching a disk to this loop, it follows that $\Omega_{n+1}$ is simply connected.  
If $v_1 \ne v_2$, then let $\Omega_n'$ and $\Omega_{n+1}'$ be the complexes obtained by collapsing the contractible subspaces $e_1 \cup e_2$ and $e$ in $\Omega_n$ and $\Omega_{n+1}$ respectively to points.  Then $\Omega_n'$ is homotopy equivalent to $\Omega_n$, $\Omega_{n+1}'$ is homotopy equivalent to $\Omega_{n+1}$, and $\Omega_n'$ is homeomorphic to $\Omega_{n+1}'$.  
Again, we conclude that $\Omega_{n+1}$ is simply connected.  
We have established that $\Omega_n$ is simply connected for all $n \ge 0$, and it readily follows that $\Omega$ is simply-connected as well.
	
	Next, we show that $\Omega$ is non-positively curved.   By Proposition~\ref{prop_no_bigon_nonpos} it is enough to show that $\Omega$ does not contain a commuting bigon. 
	Since $\Omega$ is simply-connected, Lemma~\ref{lemma_omega_paths}(3) implies that the group associated to $(\Omega, B)$ is trivial, and therefore torsion-free.  By Proposition~\ref{prop_torsion_free} there are no commuting bigons.
	
	It follows that  $\Omega$ is CAT(0) as it is  simply-connected and non-positively curved.

	We claim that the diameter of $\Omega$ is at most~$E$, where $E$ is  the number of edges of $X$.
For consider a geodesic $\alpha$ in~$\Omega$. By Proposition~\ref{prop_omega_hyps}, every hyperplane that intersects $\alpha$ must also intersect the image of $X$ in $\Omega$. Furthermore, as $\Omega$ is CAT(0) (and not just non-positively curved), no hyperplane intersects $\alpha$ twice. Thus, the length of $\alpha$ is at most $E$, and, as $\alpha$ was an arbitrary geodesic, the diameter of $\Omega$ is also at most $E$.  
It follows that $\Omega$ is finite, as it is locally finite (since it is folded) and has finite diameter.
	Finally, Proposition~\ref{prop_finite_completion} implies that 
there is a 	finite bound on the length of the completion sequence.
\end{proof}

When $\G$ is triangle-free, we get the following more precise bound on the length of 
 a standard 
 completion sequence,
which is used in Theorem~\ref{thm_completion_of_reflection_subgroup_triangle_free}.

\begin{proposition} \label{prop_finite_cat0_triangle_free}
With the set-up of Proposition~\ref{prop_finite_cat0}, if  $\G$ is additionally triangle-free, then 
there is a finite bound on the length of the completion sequence depending only on the number of edges of $X$ 
		and on $|V(\Gamma)|$.
\end{proposition}

\begin{proof}
	By Proposition~\ref{prop_finite_cat0}, $\Omega = \Omega_N$ for some $N$ and $\Omega$ is a finite CAT(0) cube complex. Let $E$ be the number of edges of $X$. We are left to prove that $N$ only depends on $E$ and on $|V(\G)|$. Consider the subsequence of all \textit{folded} complexes of the given standard completion:
	\[\Theta_1 = \Omega_{i_1}, \Theta_2 = \Omega_{i_2}, \dots, \Theta_n = \Omega_{i_n} \] 
	
	By Proposition~\ref{prop_no_bigon_nonpos}, we know that $\Theta_i$ is a CAT(0) cube complex. Furthermore, $\Theta_i$ has diameter at most $E$, by the proof of Proposition~\ref{prop_finite_cat0}.
	
	We claim that the complex $\Theta_j$ is not isometric to $\Theta_{k}$ for all $k > i$. Suppose otherwise for a contradiction. Consider the sequence of operations performed to $\Theta_{j}$ in order to obtain $\Theta_{k}$. We can repeat this same sequence of operations to $\Theta_{k}$ in order to obtain another folded complex isometric to $\Theta_j$. By iteratively repeating this process, we obtain a standard completion sequence which is infinite. Furthermore, the direct limit $\Omega'$ of this new standard completion sequence must be a finite complex.	To see this, note that given $m$ distinct cells in $\Omega'$, there is some complex isometric to $\Theta_j$ in the completion sequence, which contains	$m$ distinct preimages of the cells (since there are infinitely many such complexes in the sequence).  Thus the size of $\Omega'$ is bounded by the size of $\Theta_j$. 
	However, this contradicts Proposition~\ref{prop_finite_completion}.
	
	Let $F$ be the number of all possible CAT(0) cube complexes 
	of diameter at most $E$ and with at most $|V(\Gamma)|$ edges incident to each vertex.
	As $\Theta_j$ is not isometric to $\Theta_k$ for all $j \neq k$, it follows that $n \le F$. 
	For each $1 \le j \le n$, the number of cube attachments that can be applied to $\Theta_{j}$, and the number of fold and cube identification operations that can be applied to the resulting complex is bounded by a number $K$ which depends only on $E$. Thus,  $N \le KF$ where $K$ and $F$ depend only on $E$ and on $|V(\G)|$. 
\end{proof}

\section{Quasiconvexity} \label{sec_qc}

Let $H$ be a group with fixed generating set. Recall that a subgroup $G$ of $H$ is \textit{$M$-quasiconvex}, for $M \ge 0$, if any geodesic path in the Cayley graph of $H$ with endpoints in $G$ lies in the $M$-neighborhood of $G$. We say $G$ is \textit{quasiconvex} if it is $M$-quasiconvex for some $M$. 
In general, $G$ may be quasiconvex with respect to one generating set for $G$ but not another. 
However, if $G$ is quasiconvex with respect to some generating set, then it is quasi-isometrically embedded with respect to any generating set \cite[Chapter III.$\Gamma$, Lemma 3.5]{BH}.

When we say a subgroup is quasiconvex in a RACG, we will always mean with respect to the standard generating set. 
The main result of this section is that a subgroup of a RACG is quasiconvex if and only if any standard completion of the subgroup is finite.

We first prove a lemma relating distances in a completion of a subgroup to distances in the Cayley graph (associated to the standard generating set) of the RACG.

\begin{lemma} \label{lemma_distances_in_omega}
	Let $G$ be a subgroup of the RACG $W_{\Gamma}$ and let $(\Omega, B)$ be a completion of $G$. Let $w$ be the label of a path in $\Omega$ from the basepoint $B$ to some vertex $v \in \Omega$. Let $\mathcal{C}$ be the Cayley graph of $W_{\Gamma}$, and let $v_w$ be the vertex in $\mathcal{C}$ which represents the element of $W_\G$ corresponding to $w$. Then $d_{\Omega}(B, v) = d_{\mathcal{C}}(G, v_w)$. (Here $G$ is naturally identified with the vertices in $\mathcal{C}$ which represent elements of $G$.)
\end{lemma}
\begin{proof}
	By Lemma~\ref{lemma_omega_paths}, there is a path in $\Omega$ from $B$ to $v$ with label a reduced expression for $w$. Thus, without loss of generality, we may assume that $w$ is reduced.
	Let $\alpha$ be a geodesic in $\mathcal{C}$ from $v_{id}$ to $v_w$ with label $w$, where $v_{id}$ is the vertex in $\mathcal{C}$ labeled by the identity element. Let $\beta$ be a geodesic in $\mathcal{C}$ from $v_w$ to $G$ which realizes the distance from $v_w$ to $G$. Let $h$ be the label of $\beta$. It follows that $h$ is a reduced word. As $\alpha \beta$ is a path from $v_{id}$ to a vertex of $G$, the word $k = wh$ represents an element of $G$. By Lemma~\ref{lemma_reduced_expression}, there is a reduced expression $\hat{k} = \hat{w}\hat{h}$ for $k$ in $W_{\G}$ such that $w' = \hat{w}s_1 \dots s_m$ is a reduced expression for $w$ and $h' = s_m \dots s_1 \hat{h}$ is a reduced expression for $h$, where $s_i \in V(\G)$ for $1 \le i \le m$. 

	By Lemma~\ref{lemma_omega_paths}(\ref{lemma_omega_paths1}), there is a path $\alpha'$ 
	with label $w'= \hat{w}s_1 \dots s_m$ from $B$ to $v$ in $\Omega$. Furthermore, by the definition of the completion of a subgroup, there is a loop $l$ with label $\hat{k}= \hat{w}\hat{h}$ in $\Omega$ based at $B$. 
	Since $\Omega$ is folded, $\alpha'$ and $l$ overlap on the part labeled $\hat w$. It follows that  there is a path from $B$ to $v$ labeled by $h'^{-1} = \hat{h}^{-1}s_1\dots s_m$.

	Let $\gamma$ be a geodesic in $\Omega$ from $v$ to $B$ with label $z$. Note that $z$ must be a reduced word, and that $|z| \le |h'| = |h|$. As $wz$ is the label of a loop in $\Omega$, it follows by the definition of a completion that $wz$ represents an element of $G$. Thus there is a path in $\mathcal{C}$ from $v_w$ to $G$ with label $z$. By the minimality of $\beta$, we have that $|h| \le |z|$. Hence, $|z| = |h|$. It now follows that
$	d_{\Omega}(v_w, G) = |\beta| = |h| = |z| = |\gamma| = d_{\Omega}(B, v_w)$
\end{proof}

\begin{lemma} \label{lemma_is_wqc}
	Let $G$ be a subgroup of the RACG $W_{\Gamma}$. If some completion $(\Omega, B)$ of $G$ is finite, then $G$ is $M$-quasiconvex in $W_{\Gamma}$, where $M$ is the maximal distance of a vertex in $\Omega$ from $B$.
\end{lemma}
\begin{proof}
Let $\alpha$ be a geodesic in the Cayley graph of $W_{\Gamma}$ between two elements of $G$. Without loss of generality, we may assume that
 $\alpha$ goes between 
 the identity vertex $v_{id}$ and some vertex labeled by an element  $g$ of $G$. 
Then the label $w$ of $\alpha$  is a minimal length word representing $g$. Let $v$ be any vertex along $\alpha$. Let $w'$ be the label of the subpath of $\alpha$ from $v_{id}$ to $v$.

	By the definition of a completion, there is a loop $l$ in $\Omega$ based at $B$ with label $w$. Consequently, there is an initial subpath, $l'$, of $l$ with label $w'$. Let $u$ be the vertex of $\Omega$ that is the endpoint of $l'$. By Lemma~\ref{lemma_distances_in_omega}, $d_{\Omega}(u, B) = d_{\mathcal{C}}(v, G) \le M$. 
\end{proof}

\begin{lemma} \label{lemma_qc_implies_finite}
	If $G$ is a quasiconvex subgroup of the RACG $W_{\Gamma}$, then $G$ is finitely generated and every standard completion of $G$ is finite.
\end{lemma} 
\begin{proof} 
	Suppose $G$ is $M$-quasiconvex in $W_{\Gamma}$. 
 Then $G$ must be finitely generated as it is a quasiconvex subgroup of a finitely generated group~\cite[Chapter III.$\Gamma$ Lemma~3.5]{BH}. 
	Let $(\Omega, B)$ be a standard completion of $G$.   Lemma~\ref{lemma_distances_in_omega}      
	implies that $d_{\Omega}(v, B) \le M$ for every vertex 
	$v$ in the core graph $C(\Omega, B)$ from Definition~\ref{def_core_graph}.   
	
	Suppose $\Omega$ is not finite.  
    Since it is folded and locally finite, a standard argument can be used to show that $\Omega$ has an infinite geodesic ray $\zeta$ based at $B$.
	By Proposition~\ref{prop_omega_hyps}, there are only finitely many hyperplanes in $\Omega$. 
	It follows that we may choose a geodesic subsegment $\gamma$ of $\zeta$ starting at $B$,  
	a hyperplane $H$, and an integer $C \ge 0$ such that $H$ is dual to at least 
$M + C + 2$
	 edges of $\gamma$ and such that the distance from $B$ to $H \cap \gamma$ is less than $C$ .   
	
	Starting from $B$, let $e_1$ and $e_2$ respectively be the first and last edge along $\gamma$ dual to $H$.  Let $u$ be the endpoint of $e_1$ closer to $B$, and let $\alpha_1$ be the subpath of $\gamma$ from $B$ to $u$. Note that $|\alpha_1| \le C$. Now at least one endpoint of $e_{2}$ is in the same component of $N(H) \setminus H$
	as $u$; call it $v$.  Let $\alpha_2$ be the subpath of $\gamma$ from $u$ to $v$ and let $\alpha_3$ be a geodesic in $N(H) \setminus H$  from $v$ to $u$.  
	Let $a_i$ be the label of $\alpha_i$, for $1 \le i \le 3$.  
	Then the concatenation $\alpha = \alpha_1\alpha_2\alpha_3\alpha_1^{-1}$ is a loop in $\Omega$ based at $B$, with label $a= a_1a_2a_3 a^{-1}$. 
	
	We will complete the proof by using $\alpha$ to produce a loop in $C(\Omega, B)$ which leaves the $M$-neighborhood of $B$, a contradiction.
	We first claim that a reduced expression $\hat{a}_2 \hat{a}_3$ for $a_2a_3$ is obtained by a sequence of at most $|a_2| - M - C - 1$ deletions. 
	As $a_2$ and $a_3$ are reduced, every pair of generators deleted consists of a generator in $a_2$ and one in $a_3$. 
	Let $s$ be the type of $H$.  
	The claim now follows from the fact that $a_3$ has no occurrences of $s$ (being a word in $\mathrm{link}(s)$) while $a_2$ has at least $M + C + 1$ 
	occurrences of $s$.
	
	By Lemma \ref{lemma_hausdorff_bound}, there is 
	a closed path $\hat\alpha$ 
	in $\Omega$ with label $\hat{a}_2 \hat{a}_3$ based at $u$, and of Hausdorff distance at most $|a_2| - M - C - 1$ from $\alpha_2 \alpha_3$. 
	Then $\beta = \alpha_1 \hat\alpha \alpha_1^{-1}$ is a path based at $B$ in $\Omega$ with label 	$a_1 \hat{a}_2 \hat{a}_3 a_1^{-1}$,  and $\beta$ is Hausdorff distance at most~$|a_2| - M - C - 1$ from $\alpha$.
	As  the vertex $v$ of $ \alpha_1$ has distance $|a_1|  + |a_2|$ from $B$, it follows that $\beta$ contains a vertex $v'$ whose distance from $B$ is at least $|a_1| + M + C + 1$.
	
	Let $\hat{a}$ be a reduced expression for $a_1 \hat{a}_2 \hat{a}_3 a_1^{-1}$ obtained by a sequence of deletion operations. As $a_1$ and $\hat{a}_2 \hat{a}_3$ are each reduced, it follows that each pair of generators deleted in such a sequence contains a generator in $a_1$ or in $a_1^{-1}$ (possibly in both). Thus, there can be at most $2|a_1|$ such deletions. By Lemma \ref{lemma_hausdorff_bound}, there is a path $\hat{\alpha}$ in $\Omega$ based at $B$ with label $\hat{a}$ and with Hausdorff distance at most $2|a_1|$ from $\beta$. As $\beta$ contains $v'$, it follows that $\hat{\alpha}$ contains a vertex of distance at least $(|a_1| + M + C + 1) - 2|a_1| \ge M + 1$ from $B$. This is a contradiction, as $\hat{\alpha}$ is contained in $C(\Omega, B)$, which is itself contained in the $M$ neighborhood of~$B$.
		\end{proof}

We immediately obtain the following theorem from Lemmas~\ref{lemma_is_wqc} and~\ref{lemma_qc_implies_finite}. 

\begin{theorem} \label{thm_qc} 
	Let $G$ be a subgroup of a RACG. The following are equivalent:
	\begin{enumerate}
		\item $G$ is quasiconvex.
		\item Some  completion for $G$ is finite.
		\item $G$ is finitely generated and every standard completion for $G$ is finite. \qed
	\end{enumerate}
\end{theorem} 

\begin{example}\label{ex:nonqc_completion}
Let $\G$ be the graph in Figure~\ref{fig:non_qc_completion}, and let $G = \langle abcde \rangle < W_\G$. The right of 
 this 
 figure shows a standard completion $\Omega$ for $G$.  
 The complex $\Omega$ is a bi-infinite cylinder tiled by squares, as 
the process of attaching squares never stops.
As $\Omega$ is infinite, 
$G$ is not quasiconvex by Theorem~\ref{thm_qc}. 
\begin{figure}[h!]
	\centering
	\begin{overpic}
		[scale=.6]
		{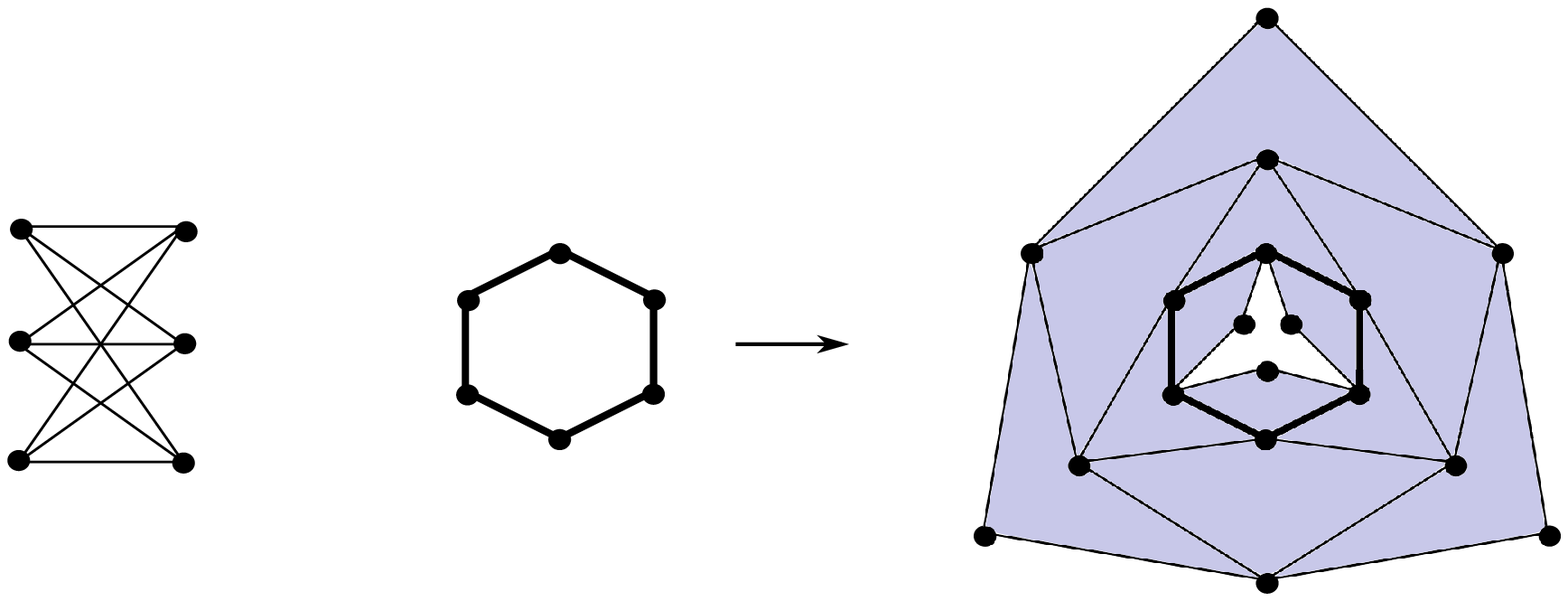}
		\put(-2,23){\small $a$}
		\put(13,23){\small $b$}
		\put(-2,16){\small $c$}
		\put(13,16){\small $d$}
		\put(-2,8){\small $e$}
		\put(13,8){\small $f$}

		\put(4.5,2){\small $\G$}
		\put(67,30){\small $\Omega$}

		\put(38.5,21){\small $a$}
		\put(42,15){\small $b$}
		\put(38,9.5){\small $c$}
		\put(30,9){\small $d$}
		\put(27,15){\small $e$}
		\put(31,21.5){\small $f$}

		\put(60, 16){\small $\dots$}
		\put(79, 16){\small $\dots$}
		\put(98, 16){\small $\dots$}
		
	\end{overpic}
	\bigskip
	\caption{
	Figure illustrating Example~\ref{ex:nonqc_completion}.
	}
	\label{fig:non_qc_completion}
\end{figure}	

\end{example}

\section{Residual finiteness and separability} \label{sec_separability}
In this section we give a proof using completions of the  well-known result that RACGs are residually finite. We additionally give a new proof of a result of Haglund which states that quasiconvex subgroups of RACGs are separable.  We begin by giving a construction which associates a new complex and associated subgroup 
to each subgroup of an RACG that has a finite completion.  This is similar 
to 
the complex used in Stallings' proof of Marshall Hall's theorem \cite{Stallings} and to 
the canonical completions of~\cite{Haglund-Wise-special}

\begin{definition}[Full valence extension and associated subgroup]\label{def:full_valence_ext}
Let $G$ be a finitely generated subgroup of a RACG $W_{\Gamma}$, and suppose that $G$ has a finite completion $(\Omega, B)$. 
For each $s \in V(\Gamma)$ and each vertex $v$ of $\Omega$ that is not  incident to an edge labeled by $s$, we add a graph-loop labeled by $s$ to $v$. We call this first resulting complex $\mathcal{E}_0$. Let $(\mathcal{E}, B)$ be a completion of $\mathcal{E}_0$ obtained by applying Lemma~\ref{lemma_adding_graph_loops}. We call $(\mathcal{E}, B)$ the \textit{full valence extension} of $\Omega$.
The complex $\mathcal{E}$ is folded and cube-full as it is a completion. Furthermore, by construction and by Lemma~\ref{lemma_adding_graph_loops}, $\mathcal{E}$ is finite and full-valence, and the natural inclusion of $\Omega$ into $\mathcal{E}$ is an isometry. 
 Let $H$ be the subgroup associated to $(\mathcal{E}, B)$ (as in Definition~\ref{def_associated_subgroup}). 
\end{definition}

The following lemma is an immediate consequence of Proposition~\ref{prop_subgroup_gen_by_loops} and Lemma~\ref{lemma_index_of_subgroup}.
\begin{lemma}\label{lem:ext_grp_fin_index}
With the notation of Definition~\ref{def:full_valence_ext}, the subgroup $H$ has finite index in $W_\G$. 
\end{lemma}

We now use full valence extensions to give a direct proof of the residual finiteness of RACGs.  The fact that
RACGs are residually finite is well-known: RACGs are linear (see for instance \cite{BB}), and by a theorem of Malcev, every finitely generated linear group is residually finite (see \cite[4.2]{Wehrfritz} for a proof).

\begin{theorem} \label{thm_res_finite}
	Every RACG is residually finite.
\end{theorem}
\begin{proof}
	Let $W_\G$ be a RACG. Let $g$ be a nontrivial element in $W_{\Gamma}$, and let $w$ be a reduced word representing $g$. Let $G$ be the trivial subgroup of $W_{\Gamma}$ given by the generating set $S_G = \{ ww^{-1} \}$. The $S_G$-complex $\Omega_0 = X(S_G)$ consists of a circle, subdivided into labeled edges, whose label, read from from a base vertex $B$, is $ww^{-1}$. We can iteratively perform fold operations to $\Omega_0$ and obtain a complex $\Omega_N$ that is a path labeled by $w$. 
	
	By Proposition~\ref{prop_finite_cat0}, there is a completion $(\Omega, B)$ of $\Omega_N$ that is a finite CAT(0) cube complex. The image of $\Omega_N$ in $\Omega$ is a path, $p'$, based at $B$ and labeled by $w$. Furthermore, the path $p'$ is not a loop in $\Omega$. This follows since $\Omega$ is a completion of the trivial subgroup, and consequently every loop in $\Omega$ based at $B$ must have as label a word that is trivial in $W_{\Gamma}$. 
	
	Let $(\mathcal{E},B)$ be the full valence extension of $(\Omega, B)$. As $\Omega$ is isometrically embedded in $\mathcal{E}$ by Lemma~\ref{lemma_adding_graph_loops}, $w$ is still not the label of a loop in $\mathcal{E}$ based at $B$. 
By Lemma~\ref{lem:ext_grp_fin_index},
 the subgroup $H$ of $W_\G$ associated to $(\mathcal{E}, B)$  has finite index in $W_{\Gamma}$.
Furthermore, $g \notin H$ as $w$ is a reduced word representing $g$ which is not the label of a loop in~$\mathcal{E}$.
\end{proof}

An additional useful property of full valence extensions is that the subgroup $H$ associated to a full valence extension of a finite completion of $G<W_\G$ has a retraction to $G$:

\begin{proposition} \label{prop_extention_props}
With the notation of Definition~\ref{def:full_valence_ext}, there is a retraction from $H$ to $G$. 
\end{proposition}

%
\begin{proof}
	We define a map $\phi: H \to G$ as follows.   
	Let $\mathcal{L}$ be the set of graph-loops in $\mathcal{E}$ that are not in $\Omega$, i.e.~the graph-loops added in the construction of $\mathcal{E}_0$ or in the  completion process.
	Given an element $h \in H$,  since $\mathcal E$ is a completion of $H$ by Proposition~\ref{prop_subgroup_gen_by_loops},
	there is a loop $l$ in $\mathcal{E}$ based at $B$ whose  label $w$ is a word representing $h$. We remove from $l$ all graph-loops it traverses which are in $\mathcal{L}$. Let $l'$ be the resulting loop in $\Omega$ based at $B$, and let $w'$ be its label. It follows that $w'$ represents an element $g \in G$. We set $\phi(h) = g$.  
	
	We first check that $\phi$ is well-defined. Let $l_1$ and $l_2$ be loops in $\mathcal{E}$ based at $B$ with labels $w_1$ and $w_2$, such that $w_1$ and $w_2$ are distinct words, each representing the same element $h \in H$. Let $w$ be a reduced word representing $h$ in $W_\G$. Let $l_1', l_2'$ and $l'$ be the loops obtained by removing graph-loops in $\mathcal{L}$ from $l_1, l_2$ and $l$ respectively. Let $w_1', w_2'$ and $w'$ be the labels of $l_1', l_2'$ and $l'$ respectively. We must show that $w_1'$ and $w_2'$ represent the same element of $W_\G$. To do so, we will show that $w'$ represents the same element in $W_\G$ as both $w_1'$ and $w_2'$.
	
	By Tits' solution to the word problem, there is a sequence of Tits moves that can be performed to $w_1$ to obtain $w$. This sequence naturally produces a sequence of corresponding loops $l_1 = q_1, q_2, \dots, q_n = l$ in $\mathcal{E}$ whose labels are the corresponding words obtained by the Tits moves. Furthermore, if a cancellation move is performed to $q_i$ in order to obtain $q_{i+1}$, then as $\mathcal{E}$ is folded, it readily follows that the edges involved in this cancellation move are either both in $\mathcal{L}$ or both not in $\mathcal{L}$. Thus, by forgetting the Tits moves performed to generators which are labels of graph-loops in $\mathcal{L}$, this sequence of Tits moves induces a sequence of Tits moves performed to $w_1'$ to produce $w'$. Hence, $w_1'$ and $w'$ represent the same element of $W_\G$. Similarly, $w_2'$ and $w'$ represent the same element of $W_\G$. Consequently, $\phi$ is well-defined.
	
	It is clear that $\phi$ is a homomorphism.
Furthermore, given an element $g \in G$ and a loop $l$ in $\mathcal{E}$ based at $B$ with label a reduced word representing $g$, we have that $l$ is contained in the subcomplex $\Omega \subset \mathcal{E}$. Thus, $l$ does not traverse any graph-loops in $\mathcal{L}$. It follows that $\phi$ restricted to elements of $G$ is the identity. Hence, $\phi$ provides the desired retraction.
\end{proof}

The above proposition can now be used to recover a theorem of Haglund. 
Recall that 
a subgroup $G$ of a group $K$ is a virtual retract if $G$ is a retract of a finite index subgroup of $K$.

\begin{theorem}[Haglund \cite{Haglund}] \label{thm_virtual_retract}
	Let $G$ be a quasiconvex subgroup of a RACG $W_{\Gamma}$. 
		Then $G$ is separable and is a virtual retract of $W_{\Gamma}$.
\end{theorem}
\begin{proof}
	By Theorem~\ref{thm_qc}, there is a finite completion $(\Omega, B)$ of $G$. Let $(\mathcal{E}, B)$ be the full valence extension of $(\Omega, B)$, 
	with associated subgroup $H$. 
	Then $H$ has finite-index in $W_\G$ by Lemma~\ref{lem:ext_grp_fin_index}  
	 and $G$ is a retract of $H$ by Proposition~\ref{prop_extention_props}.
	Thus, $G$ is a virtual retract.
	
	It is well-known that a virtual retract of a residually finite group is separable (see \cite[Proposition 3.8]{Haglund}). So, $G$ is separable.
\end{proof}

\section{Reflection Subgroups} \label{sec_reflection_subgroups}
A \textit{reflection} in $W_\G$ is an element  represented by a word of the form $wsw^{-1}$ where 
$s$ is a generator in $V(\Gamma)$ and $w$ is a word in $W_{\Gamma}$. In this section, we 
give a constructive argument to build a finite completion for any subgroup of $W_\G$ generated by a finite set of reflections.  In particular,  such subgroups are always quasiconvex and their index can be computed.

\subsection{Trimmed sets of reflections}
	In this subsection, we define
	 the notion of a trimmed set of reflections (Definition~\ref{def:trimmed}). We show in 
Lemma~\ref{lemma_reflection_generators}
	 that any subgroup generated by 
	 a non-trimmed 
	set of reflections can also be generated by a trimmed set of reflections. Moreover, the lengths of words in the trimmed set sum to a strictly smaller sum than those of the original set. This minimality property will facilitate our later arguments.  Another advantage of trimmed reflections is that they form a \textit{standard} generating set for a RACG subgroup by a theorem of Dyer (see Proposition \ref{cor_dyer}). On the other hand, a non-trimmed set of reflections usually gives a non-standard generating set for the RACG subgroup.

\begin{definition}[Trimmed reflection set]\label{def:trimmed}
	Let $W_\Gamma$ be a RACG. We say that a set 
		\[\mathcal{R} =  \{w_i s_i w_i^{-1} ~ | ~ w_i \in W_{\Gamma} \text{ and } s_{i} \in V(\Gamma),~ 1 \le i \le m \}\]
	of reduced reflections is \textit{trimmed} if  for all $i \neq j$, no reduced
	expression for $w_j$ begins with~$w_is_i$. 
\end{definition}

\begin{lemma} \label{lemma_reflection_generators}
	Let $W_{\Gamma}$ be a RACG, and let $G$ be a subgroup generated by a finite set of reflections $\mathcal{R}'$. Then $G$ is generated by a 
trimmed set of
	reflections $\mathcal{R}$. Furthermore, there is a constructive algorithm to obtain $\mathcal{R}$ from $\mathcal{R}'$,
		whose time-complexity only depends on the number $\sum_{r \in \mathcal{R}'}|r|$.
\end{lemma}

\begin{proof}
	Without loss of generality, we may assume elements in $\mathcal{R}'$ are reduced. Let $g = wsw^{-1}$ be a reflection in $\mathcal{R}'$ so that $w$ has an expression $w = w's'q$, where $h = w's'w'^{-1}$ is another reflection in $\mathcal{R}'$ and $q$ is a word in $W_{\Gamma}$.
	
	In $\mathcal{R}'$, we replace $g$ with a reduced representative of the shorter length reflection $g' = hgh^{-1} = (w'q)s(w'q)^{-1}$, to obtain a new set $\mathcal{R}''$. The set $\mathcal{R}''$ still generates $G$, as $g = h^{-1}g'h$. By iteratively performing such replacements, we obtain the desired generating set~$\mathcal{R}$. This process must end since at each step we obtain a set of generators whose lengths sum to a strictly smaller number than those in the previous step.
\end{proof}

\subsection{A completion for reflection subgroups} \label{subsec_completion_for_ref_subgroups}
Throughout this subsection, we fix the notation in the discussion below. This notation is also used in Section~\ref{sec_algorithm_for_2d_subgroups}.

Let $G$ be a subgroup of the RACG $W_{\Gamma}$, generated by a finite set of reflections:
\[\mathcal{R} =  \{w_i s_i w_i^{-1} ~ | ~ w_i \in W_{\Gamma} \text{ and } s_{i} \in V(\Gamma),~ 1 \le i \le m \}\]
By Lemma~\ref{lemma_reflection_generators} we may assume without loss of generality that $\mathcal{R}$ is trimmed.

Our goal is to give a finite completion $(\Omega_G, B)$ of $G$. We begin by describing the first complex $\Omega_0$ in this completion. For each $1 \le i \le m$, we attach a subdivided circle to the base vertex $B$, with label $w_is_iw_i^{-1}$. Next, for each $i$,  we fold the two copies of $w_i$ onto one another, and we call this resulting graph $\Omega_0$. 
Thus the graph $\Omega_0$ has, for each $1 \le i \le m$, a path  emanating from $B$ and labeled by $w_i$, with a graph-loop labeled by $s_i$ attached at its endpoint. By Theorem~\ref{thm_subgroup_completion}, any completion of $\Omega_0$ is a standard completion of $G$.

Let $\mathcal{T}$ denote the tree obtained by removing the graph-loops from $\Omega_0$. Let $\mathcal{FT}$ be the folded tree obtained by iteratively performing fold operations to $\mathcal{T}$. 
Let $\Omega_{\mathcal{FT}}$ be a standard completion of~$\mathcal{FT}$. 
By Proposition~\ref{prop_finite_cat0}, we know that $\Omega_{\mathcal{FT}}$ is a finite CAT(0) cube complex. Furthermore, $\Omega_{\mathcal{FT}}$ is also a completion of $\mathcal{T}$ by construction. 
Let $\hat{f}: \mathcal{T} \to \mathcal{FT} \to \Omega_{\mathcal{FT}}$ be the natural map. By a slight abuse of notation, we also denote by $\hat{f}$ the natural map $\hat{f}: \mathcal{FT} \to \Omega_{\mathcal{FT}}$. Let $\hat{\mathcal{T}} \coloneqq   \hat{f}(\mathcal{T}) = \hat{f}(\mathcal{FT})$. 

Given a vertex $\hat{v}$ in $\Omega_{\mathcal{FT}}$, define
\[L_{\hat{v}} = \{ s \in V(\G) ~ | ~ 
 \exists ~v \in V(\Omega_0) \text{ incident to a graph-loop labeled } s, \text{ such that } \hat{v} = \hat{f}(v) \} \]
We would like to build the completion $\Omega_G$ by ``adding back'' the graph-loops to $\Omega_{\mathcal{FT}}$ and 
applying Lemma~\ref{lemma_adding_graph_loops}. However, there is a technical issue: when adding back a graph-loop labeled $s$  to a vertex of $\hat{v}$ of $\Omega_{\mathcal{FT}}$, a priori  $\hat{v}$ might  already  be incident to an edge labeled by $s$. If this were true, then the hypothesis of Lemma~\ref{lemma_adding_graph_loops} would not be satisfied. The next two technical lemmas show this situation does not arise. 
	
\begin{lemma} \label{lemma_no_bad_paths}
	Let $ks$ be a reduced word in $W_{\Gamma}$ such that $s \in V(\Gamma)$ and $k$ is a (possibly empty) word consisting only of generators
	 that are adjacent to $s$ in $\Gamma$. Then given any $v \in \Omega_0$ incident to a graph-loop labeled by $s$, no path in $\mathcal{T} \subset \Omega_0$ starting at $v$ is labeled by an expression for~$ks$.
\end{lemma}
\begin{proof}
	For a contradiction, suppose  such a path $\alpha$ exists.   We may assume that $\alpha$ is a geodesic in $\mathcal{T}$.  For if not, 
	 then as $\mathcal{T}$ is a tree, some generator would  be consecutively repeated in the label of $\alpha$, and we   would be able to pass to a homotopic path. 
		Let $u$ be the endpoint of $\alpha$.

	Let $\beta_1$ be the geodesic from the base vertex $B$ to $v$, with label $h_1$.  Then there is an element $r_1 = h_1 s h_1^{-1}$ in $\mathcal R$.  We first show that $u$ does not lie on $\beta_1$.  Suppose it does.  Then $h_1$ has a suffix which is a reduced word equal in $W_\G$ to $sk^{-1}$.  Since $k$ commutes with $s$, it follows that the expression $h_1 s h_1^{-1}$ is not reduced, a contradiction.

	Now let $\beta_2$ be the geodesic from $B$ to $u$, with label $h_2$.  Since $u$ does not lie on $\beta_1$, it follows that $h_2$ is non-empty.  Moreover, there is a reflection $r_2 \in \mathcal R$, given by $(h_2h_2')s'(h_2'^{-1}h_2^{-1})$, where $h_2'$ could be empty.  
	Next, we claim that $h_1$ is non-empty.  For if not, then 
	$r_1 = s$, and 
	$h_2$ is an expression for $ks$. 
	  It follows that
	$r_2$ has a reduced expression that begins with $s$, which is not possible as $\mathcal{R}$ is trimmed.

	Thus, $h_1$ and $h_2$ are non-empty and $ks$ is a reduced expression for $h_1^{-1}h_2$. 
	By Lemma~\ref{lemma_reduced_expression} 
	there exist (possibly empty) words $x$, $k'$ and $k''$ 
	such that either $k'x$ and $x^{-1}k''s$ are reduced expressions in $W_\G$ for respectively $h_1^{-1}$ and $h_2$, or alternatively $sk'x$ and $x^{-1}k''$ are reduced expressions in $W_\G$ for respectively $h_1^{-1}$ and $h_2$. Moreover, $k'k''$ is equal to $k$ in $W_\G$.
	The latter case implies that the reflection $r_1$ has another reduced expression $(sk'x)^{-1}s(sk'x) $. However, this is a contradiction as this word is clearly not reduced. In the former case, we have that $r_1$ has reduced expression $(k'x)^{-1}s(k'x)$ and $r_2$ has reduced expression $(x^{-1}k''sh_2')s'(x^{-1}k''sh_2')^{-1}$. Set $w_1 = (k'x)^{-1}$ and $w_2 = x^{-1}k''sh_2'$. As the given expression for $r_1$ is reduced and as $s$ commutes with $k'$, it must be that $k'$ is the empty word. In particular, $w_1 = x^{-1}$. Furthermore, since $s$ commutes with $k''$, $w_2$ has reduced expression $x^{-1}sk''h_2'$. However, this contradicts our choice of $\mathcal{R}$ since some expression for $w_2$ begins with $w_1s$. The claim follows.
\end{proof}

\begin{lemma} \label{lemma_no_bad_vertices}
	For every vertex $\hat{v}$ of $\Omega_{\mathcal{FT}}$ and every $s \in L_{\hat{v}}$, 
	there is no edge labeled by $s$ incident to $\hat{v}$ in $\Omega_{\mathcal{FT}}$
\end{lemma}
\begin{proof}
	For a contradiction, suppose that there exists some $\hat{v}$ of $\Omega_{\mathcal{FT}}$, $s \in L_{\hat{v}}$ and an edge $d$ incident to $\hat{v}$ in $\Omega_{\mathcal{FT}}$ which is labeled by $s$. Note that $\hat{v} \in \hat{\mathcal{T}}$. 
	
	Let $H$ be the hyperplane (recall that $\Omega_{\mathcal{FT}}$ is a CAT(0) cube complex) dual to $d$. In particular, $H$ is of type $s$. By Proposition~\ref{prop_omega_hyps}, $H$ intersects $\hat{\mathcal{T}}$ at some edge $\hat{e}$. Let $e$ be an edge of $\mathcal{T}$, such that $\hat{f}(e) = \hat{e}$. Note that the label of $e$ must be $s$. Let $v \in \mathcal{T}$ be such that $\hat{f}(v) = \hat{v}$.  Let $\beta$ be a geodesic in $\mathcal{T}$ from $v$ to $e$. Let $\hat{\beta} = \hat{f}(\beta)$. It follows that $\hat{\beta}$ is a path in $\Omega_{\mathcal{FT}}$ from $\hat{v}$ to $\hat{e}$. By Lemma~\ref{lemma_omega_paths}, there exists a geodesic $\hat{\beta}'$ with the same endpoints as $\hat{\beta}$ and with label a reduced expression for the label of $\hat{\beta}$. Finally, let $\gamma$ be a path in the carrier of $H$ from $\hat v$ to the endpoint of $\hat{e}$. Note that the label of $\gamma$ only consists of vertices in $\text{link}(s)$.
	
	As $\hat{\beta}'$ is geodesic, any hyperplane intersects it at most once. Thus, as hyperplanes separate $\Omega_{\mathcal{FT}}$ into two components, it follows that any hyperplane that intersects $\hat{\beta}'$ must also intersect $\gamma$. It follows that the label of $\hat{\beta}'$ consists only of vertices in $\text{link}(s)$. However, the label of $\hat{\beta}'$ and the label of $\beta$ are expressions for the same element of $W_{\Gamma}$. This implies that $\beta \cup e$ is a path in $\mathcal{T}$ based at $v$ whose label is 
 an expression for the word $ks$, 
	 where every generator in $k$ is adjacent to $s$ in $\Gamma$. This contradicts Lemma~\ref{lemma_no_bad_paths}.
\end{proof}

We are now ready to prove the main results of this section.
\begin{theorem} \label{thm_completion_of_reflection_subgroup}
	Let $G$ be a finitely generated reflection subgroup of a RACG. Then there exists a finite completion of $G$.
\end{theorem}
\begin{proof}
	As previously discussed, we first obtain the completion $\Omega_{\mathcal{FT}}$ of $\mathcal{T}$ using Proposition~\ref{prop_finite_cat0}. For each vertex $\hat{v}$ of $\Omega_{\mathcal{FT}}$ and $s \in L_{\hat{v}}$, we attach a graph-loop to $\hat{v} \in \Omega_{\mathcal{FT}}$ labeled by~$s$. Let $\Omega_{\mathcal{FT}}'$ be the resulting complex. Note that by Lemma~\ref{lemma_no_bad_vertices}, such a graph-loop is never attached to a vertex that is  incident to an edge with the same label as the graph-loop. Furthermore, note that $\Omega_{\mathcal{FT}}'$ can be obtained from $\Omega_0$ by applying the same completion sequence that was applied to $\mathcal{T}$ to obtain $\Omega_{\mathcal{FT}}$, while ``ignoring'' the graph-loops. We now get a finite completion $\Omega_G$ of $\Omega_{\mathcal{FT}}'$ by applying Lemma~\ref{lemma_adding_graph_loops}. It follows that $\Omega_G$ is a finite completion for~$G$. 
\end{proof}

\begin{example} \label{ex:reflection_subgroup}
Let $\G$ be the graph in Figure~\ref{fig:refl_completion}.  The right of 
 this 
 figure shows a completion for the reflection subgroup $\langle (dac)b(dac)^{-1}, (dc)d(dc)^{-1}\rangle$. 

\begin{figure}[h!]
	\centering
	\begin{overpic}
		[scale=1.3]
		{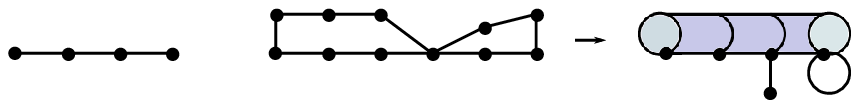}
		\put(-.5,2){\small $a$}
		\put(6,2){\small $b$}
		\put(12,2){\small $c$}
		\put(18,2){\small $d$}
		
		\put(8,-3){\small $\Gamma_4$}
		
		\put(34,3){\small $c$}
		\put(40,3){\small $a$}
		\put(46,3){\small $d$}
		\put(52,3){\small $d$}
		\put(59,3){\small $c$}

		\put(29.5,7){\small $b$}
		\put(63,7){\small $d$}

		\put(34,11){\small $c$}
		\put(40,11){\small $a$}
		\put(47,8.3){\small $d$}
		\put(52,7.8){\small $d$}
		\put(59,10){\small $c$}
		
		\put(80,3.5){\small $c$}
		\put(86,3.5){\small $a$}
		\put(88,1.5){\small $d$}
		\put(92,3.5){\small $c$}
		\put(96,-2){\small $d$}
		
		\put(73,7){\small $b$}
		\put(100,7){\small $b$}
		\put(84,7){\small $b$}
		\put(90,7){\small $b$}
		
	\end{overpic}
	\bigskip
	\caption{A completion for a reflection subgroup. }
	\label{fig:refl_completion}
\end{figure}	
\end{example}
The following corollary immediately follows from Theorem~\ref{thm_completion_of_reflection_subgroup} and Theorem~\ref{thm_qc}.
\begin{corollary}\label{cor_refection_subgroup_quasiconvex}
	Every finitely generated reflection subgroup of a RACG is quasiconvex.
\hfill{\qed}	
\end{corollary}

For $2$-dimensional RACGs, we obtain the following stronger result which shows that the time-complexity of the algorithm which builds the completion of a reflection subgroup is bounded by the size of words in the generating set of reflections. 
This is used in 
Section~\ref{sec_algorithm_for_2d_subgroups}.
\begin{theorem} \label{thm_completion_of_reflection_subgroup_triangle_free}
	Let $W_{\Gamma}$ be a $2$-dimensional RACG. Let $G$ be a subgroup of $W_{\Gamma}$ generated by a finite set of reflection words $\mathcal{R}$. 
Then there is a finite completion sequence for $G$ whose length only depends on the numbers $\sum_{r \in \mathcal{R}}{|r|}$ and $|V(\G)|$.
\end{theorem}
\begin{proof}
	By Lemma~\ref{lemma_reflection_generators}, we may assume without loss of generality that $\mathcal{R}$ is trimmed.
	As before, we first obtain the completion $\Omega_{\mathcal{FT}}$ of $\mathcal{T}$. However, this time we use the more refined Proposition~\ref{prop_finite_cat0_triangle_free} which guarantees that the number of steps in this completion sequence only depends on $|\Omega_0|$
	and $|V(\G)|$. The rest of the proof follows by repeating the proof of Theorem~\ref{thm_completion_of_reflection_subgroup} and noting that Lemma~\ref{lemma_adding_graph_loops}(\ref{lemma_adding_graph_loops_bound}) guarantees the bound.
\end{proof}

\begin{remark} \label{rmk_special_properties_of_completion}
By Lemma~\ref{lemma_adding_graph_loops}, the completion $\Omega_G$ of the reflection group~$G$ given by Theorems~\ref{thm_completion_of_reflection_subgroup} and~\ref{thm_completion_of_reflection_subgroup_triangle_free} contains the complex $\Omega_\mathcal{FT}$ as an isometrically embedded subcomplex. Moreover, the inclusion of $\Omega_{\mathcal{FT}}$ into $\Omega_G$ satisfies the additional properties given by Lemma~\ref{lemma_adding_graph_loops}. 
\end{remark}

\section{Coxeter subgroups of $2$-dimensional RACGs} \label{sec_2d_RACG}
In this section we study Coxeter subgroups of $2$-dimensional RACGs. Recall that such subgroups, by our definition, are always finitely generated.
It is clear that a Coxeter subgroup of a RACG $W_\G$ must be generated by involutions (order two elements) in $W_\G$. We show in Theorem~\ref{thm_racg_subgroup_gen_set} that under mild hypotheses these subgroups are generated by reflections.  Consequently, they are quasiconvex by Theorem~\ref{thm_completion_of_reflection_subgroup} and Theorem~\ref{thm_qc}.
	
We first prove three lemmas involving the structure of particular types of words in a RACG.
The first is well-known and addresses involutions:
\begin{lemma} \label{lemma_involutions}
	Let $g$ be an involution in the RACG $W_{\Gamma}$. Then there is an expression $wkw^{-1}$ for $g$, such that every generator in the word $k$ is in a common clique of $\Gamma$. 
\end{lemma}
\begin{proof}
	It is well known that every finite subgroup of a RACG is contained in a conjugate of a special finite subgroup (see \cite[Theorem 12.3.4]{Davis} for instance).
	The lemma follows, as a special subgroup of a RACG is finite if and only if its defining graph is a clique. 
\end{proof}

The next lemma concerns the structure of particular types of commuting words.

\begin{lemma} \label{lemma_commuting}
	Let $w = s_1 \dots s_m$ and $k = k_1 \dots k_n$ be reduced commuting words in the RACG $W_{\Gamma}$ and suppose that the vertices $k_1, \dots, k_n \in V(\G)$ are all contained in a common clique of $\G$. Then for each $1 \le i \le m$, either:
	\begin{enumerate}
		\item \label{lemma_commuting1} $s_i = k_r$ for some $1 \le r \le n$ and $m(s_i, s_j) = 2$ for all $j \neq i$ or 
		\item  \label{lemma_commuting2} $m(s_i, k_j) = 2$ for all $1 \le j \le n$
	\end{enumerate}
\end{lemma}
\begin{proof}
	Note that as $k$ is reduced and $k_1, \dots, k_n$ pairwise commute, it follows that $k_i \neq k_j$ for all $i \neq j$. As $w$ and $k$ commute, there exists a disk diagram $R$ with boundary label $wkw^{-1}k^{-1}$. We think of $R$ as a rectangle. The vertical sides are labeled, read from bottom to top, by $w$, and the horizontal sides of $R$ are labeled, read from left to right, by $k$. As $w$ and $k$ are reduced, no dual curve intersects the same side of $R$ twice. 
	
	Fix $s \in \G$ such that $s = s_i$ for some $1 \le i \le m$. Let $e_1, \dots, e_l$ be the set of edges labeled by $s$ on the left side of $R$, ordered from bottom to top.  Let $e_1', \dots, e_l'$ be the edges labeled by $s$ on the right side of $R$, ordered from bottom to top. We think of the edge $e_i$ as lying ``directly across'' from $e_i'$ in $R$. Consider the set $\mathcal{H}$ of all dual curves in $R$ of type $s$. 
	
	Suppose first that $s \neq k_r$ for all $1 \le r \le n$. As dual curves of the same type do not intersect, it follows that for each $1 \le j \le l$, there is a curve in $\mathcal{H}$ intersecting both $e_j$ and $e_j'$. 
	Let $H$ be the dual curve in $\mathcal{H}$ that is bottom-most in $R$, i.e., that intersects $e_1$ and $e_1'$. Let $\alpha$ be the path along the boundary of $R$ from the bottom of $e_1$ to the bottom of $e_1'$. 
	Let $t = k_r$ for any $1 \le r \le n$. 
	Observe that the label of $\alpha$ has an odd number of occurrences of the letter $t$. Hence some curve intersecting an edge in $\alpha$ labeled by $t$ must intersect $H$. Thus, $m(s,t) = 2$ and 
	item (\ref{lemma_commuting2}) in the statement of the lemma holds in this case.

	On the other hand, suppose that $s = k_r$ for some $1 \le r \le n$. 
	Let $d$ and $d'$ respectively be the edges on the bottom and top of $R$ labeled by $k_r$. 
	The	dual curves in $\mathcal{H}$ must take one of two possible forms.  
	The first possibility is that there are curves in $\mathcal H$ intersecting the following  pairs of edges: $d$ and $e_1$, $e_j'$ and $e_{j+1}$ for $1 \le j \le l-1$, and $e_{l}'$ and $d'$.  Otherwise, there are curves in $\mathcal H$  intersecting the pairs 
	$d$ and $e_1'$, $e_j$ and $e_{j+1}'$ for $1 \le j \le l-1$, and $e_{l}$ and $d'$.
	
	Let $t$ be a letter in $w$ such that $t \neq s$.  If $t$ also appears in $k$, then $m(s,t) = 2$. 
	Otherwise, the dual curves labeled by $t$ all go across $R$, such that for all $j$, the $j$th edge labeled $t$ on the left side is paired with the $j$th edge labeled $t$ on the right side.  
	Now the structure of dual curves in $\mathcal{H}$ (in either case) forces each dual curve labeled $t$ to intersect a curve in $\mathcal H$.  It follows that $m(s,t) = 2$. 
Thus, $s$ commutes with every generator of $w$ that is not equal to $s$. All that is left to show is that there is only one occurrence of the generator $s$ in $w$, i.e., that $l = 1$. However, if there were two occurrences of $s$, these occurrences can be deleted (as $s$ commutes with every generator in $w$). This is not possible, as $w$ is reduced. Thus, $m(s_i, s_j) = 2$ for all $1 \le j \le m$ such that $j \neq i$, and item (\ref{lemma_commuting1}) in the statement of the lemma holds.
\end{proof}

The following lemma about commuting words will be used in the proof of Theorem~\ref{thm_racg_subgroup_gen_set}.

\begin{lemma}\label{lemma_ax_commuting}
Let $b$ and $x = zs_1s_2z^{-1}$ be reduced commuting words in a RACG 
$W_\G$, where $s_1, s_2 \in V(\G)$ and $z$ is a word in $W_{\G}$.  
Suppose also that $s_1$ commutes with both $s_2$ and $z$.  Then $b$ commutes with 
$\hat{x} = zs_2z^{-1}$
\end{lemma}
\begin{proof}
	As in the previous lemma, we form a ``rectangular'' disk diagram $R$ with boundary label $bxb^{-1}x^{-1}$. The vertical sides of $R$ are labeled, read from bottom to top, by $b$, and the horizontal sides of $R$ are labeled, read from left to right, by $x$. As $b$ and $x$ are reduced, no dual curve intersects the same side of $R$ twice. Let $\mathcal{H}$ be all dual curves in $R$ of type $s_1$. 
	
	As $x$ is reduced and $s_1$ commutes with both $s_2$ and $z$, it readily follows that there is only one occurrence of $s_1$ in $x$. Let $e_t$ be the unique edge on the top of $R$ labeled by $s_1$, and let $e_b$ be the unique edge on the bottom of $R$ labeled by $s_1$.
	
	First suppose that $b$ does not contain any occurrences of the generator $s_1$. Then $\mathcal{H}$ consists of a single dual curve $H$ which intersects both $e_t$ and $e_b$. Let $N(H)$ be the set of cells in $R$ that contain an edge intersecting $H$. It follows that the boundary $\partial N(H)$ of $N(H)$ has label $s_1 ys_1y^{-1}$ for some word $y$. We can then excise $(N(H) \setminus \partial N(H)) \cup e_t \cup e_b$ from $R$ and then glue back together the resulting components along their boundary paths labeled by $y$. What results is a new disk diagram with boundary label $b \hat{x} b^{-1} \hat{x}^{-1}$. Thus, $b$ commutes with $\hat{x}$.
	
	On the other hand, suppose that $b$ has one or more occurrence of $s_1$. We consider two cases. The first case is that $s_1$ commutes with every generator of $b$. As $b$ is reduced, it readily follows that there is an unique occurrence of $s_1$ in $b$. 
Furthermore $ s_1 \hat{b}$ and $\hat{b}s_1$ are both expressions for $b$, where $\hat{b}$ is the word obtained from $b$ by removing the generator $s_1$. 
Now we have the following equalities in $W_\G$: 
	\[\hat{b} \hat{x}  =  \hat{b} s_1 s_1 \hat{x} = bx = xb = \hat{x} s_1 s_1 \hat{b} =  \hat{x} \hat{b} \]
Thus, $\hat b$ and $s_1$ both commute with $\hat x$, so $b$ does as well, and the lemma follows for this case.
	
	For the second case, suppose there are generators in $b$ that do not commute with $s_1$. We will show that this case is actually not possible as we obtain a contradiction.  
	
	Let $e_1, \dots, e_l$ be the set of edges labeled by $s_1$ on the left side of $R$ ordered from bottom to top, and let $e_1', \dots, e_l'$ be the set of edges labeled by $s_1$ on the right side of $R$ ordered from bottom to top. As in the previous lemma, the dual curves in $\mathcal H$ either intersect the pairs $e_b$ and $e_1$, $e_j'$ and $e_{j+1}$ for $1 \le j \le l-1$, and $e_{l}'$ and $e_t$, or  intersect the pairs $e_b$ and $e_1'$, $e_j$ and $e_{j+1}'$ for $1 \le j \le l-1$, and $e_{l}$ and $e_t$. We assume that the dual curves in $\mathcal{H}$ have the first configuration described (the proof in the other possible configuration same). 
	
	Consider the first occurrence in $b$ of a generator $t$ that does not commute with $s_1$, and let $d$ be the corresponding edge on the left side of $R$ labeled by $t$. Note that $d$ is the ``bottom-most" edge on the left side of $R$ with label $t$. Let $T$ be the curve in $R$ that intersects $d$. Note that $T$ cannot intersect any dual curve in $\mathcal{H}$ as $s_1$ and $t$ do not commute. Furthermore, $T$ cannot intersect an edge on the top or bottom of $R$ as the label of every such edge commutes with $s_1$. 
	Thus, it readily follows by the structure of dual curves in $\mathcal{H}$ that $d$ cannot lie before $e_1$. Similarly, $d$ cannot occur after $e_l$, as then the edge on the right side of $R$ labeled by $t$ occurs after $e_l'$, which is again not possible by the structure of $\mathcal{H}$.
	
	Thus $d$ must lie between $e_r$ and $e_{r+1}$ for some $1 \le r < l$. However, as $T$ cannot intersect a dual curve in $\mathcal{H}$ and cannot intersect the bottom of $R$, it follows that $T$ intersects an edge on the right side of $R$ that lies before $e_r'$. Correspondingly, there is an edge on the left side of $R$ lying below $e_r$ with label $t$. This contradicts the fact that $d$ is the bottom-most such edge on the left side of $R$. 
		The lemma follows.
\end{proof}

An \textit{isolated vertex} of a graph is a vertex that is not adjacent to any other vertex. 
If $W_\G$ is a  $2$-dimensional RACG, we show:

\begin{theorem} \label{thm_racg_subgroup_gen_set} 
	Let $G$ be a subgroup of a RACG $W_{\Gamma}$, where~$\G$ is triangle-free. Suppose that $G$ is isomorphic to the Coxeter group $W_{\G'}$ (which is right-angled by Proposition \ref{prop_coxeter_subgroups_are_racgs}), where $\Gamma'$ does not have an isolated vertex. Then $G$ is generated by a finite set of reflections in $W_{\G}$.
\end{theorem}
\begin{proof}
	As $G$ is a RACG, it is generated by the standard Coxeter generating set corresponding to $\G'$. 
	In particular, there exists a finite generating set $I_G$ for $G$ consisting of reduced words representing involutions in $W_{\G}$, such that 
	\begin{enumerate}
		\item There is a bijective map from $V(\G')$ to the elements of $I_G$, and
		\item  If there is an edge between two vertices of $\G'$ then the corresponding elements of $I_G$ commute. 
	\end{enumerate}
	We inductively construct a sequence of generating sets for $G$, each consisting of $|I_G|$ reduced words representing involutions in $W_\G$ and satisfies properties (1) and (2).  Furthermore, each generating set in the sequence will contain one more reflection than the previous one.

	Let $r$ be the number of reflections in $I_G$.
	If $r = |I_G|$, then we are done. Otherwise, let $h \in I_G$ be such that $h$ is not a reflection. 
	By properties (1) and (2) above, and since $\Gamma'$ does not have isolated vertices, there exists an $h' \in I_G$ which is distinct from and commutes with $h$.
	By Lemma~\ref{lemma_involutions} and since $\G$ is triangle-free, we conclude that $h = ws_1s_2w^{-1}$ and $h' = w'k'w'^{-1}$, where $s_1$ and $s_2$ are adjacent vertices in $\Gamma$, 
	$k'$ is a word of length at most two whose generators are in a common clique, and $w, w'$ are words in $W_{\Gamma}$. 
	
	Consider the subgroup $H = w'^{-1}Gw'$ of $W_{\G}$. Note that $H$ is generated by $I_H = w'^{-1}I_Gw'$. Let $x$ be a reduced expression for $w'^{-1}hw'$ obtained by applying a sequence of deletions. Then $x$ must be of the form $x = zs_1s_2z^{-1}$ for some word $z$ in $W_\G$. Note that $k' = w'^{-1}h'w' \in I_H$ and that $k'$ and $x$ commute. 
	
	We claim that either $k' = s_1$ or $k' = s_2$. First suppose for a contradiction that $k'$ is of length two, say $k' = k_1k_2$, for some distinct $k_1, k_2 \in V(\G)$. 
	If $k_1 \neq s_1$ and $k_1 \neq s_2$, then by Lemma~\ref{lemma_commuting}, we have that $m(k_1, s_1) = m(k_1, s_2) = 2$. However, this is not possible as $\G$ is triangle-free. 
 	Thus $k_1$ is equal to either $s_1$ or $s_2$, and similarly, $k_2$ is equal to either $s_1$ or $s_2$.
 	It follows that, 
up to relabeling, 
 	 $k' = s_1s_2$. 
 	Lemma~\ref{lemma_commuting} further implies that $s_1$ and $s_2$ commute with $z$. Consequently $x = s_1s_2 = k'$, which is a contradiction since $h$ and $h'$ are distinct.
	
	Suppose now that $k'$ has length one. Again by Lemma~\ref{lemma_commuting} and the fact that $\G$ is triangle-free, $k'$ cannot consist of a generator distinct from $s_1$ and $s_2$. Thus $k' = s_1$ or $k' = s_2$. By possibly relabeling, we may assume that $k' = s_1$. Now Lemma~\ref{lemma_commuting} implies that $s_1$ commutes with $z$ and $s_2$. Let $y = k'x = zs_2z^{-1}$. We replace $x$ with $y$ in $I_H$ to form the new set $I_H'$. Note that $I_H'$ is still a generating set for $H$ as $x = k'y$.
	 
	Let $b \neq s_1$ be any element of $I_H'$ which commutes with $x$.  Then since $x = zs_1s_2z^{-1}$ and $s_1$ commutes with $z$ and $s_2$,  Lemma~\ref{lemma_ax_commuting} implies that $b$ commutes with $y$.   It follows that $w'I_H'w'^{-1}$ is a generating set for $G$ which satisfies properties (1) and (2) above.  Finally, note that  the number of reflections in $I_H'$, and hence in $w'I_H'w'^{-1}$ is exactly $r+1$.

	By repeating this process enough times, we are guaranteed a finite generating set for $G$ consisting only of reflections. 
\end{proof}

It follows easily from Corollary~\ref{cor_refection_subgroup_quasiconvex} that subgroups as in Theorem~\ref{thm_racg_subgroup_gen_set} are quasiconvex:

\begin{corollary} \label{cor_RACG_subgroups_qc}
Given a $2$-dimensional RACG, every Coxeter subgroup 
whose defining graph does not have an isolated vertex 
 is quasiconvex. \hfill{\qed}
\end{corollary}

\begin{remark}\label{rmk_reflection_counterexample}
Note that the defining graph of a RACG has an isolated vertex if and only if the group splits as a free product with a $\mathbb{Z}_2$ factor.
	Some version of the non-splitting hypothesis is required in 
	Theorem~\ref{thm_racg_subgroup_gen_set} and
Corollary~\ref{cor_RACG_subgroups_qc}. 
For consider the graph $\G_2$ from Figure~\ref{fig_completion_ex2}.  The subgroup of $W_{\G_2}$ 
generated by $ab$ and $cd$ is isomorphic to the infinite dihedral group. In particular, it is a right-angled Coxeter subgroup of $W_{\G_2}$. However, it is straightforward to check that this subgroup cannot be generated by reflections and has an infinite completion (so is not quasiconvex by Theorem~\ref{thm_qc}).
\end{remark}

\section{Deciding when a RACG is a finite-index subgroup of a $2$-dimensional RACG} \label{sec_algorithm_for_2d_subgroups} 

In this section
we give an algorithm which, given a $2$-dimensional, one-ended RACG $W_{\G}$ and any RACG $W_{\G'}$, determines whether or not $W_{\Gamma}$ contains a finite-index subgroup isomorphic to $W_{\Gamma'}$ (Theorem~\ref{thm_algorithm_finite_index}).
In fact, the one-endedness hypothesis can be weakened.
When such a subgroup does exist, the output of the algorithm is a set of words in $W_{\G}$ which is a standard RACG generating set for a subgroup isomorphic to $W_{\G'}$.
Furthermore, the time-complexity of this algorithm only depends on the number of vertices of $\G$ and $\G'$.

Given a set of reflections $\mathcal{R}$ which generate a finite-index subgroup of $W_\G$, under the right hypotheses, Proposition~\ref{prop_conj_word_bound} below bounds the sizes of elements in $\mathcal{R}$ as a function of $|V(\G)|$ and $|\mathcal{R}|$. This is a key step in the proof of the main theorem of this section, as it allows us to bound the number of sets of reflections that need to be investigated by our algorithm.

To state the proposition, we require the following definition: we
say a graph $\Delta$ is \textit{almost star} if there exist  vertices $s, t \in \Delta$ (possibly not distinct) such that $V(\Delta) = \text{star}(s) \cup \{t\}$. 

\begin{proposition} \label{prop_conj_word_bound} 
	Let $W_\G$ be a RACG such that $\Gamma$ is triangle-free and not almost star. Let 
	\[\mathcal{R} = \{w_is_iw_i^{-1} ~ | ~ w_i \in W_\G \text{ and } s_i \in V(\G), 1 \le i \le N \}\]
	be a trimmed set of reflections which generates a finite-index subgroup $G < W_\G$. Then there exists a constant $M$, depending only on $|V(\Gamma)|$ and $|\mathcal{R}|$ such that $|w_i| \le M$ for all $w_is_iw_i^{-1} \in~\mathcal{R}$.
\end{proposition}

In order to prove this proposition, we establish some notation and prove some preliminary lemmas. The notation below will be fixed until the proof of Proposition~\ref{prop_conj_word_bound} is complete.

Let $M$ be the smallest integer such that if a reduced word $w$ in $W_\G$ is longer than $M$, then $w$ contains $2N+2$ occurrences of some letter of $V(\Gamma)$ (where $N = |\mathcal R|$). In particular, $M$ only depends on  $|V(\Gamma)|$ and on $|\mathcal{R}|$. This $M$ will be the same as the constant in the proposition. 

In order to establish a contradiction, we assume that $|w_1| > M$. By possibly relabeling, we assume that $|w_1| \ge |w_i|$ for all $1 < i \le N$. By the previous paragraph, 
 we may fix 
 some vertex $\bar{s}$ of $\Gamma$ which occurs at least $2N+2$ times in the word $w_1$. 

We now choose convenient expressions for the elements of~$\mathcal{R}$. Firstly, we assume $w_1$ is written in an expression where occurrences of $\bar{s}$ appear as far left  as possible. More formally, if $w_1 = s_1 \dots s_m$ and $s_i = \bar{s}$, then for all $j < i$ there is no expression for $w_1$ in $W_\G$ equal to the word $s_1 \dots s_{j-1}s_{j+1} \dots s_i s_j s_{i+1} \dots s_m$.

Given two words $w$ and $w'$ in $W_{\Gamma}$, let $\phi(w, w')$ denote the length of their largest common prefix.  For each $2 \le i \le N$, we choose an expression for $w_i$ so that $\phi(w_1, w_i)$ is maximal out of all such possible choices for $w_i$. Clearly, there is no loss of generality in making these assumptions on $w_1, \dots w_N$.

We will now use  the notation established in Section~\ref{subsec_completion_for_ref_subgroups} associated to a reflection subgroup of a RACG. As in that section, we have the based $\G$-labeled complex $(\Omega_0, B)$, the labeled tree $\mathcal{T} \subset \Omega_0$ and the associated folded based labeled tree $(\mathcal{FT}, B)$. Furthermore, $(\Omega_{\mathcal{FT}}, B)$ is a based finite CAT(0) cube complex which is a completion of $(\mathcal{FT}, B)$. By Theorem~\ref{thm_completion_of_reflection_subgroup_triangle_free}, there is a completion $(\Omega_G, B)$ of $(\Omega_0, B)$ (which is also a completion for $G$) whose associated completion sequence has length bounded by a function which depends only on $\sum_{r \in \mathcal{R}}|r|$ and $|V(\G)|$. 
We again denote by $\hat{f}$ the natural map 
\[\hat{f}: \mathcal{FT} \to \Omega_{\mathcal{FT}} \subset \Omega_G\]
 and recall that $\Omega_{\mathcal{FT}}$ is isometrically embedded in $\Omega_G$.

Let $\mathcal{V}$ be the set of vertices in $\mathcal{FT}$ that are the image of a vertex of $\mathcal{T} \subset \Omega_0$ which has a graph-loop attached to it. Observe that  $|\mathcal{V}| \le N$. Also note that at most $N$ vertices in $\mathcal{FT}$ have valence larger than $2$.

Let $\alpha$ be the path in $\mathcal{FT}$ based at $B$ with label $w_1$. As there are at least $2N+2$ occurrences of $\bar{s}$ in the word $w_1$, 
there must exist two edges of $\alpha$, say $e_1$ and $e_2$, each with label $\bar{s}$ such that every vertex between $e_1$ and $e_2$ has valence $2$ and is not in $\mathcal{V}$. Let $\gamma$ be the geodesic in $\mathcal{FT}$ between $e_1$ and $e_2$. By possibly passing to a subpath, we may assume that no edge in $\gamma$ has label $\bar{s}$. We also assume that $e_1$ is closer to $B$ than $e_2$.

We now sketch how a contradiction will be established.
As $G$ is a finite-index subgroup of $W_\G$, it will follow that $\Omega_G$ must be full-valence. We then focus on a specific vertex of $\Omega_G$ which is contained in $\hat{f}(\gamma)$. We show that the structure of edges and graph-loops incident to this vertex, together with  the assumption that $\Omega_G$ is full valence, must imply that $\Gamma$ is an almost star graph, contradicting the hypotheses of Proposition~\ref{prop_conj_word_bound}. In order to carry out this argument, we must first gain a solid understanding of the structure of the subcomplex $\hat{f}(\gamma)$. This is the purpose of the next four lemmas.

\begin{lemma} \label{lemma_injectivity}
	Let $v$ be a vertex of $\gamma$, and let $v'$ be any vertex in $\mathcal{FT}$. If $\hat{f}(v) = \hat{f}(v')$, then $v = v'$. 
\end{lemma}
\begin{proof}
	Let $\beta$ be a path in $\mathcal{FT}$ from the base vertex $B$ to $v$, and let $\beta'$ be a path in $\mathcal{FT}$ from $B$ to $v'$. The label $l$ of $\beta$ is a prefix of $w_1$. Similarly, the label $l'$ of $\beta'$ is a prefix of $w_i$, for some $1 \le i \le N$.
	Let $\hat{\beta} = \hat{f}(\beta)$ and $\hat{\beta}' = \hat{f}(\beta')$ be the images of these geodesics in $\Omega_{\mathcal{FT}}$.
	
	As $\Omega_{\mathcal{FT}}$ is a CAT(0) cube complex, the loop $\hat{\beta}' \cup \hat{\beta}^{-1}$ is homotopic, relative to basepoint, to $B$. 
Thus, by Lemma~\ref{lemma_omega_paths}(\ref{lemma_omega_paths3}), the label of $\hat{\beta}' \cup \hat{\beta}^{-1}$ is equal to the identity element in $W_{\Gamma}$. It follows that $l'$ and $l$ represent the same element of $W_{\Gamma}$. However, our choice of $w_i$ guarantees that $w_i$ and $w_1$ share the largest possible prefix. It follows that $l$ and $l'$ are the same word. As $\mathcal{FT}$ is folded, it follows that $\beta = \beta'$ and that $v = v'$.
\end{proof}

\begin{lemma} \label{lemma_paths_with_reduced_labels}
	Let $\beta$ be a path in $\Omega_{\mathcal{FT}}$ with label a reduced word in $W_\G$. Then $\beta$ is a geodesic. 
\end{lemma}
\begin{proof}
	Suppose $\beta$ is not geodesic. As $\Omega_{\mathcal{FT}}$ is a CAT(0) cube complex, it follows some hyperplane $K$ is dual to two distinct edges $k$ and $k'$ of $\beta$. Furthermore, we can choose $K$, $k$ and $k'$ so that every hyperplane dual to an edge of $\beta$ between $k$ and $k'$ intersects $K$. However, from this it readily follows that the label of $k$ (and of $k'$) commutes with the label of any edge of $\beta$ between $k$ and $k'$. This implies that the label of $\beta$ is not reduced, a contradiction.
\end{proof}

For the next two lemmas, let $\hat{e}_1 = \hat{f}(e_1)$ and $\hat{e}_2 = \hat{f}(e_2)$. Furthermore, let $H_1$ and $H_2$ be the hyperplanes in $\Omega_{\mathcal{FT}}$ that are dual respectively to $\hat{e}_1$ and $\hat{e}_2$.

\begin{lemma} \label{lemma_unique_hyp_intersection}
	The edge $\hat{e}_1$ is the only edge of $\hat{f}(\mathcal{FT})$ that is dual to $H_1$. Similarly, the edge $\hat{e}_2$ is the only edge of $\hat{f}(\mathcal{FT})$ that is dual to $H_2$.
\end{lemma}
\begin{proof}
	We prove the claim for $H_1$. The proof is analogous for $H_2$. Let $\hat{e}$ be an edge in $\hat{f}(\mathcal{FT}) \subset \Omega_G$ dual to $H_1$. We will show that $\hat{e} = \hat{e}_1$.
	
	As $\mathcal{FT}$ is a tree, if the edge $e_1$ is removed (but its endpoints are not removed), there are exactly two resulting components. We let $C_B$ denote the component
	that includes the vertex $B$, and let $\bar{C}_B$ be the component which does not.
	
	Let $e$ be an edge of $\mathcal{FT}$ such that $\hat{f}(e) = \hat{e}$. 	We first claim that $e$ cannot be in $\bar{C}_B$. For suppose otherwise. It follows that $e$ and $e_1$ are contained in a common path $\eta$ with reduced label $w_j$ for some $1 \le j \le N$. By Lemma~\ref{lemma_paths_with_reduced_labels} $\hat{f}(\eta)$ is a geodesic in $\Omega_{\mathcal{FT}}$. However, it now follows that the hyperplane $H_1$ is dual to two edges of a geodesic, contradicting the fact that a hyperplane in a CAT(0) cube complex is dual to at most one edge of a geodesic. 
	
	Thus, we may assume that either $e \in C_B$ or $e = e_1$. Let $\beta$ be a geodesic in $\mathcal{FT}$ from $e_1$ to $e$, which includes these two edges.
	Due to the tree structure of $\mathcal{FT}$, it follows that the label of $\beta$ is $k_1^{-1}k_2$, where $wk_1$ is a prefix of $w_1$ and $wk_2$ is a prefix of $w_j$ for some $1 \le j \le N$ and some reduced word $w$. 
	Note that $w, k_1$ and $k_2$ could each be the empty word. Let $\hat{\beta} = \hat{f}(\beta)$ be the corresponding path in $\Omega_{\mathcal{FT}}$. 
	
	Let $\hat{\zeta}$ be a geodesic along the carrier of $H_1$ from the endpoint to start point of $\hat{\beta}$. Let $z$ be the label of $\hat{\zeta}$. Let $D$ be a disk diagram in $\Omega_{\mathcal{FT}}$ with boundary $\hat{\beta} \cup \hat{\zeta}^{-1}$. The label of $\partial D$
	is $k_1^{-1}k_2z$. 
	Let $b_1$ and $b_2$ be the paths along $\partial D$ with labels respectively $k_1$ and $k_2$. Let $c$ the path along $\partial D$ labeled by $z$. Write the label of $b_1$ as $k_1 = s_1s_2 \dots s_m$, where $s_l \in V(\G)$ for $1 \le l \le m$. Note that $s_m = \bar{s}$. For $1 \le l \le m$, let $d_l$ be the edge of $b_1$ with label $s_l$.

	We claim that no dual curve intersects both $b_1$ and $c$. For suppose there is such a dual curve $P$. Further suppose that $P$ intersects the edge $d_r$ such that $r$ is maximal out of such possible choices. 
	Note that $r \neq m$ as $z$ only contains letters in $\text{link}(\bar{s})$, being   
	the label of a geodesic 
	in the carrier 
	of a hyperplane of type $\bar{s}$.
	It follows that every dual curve intersecting $d_l$, for $l > r$, intersects $P$. 
	Let $p \in \Gamma$ be the type of $P$. As $P$ intersects $c$, and the label of $c$ is in $\text{link}(\bar{s})$, it follows that $p \in \text{link}(\bar{s}) \subset V(\G)$.
	Additionally, $p$ commutes with $s_l$ for every $r < l \le m$. However, this implies that $s_1 \dots s_{r-1}s_{r+1} \dots s_m s_r$ is an expression for $k_1$. As $k_1$ is a subword of $w_1$, this contradicts our choice of $w_1$ having occurences of $\bar{s}$ appear as ``left-most'' as possible. Thus, every dual curve intersecting $b_1$ must intersect $b_2$.
	
	By Lemma~\ref{lemma_disk_diagram_subwords}, $k_2$ is an expression in $W_\G$ for $k_1 k$, where $k$ is possibly empty. However, it follows from our choice of expression for $w_j$ (we chose expressions for words in $\mathcal{R}$ to have maximal common prefix with $w_1$) that $k_2$ and $k_1k$ are actually equal as words. Consequently $wk_1$ is a prefix of both $w_1$ and $w_j$. Hence, we conclude that $e = e_1$ and so $\hat{e} = \hat{e}_1$.
\end{proof}

\begin{lemma} \label{lemma_graph_loops_in_gamma}
	Let $Y$ be the subcomplex of $\Omega_{\mathcal{FT}}$ bounded by $H_1$ and $H_2$. Let $\hat{v}$ be a vertex in $Y$. Then the label of any graph-loop in $\Omega_G$ incident to $\hat{v}$ (where we think of $\hat{v} \in \Omega_{\mathcal{FT}} \subset \Omega_G$) is in $\text{link}(\bar{s}) \subset V(\Gamma)$.
\end{lemma}
\begin{proof}
	Let $u$ be any vertex of $\Omega_0$ that is incident to a graph-loop. By construction, the image of $u$ in $\mathcal{FT}$ is not contained in $\gamma$. By Lemma~\ref{lemma_unique_hyp_intersection}
	 and the fact that every vertex of $\gamma$ has valence 2 by construction,
	 $\hat{f}(\mathcal{FT}) \cap Y = \gamma$. Thus
either $H_1$ or $H_2$ separates $\hat{v}$ from $\hat{f}(u)$. 
	
	Suppose the label of the graph-loop attached to $\hat{v}$ is $t$. 
	By Lemma~\ref{lemma_adding_graph_loops} and 
	Remark~\ref{rmk_special_properties_of_completion}, there is a path $\eta$ in $\Omega_{\mathcal{FT}}$ from $\hat{v}$ to a vertex $\hat{u}'$ such that $\hat{u}' = \hat{f}(u')$, where $u' \in \Omega_0$ is a vertex that is incident to a graph-loop labeled by $t$. Furthermore, the label of $\eta$ is a word in  $\text{link}(t)$. It follows that $\eta$ must intersect either $H_1$ or $H_2$. Hence, the label of $\eta$ contains the generator $\bar{s}$. Thus, $\bar{s} \in \text{link}(t)$ and so $t \in \text{link}(\bar s)$. 
\end{proof}

We are now ready to prove the proposition.

\begin{proof}[{Proof of Proposition~\ref{prop_conj_word_bound}}]
	Let $e$ be the edge of $\gamma$ adjacent to $e_2$. Let $v$ be the vertex $e \cap e_2$. Let $t$ be the label of $e$. By our choice of $w_1$ (having $\bar{s}$ occurrences appear ``left-most''), $t$ and $\bar{s}$ are not adjacent vertices of $\Gamma$. Set $\hat{e} = \hat{f}(e)$, $\hat{e}_2 = \hat{f}(e_2)$, $\hat{v} = \hat{f}(v)$ and $\hat{\gamma} = \hat{f}(\gamma)$. 
	
	 As $\G$ is not almost star, it readily follows that $\text{star}(s) \subsetneq V(\G)$ for any $s \in V(\G)$. In particular, 
$\Omega_G$ is resolved.
	 As $G$ is a finite-index subgroup,
	 $\Omega_G$ is full valence by 
 Theorem~\ref{thm_omega_finite_index}.
	 
	 Again, as $\G$ is not almost star, there exists a vertex $a \in \G$ such that $a \neq t$ and $a \notin \text{star}(\bar{s})$. As $\Omega_G$ is full valence, there must exist an edge $\hat{d}$ adjacent to $\hat{v}$ labeled by $a$. By Lemma~\ref{lemma_graph_loops_in_gamma}, $\hat{d}$ is not a graph-loop. Let $\hat{u}'$ be the vertex of $\hat{d}$ which is not equal to $\hat{v}$. As $\Omega_G$ is full valence, there must exist an edge $\hat{d}'$ adjacent to $\hat{u}'$ with label $\bar{s}$ which is not a graph-loop by Lemma~\ref{lemma_graph_loops_in_gamma}. Let $H$ be the hyperplane dual to $\hat{d}'$. 
	 
	 First note that $\hat{d}'$ cannot be dual to $H_2$, for otherwise it would follow from the convexity of $N(H_2)$ that $\hat{d} \subset N(H_2)$, contradicting the fact that $a$ is not in $\text{star}(\bar{s})$. Furthermore, $\hat{d}'$ cannot be dual to $H_1$ either. For otherwise, as $\hat{\gamma}$ is geodesic (by Lemma~\ref{lemma_paths_with_reduced_labels}), it follows that the hyperplane dual to $\hat{e}$ must intersect $H_1$, contradicting the fact that $t$ is not in $\text{link}(\bar{s})$. Thus, $H \neq H_1$ and $H \neq H_2$.
	 
	 By Proposition~\ref{prop_omega_hyps}, $H$ must intersect $\hat{f}(\mathcal{FT}) \subset \Omega_G$.
	 As $\hat{\gamma}$ does not have any edges labeled by $\bar{s}$, it follows that $H$ cannot intersect $\hat{\gamma}$. Thus, by Lemma~\ref{lemma_unique_hyp_intersection}, $H$ must intersect either $H_1$ or $H_2$. However, this is a contradiction as $H$, $H_1$ and $H_2$ are all of type $\bar{s}$.
	\end{proof}

Before proving Theorem~\ref{thm_algorithm_finite_index}, we address a special case.  The next lemma describes finite-index subgroups of $W_\G$ for the case where $\G$ is a triangle-free join graph.

\begin{lemma} \label{lemma_join_case}
	Suppose $\Gamma$ is a triangle-free graph which splits as a join $\Gamma = A \star B$. Let $\mathcal{R}$ be a finite set of reduced reflection words in $W_\Gamma$ which generates the subgroup $G < W_\G$. Then $G$ is a finite-index subgroup of $W_{\Gamma}$ if and only if 
	$\mathcal{R} = \mathcal{R}_A \cup \mathcal{R}_B$ such that
	\begin{enumerate}
		\item $\mathcal{R}_A$ (resp.~$\mathcal{R}_B$) consists only of words in $W_A$ (resp.~$W_B$).
		\item $\mathcal{R}_A$ (resp.~$\mathcal{R}_B$) generates a finite-index subgroup of $W_A$ (resp.~$W_B$).
	\end{enumerate}
\end{lemma}
\begin{proof}
	The ``if'' direction is immediate. For the other direction, suppose that $G$ has finite index in $W_\G$. 
	Let 
	\[\mathcal{R}_A = \{wsw^{-1} \in \mathcal{R} ~|~ s \in V(A)
	 \} \] 	
	and let $\mathcal{R}_B = \mathcal{R} \setminus \mathcal{R}_A$. Note that if $wsw^{-1} \in \mathcal{R}_A$, then $w$ does not have a letter in $V(B)$. For if it did, $wsw^{-1}$ would not be reduced. Similarly, every reflection in $\mathcal{R}_B$ does not contain a letter of $V(A)$. This shows (1).
	
	Thus, the subgroup $G_A$ generated by $\mathcal{R}_A$ is a subgroup of $W_{A}$, and the subgroup $G_B$ generated by $\mathcal{R}_B$ is a subgroup of $W_B$. As $G$ has finite index in $W_{\G}$, it must be that $G_A$ and $G_B$ are finite-index subgroups respectively of $W_A$ and $W_B$.
\end{proof}

We also need the proposition below, which follows from known results:
\begin{proposition} \label{cor_dyer}
	Let $W$ be a RACG and let 
	$G < W$ be a subgroup generated by a set $\mathcal{R}$ of reflections. 
	Then $G$ is a RACG. Furthermore, if $\mathcal{R}$ is trimmed, then it is a standard Coxeter generating set.
\end{proposition}

\begin{proof}
Deodhar and Dyer independently proved that reflection subgroups of Coxeter groups are Coxeter groups~\cite{Dyer} \cite{Deodhar}.  This, combined with Proposition \ref{prop_coxeter_subgroups_are_racgs} implies that $G$ is a RACG.
	The second claim follows from the main theorem of \cite{Dyer} and \cite[Proposition~3.5]{Dyer}. We also refer the reader to \cite[page 69]{Dyer} for an algorithm to determine a standard generating set for a reflection subgroup of a Coxeter group. 
	\end{proof}

We are now ready to prove the main theorem of the section.

\begin{theorem} \label{thm_algorithm_finite_index}
	There is an algorithm which, given a $2$-dimensional RACG  $W_{\Gamma}$ and a RACG $W_{\G'}$ such that $\G'$ does not have an isolated vertex, determines whether or not $W_{\Gamma'}$ is isomorphic to a finite-index subgroup of $W_{\Gamma}$. The algorithm takes as input the graphs $\G$ and $\G'$, and the time-complexity of this algorithm only depends on the number of vertices of $\G$ and of $\G'$. Furthermore, if $W_{\G'}$ is isomorphic to a finite-index subgroup of $W_{\G}$, then the algorithm outputs an explicit set of words in $W_{\G}$ which generate this subgroup.
\end{theorem}
\begin{proof} 
	First observe that if there exists some $G<W_{\Gamma}$  isomorphic to $W_{\Gamma'}$, then by Theorem~\ref{thm_racg_subgroup_gen_set} and Lemma~\ref{lemma_reflection_generators}, $G$ is generated by a trimmed set of reflections $\mathcal{R}$ and $|\mathcal{R}| = |V(\Gamma')|$.
	
	We prove the theorem by analyzing a few different cases depending on the structure of $\Gamma$.	

	\smallskip\noindent{\it (i) $\Gamma$ is not almost star.}
	Let $I$ be a trimmed set of reduced reflections in $W_{\Gamma}$. We say $I$ is $M$-admissible if $|I| = |V(\Gamma')|$ and $|w| \le M$ for every reflection $wsw^{-1} \in I$. Let $\mathcal{I}_M$ be the collection of all $M$-admissible trimmed sets of reflections. Note that there is a bound on $|\mathcal{I}_M|$ depending only on $M$ and $|V(\Gamma')|$.
	
	Suppose $G$ is a finite-index subgroup of $W_\G$ which is isomorphic to $W_{\G'}$, and let $\mathcal{R}$ be a trimmed generating set for $G$ as described above. As $\G$ is not almost star, Proposition~\ref{prop_conj_word_bound} guarantees that $\mathcal{R} \in \mathcal{I}_M$, where $M$ is as in Proposition~\ref{prop_conj_word_bound}, and depends only on $|V(\Gamma)|$ and $|V(\Gamma')| = |\mathcal{R}|$. It follows that there exists a finite-index subgroup of $W_{\Gamma}$ isomorphic to $W_{\G'}$ if and only if some $I \in \mathcal{I}_M$ generates a finite-index subgroup that is isomorphic to $W_{\Gamma'}$. 
	
	Thus, to prove the theorem we only need to show that there is an algorithm to decide whether a given $I \in \mathcal{I}_M$ generates a finite-index subgroup $G$ isomorphic to $W_{\G'}$. By Theorem~\ref{thm_completion_of_reflection_subgroup_triangle_free} and 
 Theorem \ref{thm_omega_finite_index} 
	there is an algorithm 
	to decide whether or not such a $G$ is a finite-index subgroup, and the time-complexity of this algorithm only depends on $|V(\Gamma')|$ and $|V(\Gamma)|$. By Proposition~\ref{cor_dyer}, $G$ is a RACG and $I$ is a standard Coxeter generating set. As any RACG is defined by a unique graph \cite{Radcliffe}, it is straightforward to check whether $I$ generates a RACG isomorphic to $W_{\Gamma'}$. The theorem then follows in this case. 
	
	\smallskip\noindent{\it (ii) $|V(\Gamma)| \le 2$. }
 In this case $W_\Gamma$ is isomorphic to  $\mathbb{Z}_2$, $\mathbb{Z}_2 \times \mathbb{Z}_2$ or $\mathbb{Z}_2 * \mathbb{Z}_2$. 
	 Each of these groups contains finitely many
	finite-index RACGs,
	up to isomorphism, and one can easily list them.

	\smallskip\noindent{\it (iii) $\Gamma = A \star B$ where $|V(A)|, |V(B)| \ge 2$. }
	Suppose $G$ is a subgroup of $W_\G$ which is isomorphic to $W_{\G'}$. Let $\mathcal{R}$ be a trimmed generating  set of reflections for $G$ as  above.
	
	If $G$ is a finite-index subgroup, then Lemma~\ref{lemma_join_case} tells us that $\mathcal{R} = \mathcal{R}_A \cup \mathcal{R}_B$ where every reflection in $\mathcal{R}_A$ only contains generators in $A$ and every reflection in $\mathcal{R}_B$ only contains generators in $B$.  Furthermore, $\mathcal{R}_A$ generates a finite-index subgroup of $W_A$ and $\mathcal{R}_B$ generates a finite-index subgroup of $W_B$. As $|V(A)|, |V(B)| \ge 2$, both $\mathcal{R}_A$ and $\mathcal{R}_B$ are non-empty.  
	
	Let $\Delta = \Delta_A \star \Delta_B$ be the triangle-free join graph 
such that vertices of $\Delta_A$ (resp.~$\Delta_B$) correspond to elements of $\mathcal{R}_A$ 
(resp.~$\mathcal{R}_B$).
It readily follows that $W_{\Gamma'}$ is isomorphic to $W_{\Delta}$. 
	
	Thus, we can assume that  $\Gamma' = A' \star B'$, a join graph. To prove the claim, again by Lemma~\ref{lemma_join_case}, it is enough to check whether $W_{A'}$ is isomorphic to a finite-index subgroup of $W_A$ and  $W_{B'}$ is isomorphic to a 
	finite-index subgroup of $W_B$. However, as $\Gamma$ is triangle-free, $A$ does not contain any edges. It follows that either $A$ is not almost star or $A$ consists of at most two isolated vertices. The same holds for $B$. Thus, we are done by cases (i) and (ii). 

		\smallskip\noindent{\it (iv) $\Gamma$ is not as in (i), (ii) or (iii).}
	As $\Gamma$ is not as in (i), we may assume that $\Gamma$ is almost star. 
	Suppose first that $V(\G) \subsetneq \text{star}(v)$ for all $v \in V(\G)$. Let $s$ and $t$ be vertices of $\Gamma$ such that $V(\Gamma) =  \text{star}(s) \cup \{t \}$. Note that $s$ and $t$ must be distinct. 
	As $\Gamma$ is triangle-free, is not a join
	as in (iii), is not the star of a vertex, and contains more than two vertices, 
	 there must be a vertex $u$ of $\Gamma$ that is adjacent to $s$ and is not adjacent to any other vertex of $\Gamma$.

	Let $G$ be a subgroup of $W_{\Gamma}$ generated by reflections.	Let 
	$\Delta = D(\Gamma, u)$ and $\phi = \phi_u: W_{\Gamma} \to \mathbb{Z}_2$ be as in 
	Definition~\ref{def_doubling}, and  let $K =  \ker(\phi)$. Note that $\Delta$ is not almost star and that $K$ is generated by reflections.

	Let $i_G: G \to W_{\Gamma}$ be the inclusion map. Let $K'$ be the kernel of the map $\phi' = \phi \circ i_G$. We get the diagram below where all maps labeled by $i$ are the obvious inclusion homomorphisms.

\smallskip
	\begin{tikzpicture}
		\matrix (m) [matrix of math nodes,row sep=3em,column sep=8em,minimum width=2em]
		{
		K' = \ker(\phi')  & G & \\
		K = \ker(\phi) \cong W_\Delta & W_\Gamma & \mathbb{Z}_2 \\};
		\path[-stealth]
		
		(m-1-1) edge node [right] {$i$} (m-2-1)
		
		(m-1-1) edge node [below] {$i$} (m-1-2)
		
		(m-2-1) edge node [below] {$i$} (m-2-2)
		
		(m-1-2) edge node [right] {$i_G$} (m-2-2)
		
		(m-2-2) edge node [below] {$\phi$} (m-2-3)
		
		(m-1-2) edge node [right] {$\phi' = \phi \circ i_G$} (m-2-3);
	\end{tikzpicture}
\smallskip

	Recall that given a triple of groups $G_3<G_2<G_1$, 
	their indices satisfy the formula
	 $[G_1:G_3]=[G_1:G_2] [G_2:G_3]$, where infinite values are interpreted appropriately. If we apply this formula to the groups in the  diagram and note that $[G: K'] =  [W_\G :K] = 2$, 
	 we get:
	\[  2[W_\G : G] = [W_\G : G][G: K'] = [W_\G : K'] = [W_\G :K] [K : K'] = 2[K:K']\]
	Thus, $G$ is a finite-index subgroup of $W_\G$ if and only if $K'$ is a finite-index subgroup of $K$. As $\Delta$ is not almost star, it follows by (i) that there is an algorithm to check whether $K'$ is a finite-index subgroup of $K$. The theorem now follows. 	
	
	On the other hand, if  $V(\G) = \text{star}(s)$ for some  $s \in V(\G)$, then we apply the same argument as before but instead take the homomorphism $\phi = \phi_s$. In this case, $\Delta = D(\Gamma, s)$ must be a non-empty graph with no edges. Thus, by either (i) or (ii), there is an algorithm to check whether $K'$ is a finite-index subgroup of $K$, where $K'$ and $K$ are  as before.
\end{proof}

\section{Other algorithmic properties of quasiconvex subgroups}\label{sec_other_algorithmic_properties}
This section is dedicated to the proof of Theorem~\ref{introtheorem_algorithms} of the introduction. 
Throughout this section, we let $G$ be a quasiconvex subgroup of $W_\G$ given by a finite generating set, and we let $\Omega$ be the corresponding completion. 
By Theorem~\ref{thm_qc}, $\Omega$ is finite and by Proposition~\ref{prop_finite_completion} it can be computed in finite time. 

Note that (1) and (2) of Theorem~\ref{introtheorem_algorithms}
 immediately follow by Proposition~\ref{prop_torsion_free}, 
and Theorem~\ref{thm_omega_finite_index} respectively. 
 Before proving (3),
 we show that powers of an element of a RACG can be represented by words of a special form.

\begin{lemma} \label{lemma_powers}
	Let $w$ be a reduced word in the RACG $W_\G$. Then there exist reduced words $x, h$ and $k$, such that $xhkx^{-1}$ is a reduced expression for $w$ and $xh^nk^{(n\mod 2)} x^{-1}$ is a reduced expression for $w^n$ for all integers $n > 0$.
\end{lemma}
\begin{proof}
	Write $w = xyx^{-1}$ where $x$ and $y$ are reduced words and $|x|$ is maximal out of all such possible expressions. Let $K = \{k_1, \dots, k_n\}$ be the set of vertices in $\Gamma$ that appear as letters in the word $y$ and which commute with every other letter of $y$. As $w$ is reduced, each element of $K$ appears as a letter of $w$ exactly once. Define the word $k = k_1 \dots k_n$. By our choice of $k$, it follows that there exists a reduced word $h$, such that 
	$hk$ is a reduced expression for $y$.
Note that $h$ has the property that any generator which appears as the last letter of some reduced expression for $h$, cannot also appear as the first letter in some reduced expression for $h$. This follows since otherwise either $x$ is not maximal or such a generator should have been in $K$. The word $xhkx^{-1}$ will be the desired expression for $w$.
	
	The word $xh^nk^{(n\mod 2)} x^{-1}$ is clearly an expression for the word $w^n$. Furthermore, the word $xh^nk^{(n\mod 2)} x^{-1}$ must be reduced. For otherwise, it follows by the deletion property	that either $h$ is not reduced or  some generator appears as the first letter in some reduced expression for $h$ and as the last letter in some reduced expression for $h$, which is not possible.
\end{proof}

\begin{proof}[{Proof of 
Theorem~\ref{introtheorem_algorithms}
(3):}]
	Without loss of generality, we may assume that $w$ is reduced. 
	Let $N$ be the number of vertices of $\Omega$. We claim that $g^m \in G$ for some positive integer $m$ if and only if $w^{2l}$ represents an element of $G$ for some $l \le N$. The theorem clearly follows from this claim and the fact that the membership problem is solvable for $G$.
	
	One direction of the claim is clear. On the other hand, suppose $g^m \in G$ for some positive integer $m$. By possibly taking a power, we may assume that $m$ is even. By Lemma~\ref{lemma_powers}, there is a reduced expression for $w^m$ of the form 
	$z = xh^m x^{-1}$.
	Let $\beta$ be a loop in $\Omega$ based at $B$ with label $z$. For $1 \le i \le m$, let $\alpha_i$ be the first subpath of $\beta$ with label $h^i$,
	 and let $v_i$ be the endpoint of $\alpha_i$. 
	 As $\Omega$ is folded, the two subpaths of $\beta$ labeled $x$ are identified.  It follows that $\alpha_m$ is a loop based at $v_m$.
	 As there are at most $N$ vertices of $\Omega$, it follows that the set $\{v_1, \dots, v_m\}$ contains at most $N$ distinct vertices. There must then exist some loop $\alpha$ based at $v_m$ with label $h^l$ for some $l \le N$. 
	 Thus, if we replace $\alpha_m$ with $\alpha\alpha$ in $\beta$, we conclude that there is a loop in $\Omega$ based at $B$ with label $x h^{2l} x^{-1}$. By the definition of a completion, the word $x h^{2l}x^{-1}$, which is an expression for $w^{2l}$, represents an element of $G$. 
\end{proof}

Although Theorem~\ref{thm_normal} characterizes normality,
the core of a completion may be difficult to algorithmically compute in general. Thus, we now give a different characterization of normality, for finitely generated subgroups, which is better suited to the algorithmic approach. 

\begin{proposition}\label{prop_normal_fg}
Let $G< W_\G$ be generated by a finite set of reduced words $S_G$, and let $(\Omega, B)$ be a standard completion for $G$ with respect to~$S_G$. 
Consider $\Delta \subset V(\G)$ defined by:
\[\Delta = \{ s \in V(\G) \mid s \text{ commutes with every element of } G \} \]
Then $G$ is normal if and only if the following hold:
\begin{enumerate}
\item[(N1)] Given any $s \in V(\G) \setminus \Delta$, there is an edge in $\Omega$ incident  to $B$ with label $s$. 
\item[(N2')] For every generator $w \in S_G$ of G, and for every vertex $v$ of $\Omega$, there exists a loop based at $v$ with label $w$.
\end{enumerate}
\end{proposition}

\begin{proof}
If $G$ is normal, then Theorem~\ref{thm_normal} implies N1 and N2'.  On the other hand, suppose N1 and N2' hold.  As in the proof of Theorem~\ref{thm_normal}, 
to show that $G$ is normal it is enough to show $sGs \subseteq G$ for $s \in V(\G)\setminus \Delta$.  
Let $v$ be the vertex which is adjacent to $B$ via an edge labeled $s$, which exists due to N1.  
Then by Lemma~\ref{lem_basepoint_change},
 the subgroup of $W_\G$ associated to $(\Omega, v)$  is $sGs$, and 
by N2', $G$ is contained in this subgroup.  Conjugating by $s$, it follows that $sGs \subseteq G$. 
\end{proof}

\begin{proof}[{Proof of 
Theorem~\ref{introtheorem_algorithms}
 (4):}]
Let $S_G$ be a finite set of reduced words in $W_\G$ which generate $G$.
In this case, the  set $\Delta$ from Proposition~\ref{prop_normal_fg} is equal to the following set:
 $$
 \{ s \in V(\G) \mid \forall w \in S_G, \;s \text{ commutes with every letter in the support of } w\} 
 $$
Thus $\Delta$ can be computed in finite time.  
It follows that conditions (N1) and (N2') from Proposition~\ref{prop_normal_fg} can be checked in finite time as well. 
\end{proof}

\bibliographystyle{amsalpha}
\bibliography{bibliography}
\end{document}